\documentclass[10pt]{amsart}
\usepackage{latexsym, amsmath, amsfonts, amssymb, mathrsfs, fancyhdr, tikz, tikz-cd}
\usepackage{verbatim}
\usetikzlibrary{matrix,arrows,decorations.pathmorphing,calc}

\def\ds{\displaystyle}
\def\nin{\not \in}
\def\n{\noindent}
\def\R{\mathbb{R}}
\def\Rpq{\mathbb{R}^{p,q}}

\def\v{\varphi}
\def\e{\varepsilon}
\def\s{\sigma}

\def\O{\Omega}
\def\d{\delta}

\def\T{\mathcal{T}}
\def\P{\mathcal{P}}
\def\l{\lambda}

\def\X{\mathcal{X}}
\def\Y{\mathcal{Y}}
\def\V{\mathcal{V}}
\def\S{\mathcal{S}}
\def\E{\mathbb{E}}
\def\calZ{\mathcal{Z}}
\def\a{\alpha}
\def\b{\beta}

\def\g{\gamma}

\def\N{\mathcal{N}}
\def\la{\langle}
\def\ra{\rangle}
\def\dx{d_{\X}}
\def\dy{d_{\Y}}

\def\calN{\mathcal{N}}
\def\scrU{\mathscr{U}}
\def\scrV{\mathscr{V}}
\def\scrW{\mathscr{W}}

\def\G{\Gamma}
\def\dn{d_{\mathcal{N}}}
\def\D{\mathscr{D}}

\newtheorem{theorem}{Theorem}
\newtheorem{lemma}[theorem]{Lemma}
\newtheorem{corollary}[theorem]{Corollary}
\newtheorem{prop}[theorem]{Proposition}

\newtheorem*{thm}{Main Theorem}
\newtheorem*{them}{Theorem}

\theoremstyle{definition}

\newtheorem{remark}{Remark}
\newtheorem*{notation}{Notation}

\theoremstyle{remark}

\numberwithin{equation}{section}

\begin{document}

\title{On the isometric embedding problem for length metric spaces}

\author{Barry Minemyer}
\address{Department of Mathematics, The Ohio State University, Columbus, Ohio 43210}
%\curraddr{Department of Mathematical Sciences,
%Binghamton University, Binghamton, New York 13902}
\email{minemyer.1@osu.edu}
%\thanks{I was supported by the Mathematical Sciences Department at Binghamton University}

%\author{Author Two}
%\address{Mathematical Research Section, School of Mathematical Sciences,
%Australian National University, Canberra ACT 2601, Australia}

%    General info
%\subjclass[2010]{Primary 51F99, 52B11, 52B70, 57Q35, 57Q65; Secondary 52A38, 53B21, 53A35, 53B30, 53C50}


\date{\today.}

%\dedicatory{This paper is dedicated to our authors.}


\begin{abstract}
We prove that every proper $n$-dimensional length metric space admits an ``approximate isometric embedding" into Lorentzian space $\R^{3n+6,1}$.  
By an ``approximate isometric embedding" we mean an embedding which preserves the energy functional on a prescribed set of geodesics connecting a dense set of points.  
\end{abstract}

\maketitle

%\section*{This is an unnumbered first-level section head}
%This is an example of an unnumbered first-level heading.

%\specialsection*{Introduction}
%This is an example of a special section head.

\section{Introduction}\label{section 1}

In the 1950's John Nash famously proved that every Riemannian manifold admits an isometric embedding into Euclidean space (\cite{Nash1}, \cite{Nash2}).  
Since then, many mathematicians have attempted to improve on Nash's work in many different ways.  
Maybe the most obvious direction to improve on Nash's results is to attempt to lower the dimensionality of the target Euclidean Space.  
This has been done in several settings, most notably by Kuiper \cite{Kuiper}, Greene \cite{Greene}, Gromov \cite{GR}, \cite{Gromov PDR}, and G\"{u}nther \cite{Gunther}.  

The next natural avenue for generalization would be to look at a larger class of source spaces.  
That is, to look at spaces which are more general than that of a Riemannian manifold.  
The first results in this direction, proved independently by Greene \cite{Greene} and Gromov and Rokhlin \cite{GR}, dealt with isometric embeddings of differentiable manifolds endowed with indefinite metric tensors.
Of course, in this setting one needs to enlarge the class of target spaces to that of Minkowski space $\R^{p,q}$ to allow for metrics which are not positive definite (or possibly even degenerate).
The study of generalizing Nash's results to the class of Finsler manifolds (and with Banach spaces as the target) was taken up by Burago and Ivanov in \cite{BI}.

Another class of ``generalizations" that have developed have been in the setting of ``(indefinite) metric polyhedra".  
Define an {\it indefinite metric polyhedron} to be a triple $(\X, \T, G)$ where $\X$ is a topological space, $\T$ is a locally-finite (simplicial) triangulation of $\X$, and for all $\s \in \T$ a $k$-dimensional simplex, $G(\s)$ is a quadratic form on $\R^k$ which restricts to $G(\s')$ for all $\s' < \s$.  
Call $\X := (\X, \T, G)$ a {\it Euclidean polyhedron} if $G(\s)$ is positive-definite for all $\s \in \T$, so that Euclidean polyhedra are in some sense a combinatorial analogue to Riemannian manifolds.
There are various results concerning continuous and piecewise-linear isometries and isometric embeddings of Euclidean polyhedra into Euclidean space, as well as piecewise-linear and simplicial isometric embeddings of indefinite metric polyhedra into Minkowski space.  
These results show a surprising similarity to the theorems about differentiable manifolds mentioned above, and are due to Zalgaller \cite{Zalgaller}, Burago and Zalgaller \cite{BZ}, Krat \cite{Krat}, Akopyan \cite{Akopyan}, the author \cite{Minemyer1}, \cite{Minemyer2}, and Galashin and Zolotov \cite{GZ}.  
Of course though, these results are not really generalizations of the Nash isometric embedding theorems since clearly not all Riemannian metrics will be piecewise linear on some simplicial triangulation of the manifold.  

A {\it length metric space} is a metric space where the distance between any two points is equal to the infimum of the lengths of paths between those two points.
Every Riemannian manifold has a natural length metric associated to it via calculating the lengths of paths with the intrinsic Riemannian metric.
But there are many families of length metric spaces that are not (necessarily) Riemannian manifolds, some examples of which are Finsler manifolds, CAT(0) spaces, Alexandrov spaces with curvature bounded below, and the aforementioned Euclidean polyhedra.
The original motivation of this research was to generalize Nash's isometric embedding results to the much larger class of (proper and finite dimensional) length metric spaces. %which {\it are} generalizations of (complete) Riemannian manifolds.  
Unfortunately, in \cite{Le Donne} Le Donne shows that there is not much hope for such a result with the following Theorem.

\vskip 10pt

\begin{them}[Le Donne \cite{Le Donne}]\label{theorem:  Le Donne}
Any Finsler manifold which is not Riemannian does not admit a path isometric embedding into Euclidean space.  
\end{them}

\vskip 2pt

This result was also stated by Petrunin in \cite{Petrunin}, and possibly known by Burago and Ivanov in \cite{BI}.
 
The argument to prove Le Donne's Theorem is pretty straightforward and goes as follows.  
Any path isometry $f$ into Euclidean space must be $1$-Lipschitz, and thus locally Lipschitz.  
Then by Rademacher's Theorem (Theorem \ref{Rademacher} in Section 8), $f$ must be differentiable almost everywhere.  
Thus the Finsler norm is induced by an inner product almost everywhere, and hence everywhere by the continuity of the Finsler structure.  
Therefore, the original Finsler manifold was in fact Riemannian.

In light of this result, the attention of this study turned to considering isometric embeddings of length metric spaces into {\it Minkowski} space $\R^{p,q}$.  
Le Donne's result above does not quite rule out such a possibility, but it does give a restrictive necessary condition stated as the following Proposition.

\vskip 10pt

\begin{prop}\label{Proposition}
Let $\X$ be a Finsler manifold which is not Riemannian, and suppose that a (path) isometric embedding $f: \X \to \R^{p,q}$ exists.  
Let $\pi^+$ and $\pi^-$ denote the natural projections from $\R^{p,q}$ onto $\R^{p,0}$ and $\R^{0,q}$, respectively.  
Then it must be the case that $p > 0$, $q > 0$, and both maps $\pi^+ \circ f$ and $\pi^- \circ f$ are not locally Lipschitz.
\end{prop}

\vskip 2pt

The proof of this Proposition is straightforward and can be found with the preliminaries in Section 8.
What this Proposition says is that any function which is a candidate to be a path isometric embedding cannot be locally Lipschitz in either of its positive or negative components.
Then to be a path isometry the energy of the image of any path with respect to the positive and negative components would need to ``cancel'' in a manner which preserves the intrinsic metric.
At first glance it may seem that such a map surely cannot exist.  
But the Main Theorem of this paper proves that this reasoning is incorrect, at least in a discrete sense.

\vskip 10pt

\begin{thm}
Let $\X$ be a proper $n$-dimensional length metric space, and let $\mathscr{D} \subseteq \X$ be any countable dense subset.  
Then there exists a collection of geodesics $\G$ associated to $\D$ and an embedding $f: \X \to \R^{3n+6,1}$ such that $E(\g) = E(f \circ \g)$ for all $\g \in \G$.  
Moreover, if desired, the map $f$ can be constructed so that its projection onto $\R^{3n+6,0}$ is not locally Lipschitz.
\end{thm}

\vskip 2pt

%\n \textbf{Remarks.}
Some remarks about the Main Theorem are as follows:

\vskip 10pt

	%\begin{enumerate}
	%\item 
	
\n (1)  In this paper the term ``dimension" will always mean ``covering dimension".

\vskip 10pt

\n (2)  Given a subset $\D \subseteq \X$, a {\it collection of geodesics associated to $\D$} is a set $\G$ of the form
	\[
	\G := \left\{ \g: [a,b] \to \X \, | \, \g \text{ is a geodesic} \right\}
	\]
which satisfies that for all $x, y \in \D$, there exists a geodesic $\g \in \G$ such that $\g(a) = x$ and $\g(b) = y$ (or vice versa).  
Let us note that we allow the domain $[a,b]$ to differ for the different geodesics in $\G$.  

\vskip 10pt

\n (3)  It is well known that every proper length metric space is a geodesic space (see \cite{BBI} and/or \cite{BH}), meaning that between any two points $x, y \in \X$ there exists a path $\a: [0,1] \to \X$ such that $\a(0) = x$, $\a(1) = y$, and $|s-t| \cdot d(x,y) = d(\a(s), \a(t))$ for all $s, t \in [0,1]$.  
So we can upgrade the space $\X$ in the Main Theorem to being a geodesic metric space.

\vskip 10pt

	%\item 
\n (4)  Recall that for a Riemannian manifold $(M, g)$ and a smooth path $\g: [a,b] \to M$, the {\it length} and {\it energy} functionals are defined as
		\begin{equation*}
		\ell(\g) = \int_a^b \sqrt{g(\g'(t), \g'(t))} \, dt 	\qquad	\text{and}	\qquad	E(\g) = \int_a^b g(\g'(t), \g'(t)) \, dt.
		\end{equation*}
	In a pseudo-Riemannian manifold (such as Minkowski space $\R^{p,q}$ when $q > 0$) the length functional is not always well-defined over the reals since the quadratic form is not positive-definite.
	That is why we consider the energy functional $E()$ in the Main Theorem instead of the length functional.  
	
\vskip 10pt
	
	%\item 
\n (5)  Let $f: \X \to \R^{3n+6,1}$ be the map guaranteed by the Main Theorem, and let $\g \in \G$.  
	Note that $E(\g)$ depends on the parameterization of $\g$ whereas $\ell(\g)$ does not.  
	So even though $E(f \circ \g) = E(\g)$, in theory it could be the case that $f$ increases the energy of certain parts of $\g$ while decreasing the energy in other parts in a manner which evens out to preserve the intrinsic energy.
	But an inspection of the proof of the Main Theorem shows that this is not the case here.
	If $\a$ is any geodesic segment of $\g$, then the map $f$ from the Main Theorem satisfies that $E(f \circ \a) = E(\a)$.  
	ie, $E(f \circ \g) = E(\g)$ with respect to {\it any} parameterization of $\g$.  
	So we see that the map $f$ really does preserve the intrinsic metric of $\X$ over the entirety of the geodesic paths in $\G$, which justifies the terminology in the abstract (and below) that this map is an ``approximate isometric embedding''.
	
\vskip 10pt
	
	%\item 
\n (6)  Without the one ``negative direction'', $\R^{3n+6,0} \cong \E^{3n+6}$ would be a metric space.  
	Then by the triangle inequality the map guaranteed by the Main Theorem would be an isometric embedding.  
	So in terms of the negative dimensions involved, the Main Theorem is the best that one can hope for.  
	Of course, one could hope to reduce the positive dimensions to as low as $2n$ and, in particular, any improvements to the dimensions in Theorem \ref{Minemyer} in Section 3 would directly reduce the positive dimension requirement of $3n+6$.  
	
\vskip 10pt

\n (7)  In \cite{OS}, Otsu and Shioya prove that every Alexandrov space with curvature bounded below admits an almost everywhere $C^0$-differential structure which induces the intrinsic metric almost everywhere.
While there are many differences between this result and our Main Theorem (such as no curvature restrictions here, and the vast difference between ``almost everywhere" and ``countable dense"), these results share the similarity that they extend notions from Riemannian geometry to the singular case, but not quite over the entire space.

\vskip 10pt
	
	%\item 
\n (8)  The space $\R^{N,1}$ is typically referred to as {\it Lorentzian space}.  
	If we define an {\it approximate isometric embedding} to be an embedding which is an isometry when restricted to a dense set of points, then the following Corollary is essentially a restatement of the Main Theorem.
	
	%\end{enumerate}

\vskip 10pt

\begin{corollary}
Every proper length metric space with finite covering dimension admits an approximate isometric embedding into Lorentzian space of an appropriate dimension.
\end{corollary}

\vskip 2pt

%So the Main Theorem says that, given any countable dense set of points in $\X$, there exists an embedding into Minkowski space with {\it only one} negative dimension which preserves the energy of some prescribed collection of geodesics between the points of this subset.
%Without the one negative dimension such a map would necessarily be a path isometry (in the case when $\X$ is a Finsler manifold) and the Main Theorem would be false.  
%But the presence of this one negative dimension allows ``infinite stretching" in both the positive and negative directions.

This paper is organized as follows.  
In order to prove the Main Theorem, we need to understand the energy functional on general metric spaces and how it behaves under perturbations of maps.  
So this is studied in Section 2, and many of the results are analogous to those related to the length functional developed by Petrunin in \cite{Petrunin}.
In Section 3 we discuss the theory of {\it indefinite metric polyhedra} and prove a technical Lemma (Lemma \ref{technical lemma}).  
Lemma \ref{technical lemma} is arguably the main construction which makes the proof of the Main Theorem work.

Sections 4 through 7 are dedicated to proving the Main Theorem.  
It should be noted that, for simplicity, we both assume that $\X$ is compact and we ignore the last sentence of the Main Theorem in Sections 4 through 6.
In Section 7 we finish the proof and discuss how to relax the previous two assumptions.
Finally, in Section 8, we outline a few of the more ``well known" preliminaries and prove Proposition \ref{Proposition}.

\vskip 10pt

An outline of the proof of the Main Theorem is as follows.

\vskip 10pt

{\it Step 1:}  We write $\D$ as the increasing union of finite subsets.  
That is, $\D = \bigcup_{i=1}^\infty \D_i$, $\D_i$ finite, and $\D_i \subset \D_{i+1}$ for all $i$.
For each $\D_i$ we construct an associated collection of ``allowable" geodesics $\G_i$ in such a way that $\G_i \subset \G_{i+1}$ for all $i$.  
Then $\G = \bigcup_{i=1}^\infty \G_i$.
Step 1 is covered in Section 4.

\vskip 10pt

{\it Step 2:}  We prove the following Lemma, which may be of its own independent interest.

\vskip 10pt

\begin{lemma}\label{Key Lemma}
Let $(\X, d)$ be an $n$-dimensional proper length metric space, let $\D \subseteq \X$ be any finite subset of $\X$, and let $\d > 0$.  
Then there exists a map $f: \X \rightarrow \mathbb{E}^{2n+5}$ which satisfies:
	\begin{enumerate}
	\item The map $f$ is $\d$-close to being an embedding.  That is, $f$ satisfies
	\[
	f(x) = f(x') \qquad \Longrightarrow \qquad d(x, x') < \d 
	\]
	\item The map $f$ is an isometry when restricted to $\D$.  That is, 
		\[
		d_{f(\X)}(f(x), f(x')) = d_{\X}(x, x')
		\]
		for all $x, x' \in \D$.
	\end{enumerate}
\end{lemma}

\vskip 2pt

In Lemma \ref{Key Lemma}, the notation $d_{f(\X)}(,)$ means the infimum of the lengths of all paths in $f(\X)$ between the given points.
For the purposes of proving the Main Theorem, the finite set $\D$ in Lemma \ref{Key Lemma} will be the initial set $\D_1$. 
The map $f = f_1$ from the Lemma will preserve the lengths of the geodesics in $\G_1$ and will ``blow up" the lengths of all other paths.  

The proof of the Main Theorem depends on a recursive construction.  
The proof of Lemma \ref{Key Lemma} is essentially the base case (when $i=1$) of this construction, and goes as follows.
%This proof contains the key construction in this paper, which is as follows.  
Given $\D_1$ and $\G_1$, we construct an open cover $\O_1$ which ``protects" the geodesics in $\G_1$.  
Let $\N_1$ denote the nerve of $\O_1$.  
We endow $\N_1$ with a metric $g_1$ which closely resembles the geometry of $\X$ near the image of $\G_1$ while blowing up the metric everywhere else.
Theorem \ref{Minemyer1} then allows us to find an isometric embedding $h_1: \N_1 \to \E^{2n+5}$.  
We then use a partition of unity to construct a map $\psi: \X \to \N_1$, and the composition $(h_1 \circ \psi)$ of the two preceding functions is the map $f$ in Lemma \ref{Key Lemma}.

This step will be completed in Section 5.

\vskip 10pt

{\it Step 3:}  Following the ideas above, we construct open covers $\O_i$ corresponding to the pair $(\D_i , \G_i)$ in such a way that the mesh of $\O_{i+1}$ is strictly less than one-third of the Lebesgue number for $\O_i$ for each $i$.  
We let $\N_i$ be the nerve of $\O_i$, and we use a result of Isbell \cite{Isbell} to construct maps $\psi_i : \X \to \N_i$ and piecewise linear maps $\v_{i+1,i} :\N_{i+1} \to \N_i$ for each $i$.  
The space $\X$ is then homeomorphic to the inverse limit of the system $\{\N_i, \v_{j,i} \}$ (see Figure \ref{fig1}).  

Using both the geometry of $\X$ and the maps $\left( \v_{i+1,i} \right)$ we construct Euclidean metrics on each nerve $\N_i$.  
The map $\v_{i+1,i}$ will be $1$-Lipschitz over ``most" of $\N_{i+1}$, but there will be a small controlled region where this map is not 1-Lipschitz.  
We use the map $\v_{i+1,i}$ to recursively construct piecewise linear isometric embeddings $h_{i+1} :\N_{i+1} \to \R^{3n+6,1}$.  
The map $h_{i+1}$ will be an approximation of $h_i$, and we need the one negative dimension to ``fix" the regions where $\v_{i+1,i}$ is not $1$-Lipschitz.
All of this is done in Section 6, which constitutes the bulk of this paper.

\vskip 10pt

{\it Step 4:}  We let $f_i := h_i \circ \psi_i$, and then $f := \lim_{i \to \infty} f_i$.  
This limit will converge uniformly, ensuring that $f$ is continuous.  
We then prove all of the necessary properties of $f$:  that it is injective and that it preserves the energy of all paths contained in $\G$.  
Finally, we discuss how to alter the proof to guarantee that $f$ is not locally Lipschitz (anywhere) when projected onto the first $3n+6$ positive coordinates, and how to deal with spaces which are proper instead of compact.
This is the content of Section 7, and completes the proof of the Main Theorem.

%\vskip 10pt

%The last section of this paper, Section 7, contains many of the preliminary definitions and results used in this paper.  

\begin{figure}
\begin{center}
\begin{tikzpicture}[scale=0.8]

\draw (0,2.5)node{$\X$};
\draw (0,0)node{$\N_i$};
\draw (-2.5,0)node{$\N_{i-1}$};
\draw (2.5,0)node{$\N_{i+1}$};
\draw (0,-2.5)node{$\R^{3n+6,1}$};

\draw[->] (-0.3,0) -- (-2,0);
%\draw (-2.1,0) -- (-2,0.1);
%\draw (-2.1,0) -- (-2,-0.1);

\draw[<-] (0.3,0) -- (2,0);
%\draw (0.3,0) -- (0.4,0.1);
%\draw (0.3,0) -- (0.4,-0.1);

\draw[<-] (-4.7,0) -- (-3,0);
%\draw (-4.7,0) -- (-4.6,0.1);
%\draw (-4.7,0) -- (-4.6,-0.1);

\draw[->] (4.7,0) -- (3,0);
%\draw (2.9,0) -- (3,0.1);
%\draw (2.9,0) -- (3,-0.1);

\draw[->] (0,2.2) -- (0,0.3);
%\draw (0.1,0.4) -- (0,0.3);
%\draw (-0.1,0.4) -- (0,0.3);

\draw[<-] (0,-2.2) -- (0,-0.3);
%\draw (0.1,-2.1) -- (0,-2.2);
%\draw (-0.1,-2.1) -- (0,-2.2);

\draw[<-] (2.25,0.35) -- (0.25,2.25);
%\draw (2.25,0.25) -- (2.05,0.25);
%\draw (2.25,0.25) -- (2.25,0.45);

\draw[<-] (-2.25,0.35) -- (-0.25,2.25);
%\draw (-2.25,0.25) -- (-2.05,0.25);
%\draw (-2.25,0.25) -- (-2.25,0.45);

\draw[->] (2.25,-0.25) -- (0.25,-2.15);
%\draw (0.25,-2.25) -- (0.25,-2.05);
%\draw (0.25,-2.25) -- (0.45,-2.25);

\draw[->] (-2.25,-0.25) -- (-0.25,-2.15);
%\draw (-0.25,-2.25) -- (-0.25,-2.05);
%\draw (-0.25,-2.25) -- (-0.45,-2.25);

\draw[fill=black!] (-4.9,0) circle (0.1ex);
\draw[fill=black!] (-5.1,0) circle (0.1ex);
\draw[fill=black!] (-5.3,0) circle (0.1ex);
\draw[fill=black!] (4.9,0) circle (0.1ex);
\draw[fill=black!] (5.1,0) circle (0.1ex);
\draw[fill=black!] (5.3,0) circle (0.1ex);

\draw (0,1.25)node[right]{$\psi_i$};
\draw (1.45,1.25)node[right]{$\psi_{i+1}$};
\draw (-1.45,1.25)node[left]{$\psi_{i-1}$};
\draw (-1.15,0)node[above]{$\v_{i, i-1}$};
\draw (1.15,0)node[above]{$\v_{i+1, i}$};
\draw (0,-1.25)node[right]{$h_i$};
\draw (1.25,-1.25)node[right]{$h_{i+1}$};
\draw (-1.25,-1.25)node[left]{$h_{i-1}$};

\draw (3,-2)node[right]{$f_i := h_i \circ \psi_i$ for all $i$};

\end{tikzpicture}
\end{center}
\caption{Diagram of the maps and spaces involved in the proof of the Main Theorem.}
\label{fig1}
\end{figure}
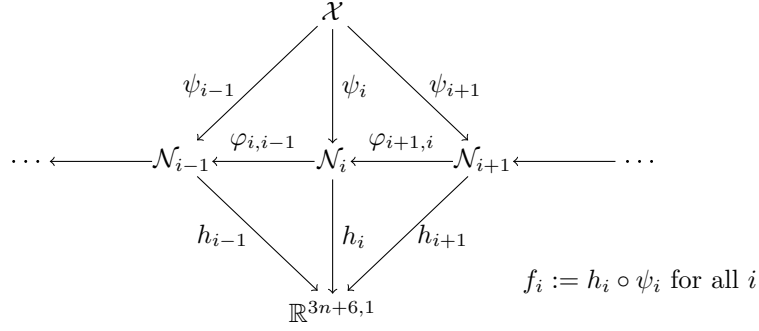

\vskip 20pt

\section{The length and energy functionals on paths in metric spaces}
Let $(\X, d)$ denote an arbitrary metric space, and let $\a : [a,b] \to \X$ be a continuous path.  
Define the {\it length} of $\a$, denoted $\ell(\a)$, by
	\begin{equation}\label{definition of length}
	\ell(\a) = \text{sup} \sum_{k=1}^n d(\a(t_{k-1}),\a(t_k))
	\end{equation}
where the supremum is taken over all subdivisions $a = t_0 < t_1 < \hdots < t_n = b$ with no bound on $n$.  
From the triangle inequality it is clear that finer subdivisions increase the above sum (or, at least, do not decrease it).
So one could instead only consider equidistant subdivisions where for general $k$, $t_k = a + k \Delta t$ and $\Delta t = \frac{b-a}{n}$.  
It is also clear from equation \eqref{definition of length} that $\ell(\a)$ does not depend on the parameterization of $\a$.  

Define the {\it velocity} $v_\a(t)$ by
	\begin{equation*}
	v_\a (t) = \lim_{\e \to 0} \frac{d(\a(t), \a(t+\e))}{| \e |}
	\end{equation*}
provided the limit exists.  
If the path $\a$ is Lipschitz then, by Rademacher's Theorem \ref{Rademacher}, $v_\a (t)$ exists almost everywhere.  
Moreover, it is proved in \cite{BBI} that for a Lipschitz path $\a$
	\begin{equation}\label{Lebesgue integral}
	\ell(\a) = \int_a^b v_\a (t) \, dt
	\end{equation}
where the integral in equation \eqref{Lebesgue integral} denotes the Lebesgue integral.  
Equation \eqref{Lebesgue integral} becomes very intuitive if one rewrites equation \eqref{definition of length} as
	\begin{equation}\label{definition of length 2}
	\ell(\a) = \text{sup} \sum_{k=1}^n \left[ \frac{d(\a(t_{k-1}),\a(t_k))}{t_k - t_{k-1}} \left( t_k - t_{k-1} \right) \right].
	\end{equation}
	
Since $v_\a (t)$ is defined almost everywhere on $[a,b]$ and is Lebesgue integrable (for a Lipschitz path $\a$), so to is $v_\a^2(t)$.
Then motivated by both classical Riemannian geometry and equation \eqref{Lebesgue integral}, we define the {\it energy} of a path $\a$ by
	\begin{equation}\label{integral for energy}
	E(\a) = \int_a^b v_\a^2 (t) \, dt.
	\end{equation}
	
It follows directly from the argument in \cite{BBI} for equation \eqref{Lebesgue integral} that
	\begin{equation}\label{definition of energy}
	E(\a) = \text{sup} \sum_{k=1}^n \left[ \frac{d^2(\a(t_{k-1}),\a(t_k))}{(t_k - t_{k-1})^2} \left( t_k - t_{k-1} \right) \right].
	\end{equation}
Of course, equation \eqref{definition of energy} simplifies to
	\begin{equation}\label{definition of energy 2}
	E(\a) = \text{sup} \sum_{k=1}^n \frac{d^2(\a(t_{k-1}),\a(t_k))}{t_k - t_{k-1}}.
	\end{equation}
In equation \eqref{definition of energy} it may not be so clear that this sum is nondecreasing with respect to finer and finer subdivisions.  
But this is the case, and is the content of the following Lemma.

\begin{lemma}\label{triangle inequality for energy}
The summands of the energy functional satisfy the triangle inequality.
That is, for all $p, q, r \in \X$ and $\d, \e > 0$ with $0 < \d < \e$ we have that
	\begin{equation*}
	\frac{d^2(p, r)}{\e} \leq \frac{d^2 (p, q)}{\d} + \frac{d^2(q, r)}{\e - \d}.
	\end{equation*}
Moreover, if equality holds then
	\[
	d(p,r) = d(p,q) + d(q,r)
	\]
\end{lemma}

The idea for this proof is virtually identical to that of Lemma 3.1.2 in \cite{Salviano}

\begin{proof}
Let
	\begin{equation*}
	v = \frac{d(p,r)}{\e}  \qquad
	v_1 = \frac{d(p,q)}{\d}  \qquad
	v_2 = \frac{d(q,r)}{\e - \d}
	\end{equation*}
So what we need to prove is that
	\begin{equation}\label{equation in energy lemma}
	\e v^2  \leq \d v_1^2 + (\e - \d) v_2^2.
	\end{equation}

Let $a_1, a_2 \in \R$ be such that $v_1 = v + a_1$ and $v_2 = v + a_2$.  
Then
	\begin{align}
	&\e v^2  \leq \d v_1^2 + (\e - \d) v_2^2  \nonumber  \\
	\Longleftrightarrow \qquad &\e v^2 \leq \d (v+a_1)^2 + (\e - \d)(v+a_2)^2  \nonumber  \\
	\Longleftrightarrow \qquad  &\e v^2 \leq \d (v^2 + 2a_1v + a_1^2) + (\e - \d)(v^2 + 2a_2v + a_2^2)  \nonumber  \\
	\Longleftrightarrow \qquad  &0 \leq 2v(a_1 \d + a_2(\e - \d)) + \d a_1^2 + (\e - \d)a_2^2.  \label{0 inequality}
	\end{align}
Since both $2v$ and $\d a_1^2 + (\e - \d)a_2^2$ are nonnegative, we complete the proof by showing that $a_1 \d + a_2(\e - \d) \geq 0$.  

From the triangle inequality in $\X$
	\begin{align*}
	&d(p,r) \leq d(p,q) + d(q,r)  \\
	\Longrightarrow \qquad &\e v \leq \d v_1 + (\e - \d) v_2 = \d(v+a_1) + (\e - \d)(v + a_2)   \\
	\Longrightarrow \qquad &0 \leq a_1 \d + (\e - \d) a_2.
	\end{align*}
	
For the last part of the Lemma, just note that inequality \eqref{equation in energy lemma} is an equality if and only if inequality \eqref{0 inequality} is an equality.  
But that is true if and only if $a_1 = a_2 = 0$, and thus $v = v_1 = v_2$.
This gives that
	\[
	\frac{d(p,r)}{\e} = \frac{d(p,q)}{\d}	\qquad	\text{and}	\qquad	\frac{d(p,r)}{\e} = \frac{d(q,r)}{\e - \d}.
	\]
Cross multiplying both equations yields
	\[
	\d d(p,r) = \e d(p,q)	\qquad 	\text{and}	\qquad	(\e - \d) d(p,r) = \e d(q,r).
	\]
Then adding these two equations and dividing by $\e$ gives the desired equality.  
\end{proof}

\begin{remark}
Since the summands of the energy functional satisfy the triangle inequality, we can rewrite equation \eqref{definition of energy 2} as
	\begin{equation*}
	E(\a) = \lim_{n \to \infty} \sum_{k=1}^n \frac{d^2(\a(t_{k-1}),\a(t_k))}{t_k - t_{k-1}}
	\end{equation*}
where the subdivision is of $n$ equidistant subintervals (and, of course, provided that the limit exists).  
That is, where $t_k = a + k \Delta t$ and $\Delta t = \frac{b-a}{n}$.  
\end{remark}

In this paper we are mainly concerned with computing the energy of geodesic paths, in which case equations \eqref{definition of length} through \eqref{definition of energy 2} simplify greatly.  
Let $\a: [a,b] \to \X$ be a geodesic between the points $p = \a(a)$ and $q = \a(b)$.  
The definition that $\a$ is a {\it geodesic} is that for all $s, t \in [a,b]$ with $s < t$ we have that
	\[
	\ell \left( \a \bigr|_{[s,t]} \right) = \frac{d(\a(s),\a(t))}{t-s}.
	\]
Comparing with the definition of the velocity and equation \eqref{Lebesgue integral} we see that $v_\a(t)$ is constant for any geodesic $\a$ and moreover
	\[
	v_\a(t) := v_\a = \frac{d(\a(a),\a(b))}{b-a} = \frac{d(p,q)}{b-a}.
	\]
Plugging this into equations \eqref{Lebesgue integral} and \eqref{integral for energy} we have that, for $\a$ a geodesic from $p$ to $q$ defined on the interval $[a,b]$:
	\begin{align}
	\ell(\a) &= d(p,q) = v_\a (b-a)  \label{length of a geodesic}  \\
	E(\a) &= \frac{d^2(p,q)}{b-a} = v_\a^2 (b-a)  \label{energy of a geodesic}.
	\end{align}
	
\begin{remark}\label{energy is a scalar multiple of length}
Notice from equations \eqref{length of a geodesic} and \eqref{energy of a geodesic} that for a geodesic $\a$ defined on the interval $[a,b]$
	\begin{equation*}
	E(\a) = v_\a  \ell(\a).
	\end{equation*}
So the energy of a geodesic is just a scalar multiple of the length (where the scalar depends on $[a,b]$).
\end{remark}

\vskip 20pt

\subsection{Maps which preserve energy}
Let $(\X, \dx)$ and $(\Y, \dy)$ be two metric spaces, and let $f: \X \to \Y$ be a continuous map.
For any path $\a: [a,b] \to \X$, define $f^*E(\a) := E(f \circ \a)$ to be the {\it pull-back energy functional} with respect to $f$.  
We would like to study how the pull-back energy functional behaves under perturbations of maps.  
More precisely, if $f, g: \X \to \Y$ are two maps such that $\dy (f(x), g(x)) < \d$ for some small $\d > 0$, we would like to know if there is any relationship between $f^* E$ and $g^* E$.  

The energy functional, being intimately related to the length functional, should share a lot of similar properties.  
The following Lemma \ref{variational energy lemma}, Theorem \ref{perturbed energy theorem}, and Corollary \ref{perturbed energy corollary} are analogous to Lemma 2.3 and Proposition 2.2 in \cite{Petrunin}.
Due to the dependence of the energy functional on the parameterization of the given path, these results are a little weaker than their counterparts proved by Petrunin.  
But they are sufficient for what will be needed to prove the Main Theorem.

\vskip 10pt

\begin{lemma}\label{variational energy lemma}
Let $f:\X \to \Y$ be a continuous function between two metric spaces and let $\d > 0$.  
Fix a positive integer $n$ and an interval $[a,b]$.
Then for all $C_1 > 1$ there exists a constant $C_2 = C_2(C_1,n,[a,b]) > 0$ such that 
	\begin{equation}\label{energy inequality}
	\sum_{k=1}^n \frac{\dy^2(f \a(t_{k-1}),f \a(t_k))}{t_k - t_{k-1}} \leq C_1 \sum_{k=1}^n \frac{\dy^2(g \a(t_{k-1}),g \a(t_k))}{t_k - t_{k-1}} + C_2 \d^2
	\end{equation}
for all $g: \X \to \Y$ which satisfies that $\dy(f(x), g(x)) < \d$ for all $x \in \X$, and for all paths $\a :[a,b] \to \X$.
In equation \eqref{energy inequality}, $t_k = a + k \Delta t$, $\Delta t = \frac{b-a}{n}$, and $f \a$ and $g \a$ denote $f \circ \a$ and $g \circ \a$, respectively.
\end{lemma}

\vskip 2pt

\begin{proof}
By Lemma \ref{triangle inequality for energy}, on each interval $[t_{k-1},t_k]$ we have that
	\begin{align*}
	&\frac{\dy^2(f \a(t_{k-1}),f \a(t_k))}{t_k - t_{k-1}}  \\
	%&\leq \frac{\dy^2(\b(t_{k-1}), \b(t'))}{\frac{C_1-1}{2C_1}(t_k - t_{k-1})} + \frac{\dy^2(\b(t'),\b(t''))}{\frac{1}{C_1}(t_k - t_{k-1})} + \frac{\dy^2(\b(t''), \b(t_k))}{\frac{C_1-1}{2C_1}(t_k - t_{k-1})}  \\
	&\leq \frac{\dy^2(f\a(t_{k-1}), g\a(t_{k-1}))}{\frac{C_1-1}{2C_1}(t_k - t_{k-1})} + \frac{\dy^2(g\a(t_{k-1}),g\a(t_k))}{\frac{1}{C_1}(t_k - t_{k-1})} + \frac{\dy^2(g\a(t_k), f\a(t_k))}{\frac{C_1-1}{2C_1}(t_k - t_{k-1})}  \\
	&\leq \frac{C_1 \dy^2(g\a(t_{k-1}),g\a(t_k))}{t_k - t_{k-1}} + \frac{4 \d^2 C_1}{(C_1 - 1)(t_k - t_{k-1})}  \\
	&= C_1 \frac{\dy^2(g\a(t_{k-1}),g\a(t_k))}{t_k - t_{k-1}} + \frac{4 n C_1}{(C_1 - 1)(b-a)} \d^2.
	\end{align*}
Therefore,
	\begin{align*}
	\sum_{k=1}^n \frac{\dy^2(f \a(t_{k-1}),f \a(t_k))}{t_k - t_{k-1}} &\leq \sum_{k=1}^n C_1 \frac{\dy^2(g\a(t_{k-1}),g\a(t_k))}{t_k - t_{k-1}} + \sum_{k=1}^n \frac{4 n C_1}{(C_1 - 1)(b-a)} \d^2  \\
	&= C_1 \sum_{k=1}^n \left( \frac{\dy^2(g\a(t_{k-1}),g\a(t_k))}{t_k - t_{k-1}} \right) + \frac{4 n^2 C_1}{(C_1 - 1)(b-a)} \d^2  \\
	&= C_1 \sum_{k=1}^n \frac{\dy^2(g\a(t_{k-1}),g\a(t_k))}{t_k - t_{k-1}} + C_2 \d^2
	\end{align*}
where $\ds{C_2 = \frac{4n^2 C_1}{(C_1 - 1)(b-a)}}$.

\end{proof}

\vskip 10pt

\begin{theorem}\label{perturbed energy theorem}
Let $f: \X \to \Y$ be a continuous function between two metric spaces and let $\a: [a,b] \to \X$ be any path such that $f^*E(\a) < \infty$.
Then given $C > 1$ and $\l > 0$, there exists $\d = \d(f, \l, C, \a) > 0$ such that for any continuous map $g: \X \to \Y$ satisfying
	\begin{equation*}
	\dy(f(x), g(x)) < \d \qquad \text{for all }x \in \X
	\end{equation*}
we have that
	\begin{equation*}
	f^*E(\a) < C g^*E(\a) + \l.
	\end{equation*}
\end{theorem}

\vskip 2pt

\begin{proof}
Let $C = C_1$ from Lemma \ref{variational energy lemma}, and let $C_2 > 0$ be the constant guaranteed to exist by the same Lemma.  
Then since $f^*E(\a) < \infty$, there exists an integer $n$ so that
	\begin{equation*}
	f^*E(\a) \leq \sum_{k=1}^n \frac{\dy^2(f \a(t_{k-1}),f \a(t_k))}{t_k - t_{k-1}} + \frac{\l}{2}
	\end{equation*}
where, as in Lemma \ref{variational energy lemma}, $t_k = a + k \Delta t$ and $\Delta t = \frac{b-a}{n}$.  
Choose $\d = \sqrt{\frac{\l}{2 C_2}}$.  
Then
	\begin{align*}
	\sum_{k=1}^n \frac{\dy^2(f \a(t_{k-1}),f \a(t_k))}{t_k - t_{k-1}} + \frac{\l}{2}
	&\leq C \sum_{k=1}^n \frac{\dy^2(g \a(t_{k-1}),g \a(t_k))}{t_k - t_{k-1}} + C_2 \d^2 + \frac{\l}{2}  \\
	&= C \sum_{k=1}^n \frac{\dy^2(g \a(t_{k-1}),g \a(t_k))}{t_k - t_{k-1}} + \l  \\
	&\leq C g^*E(\a) + \l.
	\end{align*}

\end{proof}

In the proof of the Main Theorem we need to simultaneously apply Theorem \ref{perturbed energy theorem} to a finite number of paths.  
So for convenience we note the following direct Corollary.

\vskip 10pt

\begin{corollary}\label{perturbed energy corollary}
Let $f: \X \to \Y$ be a continuous function between two metric spaces and let $\a_1, \hdots, \a_m$ be a finite number of paths in $\X$ such that $f^*E(\a_k) < \infty$ for all $1 \leq k \leq m$.  
Then, given $C > 1$ and $\l > 0$, there exists $\d = \d(f, \l, C, \a_1, \hdots, \a_m) > 0$ such that for any continuous map $g: \X \to \Y$ satisfying
	\begin{equation*}
	\dy(f(x), g(x)) < \d \qquad \text{for all }x \in \X
	\end{equation*}
we have that
	\begin{equation*}
	f^*E(\a_k) < C g^*E(\a_k) + \l
	\end{equation*}
for all $1 \leq k \leq m$.  
\end{corollary}

\vskip 20pt

%beginning of section 3
\section{Piecewise flat (indefinite) metrics on polyhedra}

\subsection{Quadratic forms associated to metric polyhedra}

Many parts of this Subsection are also contained in \cite{Minemyer4}.

Let $(\X, \T, g)$ be an {\it indefinite metric polyhedron}.  
This just means that $\X$ is a topological space, $\T$ is a locally finite simplicial triangulation of $\X$, and $g$ is a function which assigns a real number to each edge of $\T$.
This function $g$ defines a unique indefinite metric over each simplex $\s \in \T$, and thus over all of $\X$, as follows.

%Let $x \in \X$ be a point.  
%Then there is a unique $k$-dimensional simplex $\s_x = \langle v_0, v_1, ..., v_k \rangle \in \T$ such that $x$ is interior to $\s_x$.
%One can then consider a $k$-dimensional tangent space at $x$, denoted by $T_x \X$, whose dimension certainly depends on the triangulation $\T$.  
Let $\s = \langle v_0, v_1, ..., v_k \rangle \in \T$ be a $k$-dimensional simplex.
Embed $\s$ into $\R^k$ by identifying $v_0$ with the origin, and for $1 \leq i \leq k$ identifying $v_i$ with the terminal point of the $i^{th}$ standard basis vector.
Let $\vec{w}_i := v_i - v_0$ denote the $i^{th}$ standard basis vector, and let $e_{ij}$ denote the edge in $\s$ between the vertices $v_i$ and $v_j$.  

The indefinite metric $g$ (and our choice of ordering of the vertices of $\s$) defines a quadratic form $G$ on $\R^k$ as follows.  Define
	\begin{align*}
	G(w_i) &= s(g (e_{0i}))  \\
	G(w_i - w_j) &= s(g(e_{ij}))
	\end{align*}
where
	\begin{equation*}
	\ds{s(x) = \left\{ \begin{array}{rl} x^2 & \quad \text{if } x \geq 0 \\ -x^2 & \quad \text{if } x < 0 \end{array} \right.  }
	\end{equation*} 
is the {\it signed squared} function.
Let $\la , \ra_g$ denote the symmetric bilinear form associated to $G$.  
A simple calculation, worked out in \cite{Minemyer2} and \cite{Minemyer Geoghegan}, shows that
	\begin{equation}\label{definition of quadratic form}
	\langle \vec{w}_i , \vec{w}_j \rangle_g \, = \frac{1}{2} \left( G(\vec{w}_i) + G(\vec{w}_j) - G(\vec{w}_i - \vec{w}_j) \right).   
	\end{equation}
So $G$ is completely determined by the above definition, which is sometimes called the {\it polarization identity} of $G$.  
We will abuse notation and refer to $G$ as a quadratic form on $\s$, when rigorously $G$ is really a quadratic form on $\R^k$.  

Given a quadratic form $G$ on $\s$ as above, define the {\it energy} of an edge $e$ to simply be $G(e)$.  
Equation \eqref{definition of quadratic form} shows that a quadratic form is uniquely determined by the energy that it assigns to each edge.  
Thus, the set of quadratic forms on a $k$-dimensional simplex $\s$ can naturally be identified with $\R^n$ where $n = $  $k+1 \choose 2$.  
Each coordinate in $\R^n$ is parameterized by the energy of the corresponding edge of $\s$.  

Now let $f: \X \rightarrow \R^{p,q}$ be any continuous function, where $\R^{p,q}$ denotes Minkowski space of signature $(p,q)$.  
Let $\s$ be as above.  
The map $f$ determines a unique indefinite metric $g_f$ on $(\X, \T)$ by defining
	\begin{equation}\label{definition of induced metric}
	g_f (e_{ij}) := \la f(v_i) - f(v_j) , f(v_i) - f(v_j) \ra
	\end{equation}
where $v_i$ and $v_j$ are the vertices incident with $e_{ij}$, and where $\la , \ra$ is the Minkowski bilinear form on $\Rpq$.  
The indefinite metric $g_f$ induces a quadratic form $G_f$ on $\R^k$ just as above, called the {\it induced quadratic form} of $f$.  
We say that $f$ is a {\it piecewise linear isometry} (or {\it pl isometry}) of $\X$ into $\Rpq$ if $f$ is piecewise linear (meaning that it is simplicial on some subdivision of $\T$) and if $G = G_f$ on all simplices in a subdivision of $\T$ on which $f$ is simplicial.  
The map $f$ is a {\it pl isometric embedding} if in addition to being a pl isometry it is also an embedding.

We say that an indefinite metric polyhedron $(\X, \T, g)$ is {\it Euclidean} if the quadratic form $G(\s)$ induced by $g$ on $\s \in \T$ is positive definite for all $\s \in \T$.  
So Euclidean polyhedra can be viewed as combinatorial analogues to Riemannian manifolds.
It is well known that the collection of positive definite quadratic forms is closed under addition and positive scalar multiplication.  
Thus, they form an open cone within the collection of all indefinite metric polyhedra, an observation which was also pointed out by Rivin in \cite{Rivin}.

%Let $(\X, \T, g)$ be an indefinite metric polyhedron, let $f: \X \to \R^{p,q}$ be a pl map, and let $\T'$ be a subdivision of $\T$ on which $f$ is simplicial.
%We say that $f$ is {\it short}, or {\it 1-Lipschitz}, if $G(\s) - G_f(\s) \geq 0$ for all $\s \in \T'$.  
%Note that if $\X$ is a Euclidean polyhedron and if $q=0$, then this definition is equivalent to the usual definition of a map being 1-Lipschitz.

%\vskip 20pt

%\subsection*{1-Lipschitz maps}
%Let $(\X, \T, g

\vskip 20pt

\subsection{1-Lipschitz maps and an approximation Lemma.}

Let $(\X, \T, g)$ be an indefinite metric polyhedron, let $f: \X \to \R^{p,q}$ be a pl map, and let $\T'$ be a subdivision of $\T$ on which $f$ is simplicial.
We say that $f$ is {\it 1-Lipschitz}, or {\it short}, if $G(\s) - G_f(\s) \geq 0$ for all $\s \in \T'$.  
Note that if $\X$ is a Euclidean polyhedron and if $q=0$, then this definition is equivalent to the usual definition of a map being 1-Lipschitz.

The following technical Lemma will be needed in the construction of the maps $(h_i)$ is Subsection 6.5.

\vskip 10pt

\begin{lemma}\label{technical lemma}
Let $\s^k = \langle v_0, v_1, \hdots, v_k \rangle$ be a $k$-dimensional simplex, and for each $i$ and $j$ let $e_{ij}$ denote the edge between $v_i$ and $v_j$.  
Let $G_M$ be the quadratic form on $\s$ defined by 
	\begin{align*}
	&G_M(e_{01}) = \a^2  \\
	&G_M(e_{ij}) = M^2 \quad \text{for all } \{ i,j \} \neq \{ 0,1 \}
	\end{align*}
where $\a, M \in \R^{> 0}$.  
Let $f: \s \to \R^{k,1}$ be a simplicial map such that $G_f(e_{01}) = \b^2$ for some positive real $\b > \a$.  
Then for all $\e > 0$ there exists a constant $M' > 0$ satisfying the following.  
If $M > M'$ then there exists a pl map $f': \s \to \R^{k,1}$ satisfying
	\begin{enumerate}
	\item  For any segment $e$ of $e_{01}$ on which $f'$ is simplicial, we have that $G_{f'}(e) = G_M(e)$.  So $f'$ preserves the ``length" of $e_{01}$.
	\item  The map $f'$ is $1$-Lipschitz with respect to $G_M$.
	\item  The map $f'$ is an $\e$-approximation of $f$.
	\end{enumerate}
\end{lemma}

\vskip 10pt

%%%%%%%%%%%%%%%%%%%%%%%%%%%%%%%%%%%%%%%%%%%%%%%%%%%%%%%%%%%%%%%%%%%%
\begin{comment}
Before proving Lemma \ref{technical lemma}, let us first discuss what is meant by ``$(f')^*E(e_{01}) = G_M(e_{01}) = \a^2$" in statement (1) of the Lemma.  
Since $f$ is simplicial, $f(e_{01})$ is a line segment in $\R^{k,1}$.  
Thus we can compute $G_f(e_{01})$ simply as $\langle f(v_1) - f(v_0), f(v_1) - f(v_0) \rangle$, which by assumption is $\b^2$.  
The subdivision of $\s$ on which $f'$ is simplicial will subdivide $e_{01}$ into $N$ equidistant subintervals labeled $e_1, \hdots, e_N$.  
We then have that $G_M(e_i) = \frac{\a^2}{N^2}$ for each $i$, 
\end{comment}
%%%%%%%%%%%%%%%%%%%%%%%%%%%%%%%%%%%%%%%%%%%%%%%%%%%%%%%%%%%%%%%%%%%%

The idea of the proof is as follows.  
We take a very fine subdivision of $f(e_{01})$ and ``wiggle" it in the one negative direction, decreasing the energy of this edge to $\a^2$.  
We then choose $M'$ large enough so that the resulting pl map is 1-Lipschitz.  

\begin{proof}

We first construct a subdivision $\T$ of $\s$ as follows.  
Choose a large even positive integer $N$ and subdivide the edge $e_{01}$ into $N$ equidistant subintervals.  
Label these new vertices $v_0 = w_0, w_1, w_2, \hdots, w_N = v_1$.  
We then extend to a subdivision over all of $\s$ by gluing in a $(k-1)$-simplex $\langle w_i, v_2, v_3, \hdots, v_k \rangle$ for each $i$, and correspondingly inserting $(N+1)$ $k$-simplices of the form $\langle w_{i-1}, w_i, v_2, \hdots, v_k \rangle$.  
Let $e_i$ denote the edge between $w_{i-1}$ and $w_i$, and note that $G_M(e_i) = \frac{\a^2}{N^2}$ for all $i$.  

Let us now construct the map $f'$, which will be simplicial over the subdivision $\T$.  
First off, define $f'(v_i) = f(v_i)$ for all $0 \leq i \leq k$.  
So let us define $f'$ over $w_1, \hdots, w_{N-1}$.  
If $i$ is even then define $f'(w_i) = f(w_i)$.  
For $i$ odd, define
	\begin{equation*}
	f'(w_i) = f(w_i) + \frac{\sqrt{\b^2 - \a^2}}{N} \vec{v}
	\end{equation*}
where $\vec{v}$ is a vector that is Lorentz orthogonal to $f(e_{01})$ and satisfying $\langle \vec{v}, \vec{v} \rangle = -1$ (see Figure \ref{Figure for technical lemma}).  
Such a vector $\vec{v}$ exists since $\langle f(e_{01}), f(e_{01}) \rangle = \b^2 > 0$.

For $i$ odd, let us compute the energy of the line segment between $f'(w_{i-1})$ and $f'(w_i)$ (which will be the same when $i$ is even, since $(i-1)$ will then be odd).  
	\begin{align*}
	&\langle f'(w_i) - f'(w_{i-1}), f'(w_i) - f'(w_{i-1}) \rangle  \\
	&= \left\langle f(w_i) + \frac{\sqrt{\b^2 - \a^2}}{N} \vec{v} - f(w_{i-1}),  f(w_i) + \frac{\sqrt{\b^2 - \a^2}}{N} \vec{v} - f(w_{i-1}) \right\rangle  \\
	&= \langle f(w_i) - f(w_{i-1}), f(w_i) - f(w_{i-1}) \rangle + \frac{\b^2 - \a^2}{N^2} \langle \vec{v}, \vec{v} \rangle  \\
	&= \frac{\b^2}{N^2} - \left( \frac{\b^2 - \a^2}{N^2} \right) = \frac{\a^2}{N^2}.
	\end{align*}
Thus $G_{f'}(e_i) = \frac{\a^2}{N^2} = G_M(e_i)$.  
Also, by choosing $N$ large we can ensure that $f'$ is as close of an approximation to $f$ as we like.  
Therefore, we may choose $N$ large enough so that $f'$ satisfies conditions (1) and (3) of the Lemma.

For condition (2), Let $W_i$ denote the $k$-simplex $\langle w_{i-1}, w_i, v_2, \hdots, v_k \rangle \in \T$.  
Let $G_i^M$ denote the quadratic form induced by $G_M$ on $W_i$, and let $G_i'$ denote the quadratic form on $W_i$ induced by the map $f'$.

Recall that the collection of (flat) Euclidean metrics on a $k$-dimensional simplex forms an open cone $\mathcal{C}$ in $\R^n$ (with $n = $ $k+1 \choose 2$) where the axes of $\R^n$ are parameterized by the energies of the edges of the simplex.  
The cone $\mathcal{C}$ is obviously symmetric about the line segment starting at the origin and passing through $(1,1,\hdots,1)$, since each edge is the same as any other.  
We know that $G_i^M(e_i) = \frac{\a^2}{N^2} = G_i'(e_i)$, and so let us project $\R^n$ onto the hyperplane spanned by the other $n-1$ edges (which we will denote $\R^{n-1}$).  
Let $g_i^M$, $g_i'$, and $\mathscr{C}$ denote the images of $G_i^M$, $G_i'$, and $\mathcal{C}$ under this projection, respectively.  
Due to the symmetry of $\mathcal{C}$, $\mathscr{C}$ is just the corresponding cone in $\R^{n-1}$.   

Now for $M > \a$ the quadratic form $G_M$ is positive definite.  
That is, $G_M \in \mathcal{C}$.  
Thus, each induced metric $G_i^M \in \mathcal{C}$ and therefore $g_i^M \in \mathscr{C}$.  
Note that for $N$ large, $g_i^M$ is close to the line segment at the center of $\mathscr{C}$.  
Let $\d_i = d(g_i^M, \partial \mathscr{C})$ denote the distance from $g_i^M$ to the boundary of the cone $\mathscr{C}$.  
Then we have that $\d_i > 0$ and $\lim_{M \to \infty} \d_i = \infty$.  
So we can find some positive $M_i$ such that $g_i^M - g_i' \in \mathscr{C}$ for all $M \geq M_i$.  
Letting $M' = \max \{ M_i \}$ completes the proof.

\end{proof}

\begin{remark}\label{technical lemma remark}
First note that this Lemma obviously holds for any map into $\R^{p,q}$ if $p \geq k$ and $q \geq 1$.
But also notice that in the assumptions of the Lemma we did not need the full strength that $G_M(e_{ij}) = M$ for $\{ i,j \} \neq \{ 0,1 \}$.  
We just needed that the form $G_M$ was positive definite, and that $d(G_M, \partial \mathcal{C}) \to \infty$ as $M \to \infty$.  

Finally, note that we could apply this Lemma to a simplicial map $f$ defined over a locally finite simplicial complex with respect to some specified edge $e$ (and where the conclusion would be that $f'$ is 1-Lipschitz on the closed star of $e$).  
The edge $e$ may be contained in more than one simplex, but since the complex is locally finite we just subdivide using the largest $N$ required by any simplex containing $e$.
\end{remark}

\begin{figure}
\begin{center}
\begin{tikzpicture}[scale=0.85]

\draw (1,0) -- (4,2) -- (7,0) -- (1,0);
\draw[fill=black!] (1,0) circle (0.3ex);
\draw[fill=black!] (4,2) circle (0.3ex);
\draw[fill=black!] (7,0) circle (0.3ex);
\draw (2.5,1.5)node[left]{$\b^2$};
\draw (4,3)node{$f(\s)$};

\draw[->] (8,1) -- (9.5,1);

\draw (13.1,2) -- (16,0) -- (10,0);
\draw (10,0) -- (10.2,0.6);
\draw[dotted] (10.2,0.6) -- (10.6,0.4);
\draw (10.6,0.4) -- (10.8,1);
\draw[dotted] (10.8,1) -- (11.2,0.8);
\draw (11.2,0.8) -- (11.4,1.4);
\draw[dotted] (11.4,1.4) -- (11.8,1.2);
\draw (11.8,1.2) -- (12,1.8);
\draw[dotted] (12,1.8) -- (12.4,1.6);
\draw (12.4,1.6) -- (12.6,2.2);
\draw[dotted] (12.6,2.2) -- (13.1,2);
\draw (10.2,0.6) -- (16,0);
\draw[dotted] (10.6,0.4) -- (16,0);
\draw (10.8,1) -- (16,0);
\draw[dotted] (11.2,0.8) -- (16,0);
\draw (11.4,1.4) -- (16,0);
\draw[dotted] (11.8,1.2) -- (16,0);
\draw (12,1.8) -- (16,0);
\draw[dotted] (12.4,1.6) -- (16,0);
\draw (12.6,2.2) -- (16,0);

\draw[fill=black!] (10,0) circle (0.3ex);
\draw[fill=black!] (13.1,2) circle (0.3ex);
\draw[fill=black!] (16,0) circle (0.3ex);
\draw[fill=black!] (10.2,0.6) circle (0.3ex);
\draw[fill=black!] (10.6,0.4) circle (0.3ex);
\draw[fill=black!] (10.8,1) circle (0.3ex);
\draw[fill=black!] (11.2,0.8) circle (0.3ex);
\draw[fill=black!] (11.4,1.4) circle (0.3ex);
\draw[fill=black!] (11.8,1.2) circle (0.3ex);
\draw[fill=black!] (12,1.8) circle (0.3ex);
\draw[fill=black!] (12.4,1.6) circle (0.3ex);
\draw[fill=black!] (12.6,2.2) circle (0.3ex);

\draw[->] (10.5,1.8) -- (11.2,1.2);
\draw (10.5,1.8)node[above]{$\frac{\a^2}{N^2}$};
\draw (13,3)node{$f'(\s)$};

\end{tikzpicture}
\end{center}
\caption{The construction of $f'$ in Lemma \ref{technical lemma} (with $N=10$)}
\label{Figure for technical lemma}
\end{figure}
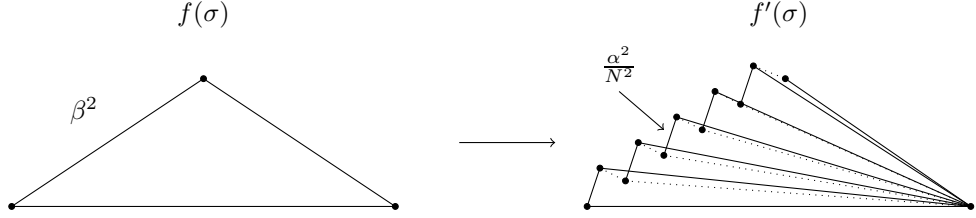

\vskip 20pt

\subsection{Isometric embeddings of polyhedra into Euclidean space}

We now want to state three results pertaining to isometries of Euclidean polyhedra into Euclidean space.  
%%%%%%%%%%%%%%%%%%%%%%%%%%%%%%%%%%%%%%%%%%%%%%%%%%%%%%%%%%%%%%%%%%%%
\begin{comment}
But we first need some terminolory.

Let $\P$ be a polyhedron and let $x \in \P$.  
Fix a triangulation $\T$ of $\P$.  
For a vertex $v$, the closed star of $v$ will be denoted by $St(v)$.  
We define $St^2(v) := \bigcup_{u \in St(v)} St(u)$ and for any $k \in \mathbb{N}$ we recursively define $St^{k + 1}(v) := \bigcup_{u \in St^k(v)} St(u)$.  
Then define the \emph{$k^{th}$ shell about $x$}, denotes by $Sh^k(x)$, as:
	\begin{enumerate}
	\item $Sh^1(x) = St(x)$
	\item $Sh^k(x) = St^k(x) \setminus St^{k - 1}(x)$ for $k \geq 2$
	\end{enumerate}
			
Notice that $Sh^k(x) \cap Sh^l(x) = \emptyset$ for $k \neq l$ and that $\bigcup_{i = 1}^{\infty} Sh^i(x) = \P$.  Also note that $St^k(x)$ and $Sh^k(x)$ both depend on the triangulation that we are considering.  If we want to emphasize the triangulation, then we will put it as a subscript.  So $St^k_{\T}(x)$ and $Sh^k_{\T}(x)$ denote the $k^{th}$ closed star and the $k^{th}$ shell of $x$ with respect to $\T$, respectively.
\end{comment}
%%%%%%%%%%%%%%%%%%%%%%%%%%%%%%%%%%%%%%%%%%%%%%%%%%%%%%%%%%%%%%%%%%
This first Theorem was proved by Krat in \cite{Krat} for the case when $n=2$, and then for general dimensions by Akopyan in \cite{Akopyan}.

\vskip 10pt

\begin{theorem}[Krat \cite{Krat}, Akopyan \cite{Akopyan}]\label{Akopyan}
Let $\P$ be an $n$-dimensional Euclidean polyhedron and let $\e > 0$.  
Let $f:\P \rightarrow \mathbb{E}^N$ be a 1-Lipschitz map with $N \geq n$.  
Then there exists a pl isometry $h: \P \to \mathbb{E}^N$ which is an $\e$-approximation of $f$.
\end{theorem}

\begin{comment}
\begin{theorem}[Krat \cite{Krat}, Akopyan \cite{Akopyan}]\label{Akopyan}
Let $\P$ be an $n$-dimensional Euclidean polyhedron with vertex set $\V$ and let $\{ \e_i \}_{i = 1}^{\infty}$ be a sequence of positive real numbers converging monotonically to $0$.  Let $f:\X \rightarrow \mathbb{E}^N$ be a short map with $N \geq n$ and fix a vertex $v \in \V$.  Then there exists a pl isometry $h:\X \rightarrow \mathbb{E}^N$ such that for any $k \in \mathbb{N}$ and for any $x \in Sh^k(v)$, $|f(x) - h(x)| < \e_k$.
\end{theorem}
\end{comment}

\vskip 10pt

The following Theorems from \cite{Minemyer1} will be needed for the proofs of the Main Theorem and of Lemma \ref{Key Lemma}. 
Theorem \ref{Minemyer} is essentially the Krat/Akopyan Theorem associated to embeddings and is needed to prove the Main Theorem.  
Theorem \ref{Minemyer1} is very similar to Theorem \ref{Minemyer} and is used to prove Lemma \ref{Key Lemma}.  
The only difference between these two Theorems is that Theorem \ref{Minemyer1} obtains a lower dimensionality for the target Euclidean space at the sacrifice of our isometric embedding no longer being piecewise linear.

\vskip 10pt

\begin{theorem}\label{Minemyer}
Let $\P$ be an $n$-dimensional Euclidean polyhedron and let $\e > 0$.  
Let $f:\X \rightarrow \E^{N}$ be a 1-Lipschitz map with $N \geq 3n$.
Then there exists an piecewise linear isometric embedding $h: \X \rightarrow \mathbb{E}^{N}$ which is an $\e$-approximation of $f$. 
\end{theorem}

\vskip 10pt

\begin{theorem}\label{Minemyer1}
Let $\P$ be an $n$-dimensional Euclidean polyhedron and let $\e > 0$.  
Let $f:\X \rightarrow \E^{N}$ be a 1-Lipschitz map with $N \geq 2n+1$.
Then there exists an (continuous) isometric embedding $h: \X \rightarrow \mathbb{E}^{N}$ which is an $\e$-approximation of $f$. 
\end{theorem}

\vskip 20pt

%beginning of section 3
\section{Step 1:  Constructions of $\{ \D_i \}$, $\{ \G_i \}$, and $\G$}
Sections 4, 5, 6, and 7 are dedicated to proving the Main Theorem.  
In Section 4 we decompose $\D$ into an increasing union of finite sets $\{ \D_i \}$, and construct the set $\G = \bigcup_{i=1}^\infty \G_i$ of ``allowable" geodesics between points of $\D$.  
In this process we also construct collections $\{ \D_i' \}$ and $\{ \G_i' \}$ which will be necessary in the following Sections.

\subsection*{Constructing $\{ \D_i \}$}
In the statement of the Main Theorem, we are given an arbitrary dense countable subset $\D \subset \X$.  
The reason that $\D$ is required to be countable is so that we can write it as an increasing union of finite subsets $\D_i$.  
One simply constructs the collection $\{ \D_i \}$ recursively by defining $\D_1$ to contain a finite number of distinct points of $\D$, and then constructs $\D_{i+1}$ from $\D_i$ by adding a single new point from $\D$.

%In the direction of constructing the collection $\{ \G_i \}$, let us first construct a set of geodesics $\G_i'$.  
\subsection*{Simultaneously constructing $\{ \G_i \}, \{ \D_i' \}, \{ \G_i' \}, \G, \D',$ and $\G'$}
The construction is recursive.  
We begin with $i=1$, and then show how to move from stage $i-1$ to stage $i$.

Let $\{ x_i, x_j \} \in \D_1$ and fix a geodesic $\g_{ij}: [0,d(x_i,x_j)] \to \X$ joining $x_i$ to $x_j$.  
Insert the geodesic $\g_{ij}$ into $\G_1$, and repeat this process for each pair of points in $\D_1$.
So notice that, rigorously, the elements of $\G_i$ for general $i$ will be functions from intervals into $\X$. 
In the remainder of the paper, the term {\it geodesic} will refer to a function whereas the term {\it geodesic segment} will refer to the image in $\X$ of a geodesic. 
Also notice that $\g_{ij}$ will have velocity $v_{ij} = 1$ at every point in its image.  
In our construction of $\G$, {\it every} geodesic will similarly be parameterized by arc-length.

Two geodesics $\g, \l \in \G_1$ have an \emph{allowable intersection} if they either have an empty intersection, they intersect in a unique point, or they intersect on an interval (that is, if there exist intervals $I \subseteq \text{domain}(\g)$ and $J \subseteq \text{domain}(\l)$ such that $\g (I) = \l (J)$). 
The intersection of $\g$ and $\l$ is {\it not} allowable if they intersect in a discrete set of points with cardinality greater than one. 
If two geodesics $\g, \l \in \G_1$ have a non-allowable intersection, then we will replace one of them (say $\l$) with a new geodesic $\bar{\l}$ whose intersection with $\g$ is allowable. 
Then since the cardinality of $\G_1$ is finite we will obtain a new set of geodesics (still called $\G_1$) in finitely many steps where the intersection of any two geodesics in $\G_1$ is allowable.
What follows next is a quick proof that this can be done, but this fact is also mentioned on pg. 267 of \cite{BBI}.

Suppose $a, b, c, e \in \D_1$ are such that $a \neq b$ and $c \neq e$, and let $\g_{ab}, \g_{ce}$ be geodesics in $\G_1$ connecting $a$ to $b$ and $c$ to $e$, respectively.  
Suppose that $\g_{ab}$ and $\g_{ce}$ have non-allowable intersection.  
Let $x$ be the ``first" point at which these two geodesics intersect.  
That is, let 
	\[
	\overline{x} = \text{inf} \{ y \in [0, d(a,b)] \; \vline \; \g_{ab}(y) \in \text{image}(\g_{ce}) \}
	\]
and let $x = \g_{ab}(\overline{x})$.
Similarly, let $y$ be the ``last" point at which $\g_{ab}$ and $\g_{ce}$ intersect.  
Then $\g_{ab}$ and $\g_{ce}$ give two different geodesics from $x$ to $y$.  
Change one of the geodesics, say $\g_{ce}$, to obtain a new geodesic $\overline{\g_{ce}}$ as follows.  
Let $\overline{\g_{ce}}$ agree with $\g_{ce}$ from $\g_{ce}(0)$ to $x$, let $\overline{\g_{ce}}$ agree with $\g_{ab}$ from $x$ to $y$, and lastly let $\overline{\g_{ce}}$ agree with $\g_{ce}$ from $y$ to $\g_{ce}(1)$.  
It is easy to check that the path $\overline{\g_{ce}}$ is still a geodesic from $c$ to $e$ and $\overline{\g_{ce}}$ intersects $\g_{ab}$ on an interval, which is an allowable intersection.  
Continuing this way we ``fix" $\G_1$ until all intersections are allowable (see Figure \ref{thirteenthfig}).

As an intermediary step to construct $\G_1'$, we use $\D_1$ and $\G_1$ to construct a new set of points $\D_1'$ as follows.
We first start with $\D_1 \subseteq \D_1'$.  
If a pair of geodesic segments from $\G_1$ do not intersect, then that pair contributes no new points to $\D_1'$.  
If a pair of geodesic segments from $\G_1$ intersects in a unique point within their interior, then add that intersection point to $\D_1'$.  
If two geodesic segments from $\G_1$ intersect on an interval, then add the two \emph{endpoints} of that interval to $\D_1'$.  
We should note here that since $\G_1$ is finite and since any pair of geodesics from $\G_1$ contributes at most two new points to $\D_1'$, the set $\D_1'$ is finite.

The set $\D_1'$ consists of all of the points in $\X$ that ``look like" they belong to $\D_1$.  
What we mean is that $\D_1'$ consists of all of the points of $\X$ that have multiple geodesic segments from $\G_1$ emanating from them.
Now, the set $\G_1'$ is geometrically the same as $\G_1$.  
To construct the geodesics in $\G_1'$, all we do is subdivide each geodesic of $\G_1$ at every new point of $\D_1'$ which it intersects (and the domain of each geodesic in $\G_1'$ is simply the restriction of the domain in $\G_1$).  
Then we add these geodesics to $\G_1'$.  
So as sets, there is really no relationship between $\G_1$ and $\G_1'$.  
But the union of the geodesic segments of the two sets coincide.
See Figure \ref{fourteenthfig} for a picture of obtaining $\D_1'$ and $\G_1'$ from $\D_1$ and $\G_1$.

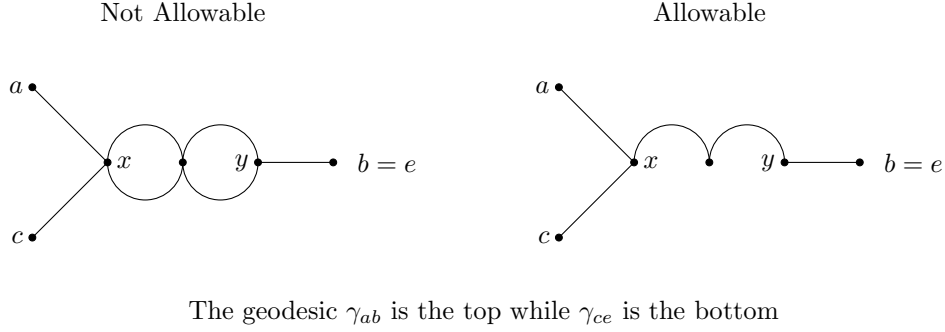
\begin{figure}
\begin{center}
\begin{tikzpicture}

\draw (-1, 1) -- (0,0);
\draw (-1, -1) -- (0,0);
\draw (1,0) arc (0:180:0.5cm);
\draw (1,0) arc (0:-180:0.5cm);
\draw (2,0) arc (0:180:0.5cm);
\draw (2,0) arc (0:-180:0.5cm);
\draw (2, 0) -- (3,0);
\draw[fill=black!] (-1,1) circle (0.3ex);
\draw[fill=black!] (-1,-1) circle (0.3ex);
\draw[fill=black!] (3,0) circle (0.3ex);
\draw[fill=black!] (0,0) circle (0.3ex);
\draw[fill=black!] (1,0) circle (0.3ex);
\draw[fill=black!] (2,0) circle (0.3ex);
\draw (0,0)node[right]{$x$};
\draw (2,0)node[left]{$y$};
\draw (-1,1)node[left]{$a$};
\draw (-1,-1)node[left]{$c$};
\draw (3.2,0)node[right]{$b=e$};
\draw (1,2)node{Not Allowable};
\draw (5,-2)node{The geodesic $\g_{ab}$ is the top while $\g_{ce}$ is the bottom};

\draw (6, 1) -- (7,0);
\draw (6, -1) -- (7,0);
\draw (8,0) arc (0:180:0.5cm);
\draw (9,0) arc (0:180:0.5cm);
\draw (9, 0) -- (10,0);
\draw[fill=black!] (6,1) circle (0.3ex);
\draw[fill=black!] (6,-1) circle (0.3ex);
\draw[fill=black!] (10,0) circle (0.3ex);
\draw[fill=black!] (7,0) circle (0.3ex);
\draw[fill=black!] (8,0) circle (0.3ex);
\draw[fill=black!] (9,0) circle (0.3ex);
\draw (7,0)node[right]{$x$};
\draw (9,0)node[left]{$y$};
\draw (6,1)node[left]{$a$};
\draw (6,-1)node[left]{$c$};
\draw (10.2,0)node[right]{$b=e$};
\draw (8,2)node{Allowable};

\end{tikzpicture}
\end{center}
\caption{An example of altering geodesics to make them ``allowable".}
\label{thirteenthfig}
\end{figure}

Now continuing recursively, suppose that we have constructed $\G_{i-1}$, $\D_{i-1}'$, and $\G_{i-1}'$.  
We construct $\G_i$ from $\G_{i-1}$ by adding in geodesics between the singleton in $\D_i \setminus \D_{i-1}$ and the points of $\D_{i-1}$. 
Now suppose that two geodesics in this set $\G_i$ have non-allowable intersection.  
Since all geodesics in $\G_{i-1}$ have allowable intersection, at least one of these geodesics must be ``new".  
We {\it always} alter new geodesics to make all of the geodesics in $\G_i$ allowable.  
In this way, we truly have $\G_{i-1} \subset \G_i$.  
From here, we construct $\D_i'$ and $\G_i'$ exactly as in the case when $i=1$.  
Notice that $\D_{i-1}' \subset \D_i'$ and, when considered as sets of geodesic segments, $\G_{i-1}' \subset \G_i'$.

Lastly, we construct $\G$, $\D'$, and $\G'$ by
	\[
	\G = \bigcup_{i=1}^\infty \G_i	\qquad	\D' = \bigcup_{i=1}^\infty \D_i'	\qquad	\G' = \bigcup_{i=1}^\infty \G_i'
	\]

\vskip 20pt

%beginning of section 4:  Proof of Lemma \ref{Key Lemma}
\section{Step 2:  Proof of Lemma \ref{Key Lemma}}\label{Proof of Key Lemma}
In Section 5 we prove Lemma \ref{Key Lemma}, which was stated in the Introduction.  
The proof of the Main Theorem involves a recursive construction, and the proof of Lemma \ref{Key Lemma} is essentially the base case of this recursive construction.
So to make things more clear in Section 6 (where we discuss the recursive step in this construction) we will include a subscript ``1" in all quantities which carry over to the proof of the Main Theorem (except for $\d$, which will have a subscript of ``0").
%Just as in the proof of the Main Theorem, we will assume that $\X$ is compact.  
%The arguments in Subsection 6.4 as to why the Main Theorem goes through for proper spaces will apply here as well.
%In Section 5 we will discuss the recursive step in this construction.

\begin{comment}
\begin{lemma}\label{Key Lemma}
Let $(\X, d)$ be a geodesic space with covering dimension $n$, let $F \subseteq \X$ be any finite point subset of $\X$, and let $\d > 0$.  
Then there exists a map $f: \X \rightarrow \mathbb{E}^{2n + 5}$ which satisfies:
	\begin{enumerate}
	\item The map $f$ is $\d$-close to being an embedding.  That is, $f$ satisfies
	\[
	f(x) = f(x') \qquad \Longrightarrow \qquad d(x, x') < \d 
	\]
	\item The map $f$ is an isometry when restricted to $F$.  That is, $\text{pull}_{f}(x, x') = d(x, x')$ for all $x, x' \in F$.
	\end{enumerate}
\end{lemma}
\end{comment}

\vskip 20pt

\subsection*{Outline of the proof of Lemma \ref{Key Lemma}}
%We will first construct a collection of \emph{allowable} geodesics $\G$ between the points of $F$.  
Given a finite subset $\D_1 \subset \X$, we construct $\G_1$, $F_1'$, and $\G_1'$ in exactly the same way as we did in Section 4.
We will then construct two collections of open sets $\mathscr{U}_1$ and $\mathscr{V}_1$ which will depend on positive parameters $\a_1$ and $\b_1$, respectively.  
The purpose of $\scrU_1$ is to isolate the points of $\D_1'$, and the purpose of $\scrV_1$ is to separate the geodesic segments contained in $\G_1'$.  
Lastly we construct a collection of open sets $\scrW_1$ which covers the complement of the unions of all of the sets contained in $\scrU_1$ and $\scrV_1$, and which depends on a positive constant $\e_1$.  
Then the collection $\O_1 := \scrU_1 \cup \scrV_1 \cup \scrW_1$ will be an open cover of $\X$ with mesh($\O_1) < \d_0$ and with order($\O_1) \leq n + 3$ (where $\d_0$ is the fixed positive constant in Lemma \ref{Key Lemma}).

Let $\N_1$ denote the nerve of $\O_1$.  
Since order$(\O_1) \leq n + 3$, dim$(\N_1) \leq n + 2$.  
Using a partition of unity we will construct a map $\psi: \X \to \N_1$ which satisfies that if $\psi(x) = \psi(x')$ then $\dx(x, x') < \d_0$.  
We then assign edge lengths to $\N_1$ to turn it into a Euclidean polyhedron.  
By choosing $\b_1$ sufficiently small we can use Theorem \ref{Bridson} (see Section 8) to ensure that $d_{\N_1}(\psi_1(x), \psi_1(y)) = \dx(x, y)$ for all $x, y \in \D_1$.  
Lastly we use Theorem \ref{Minemyer1} to obtain an intrinsic isometric embedding $h_1: \N_1 \rightarrow \mathbb{E}^{2(n + 2) + 1} = \mathbb{E}^{2n+5}$.  
Then the map $f_1:= h_1 \circ \psi_1$ will satisfy the conditions of Lemma \ref{Key Lemma}.

\begin{notation}
For $x \in \X$ and $r>0$ we will use $b(x,r)$ to denote the {\it open ball} about $x$ of radius $r$, and $B(x,r)$ to denote the {\it closed ball} about $x$ of radius $r$.
\end{notation}

\vskip 20pt

%---------------------Subsection 3.3.2------------------------%
\subsection{Construction of the Open Covering $\O_1$}	
To construct $\O_1$ we will construct three collections of open sets $\mathscr{U}_1, \scrV_1,$ and $\scrW_1$ satisfying the following properties:
	\begin{enumerate}
		\item order($\scrU_1 \cup \scrV_1) = 3$
		\item order($\scrW_1$) $\leq n + 1$
		\item $\text{mesh}(\scrU_1) \leq \a_1, \, \text{mesh}(\scrV_1) \leq \b_1, \, \text{and } \text{mesh}(\scrW_1) \leq \e_1$ where $0 < \e_1 < \b_1 < \a_1 < \frac{\d_0}{3}$
		\item $\scrU_1 \cup \scrV_1 \cup \scrW_1$ is an open cover of $\X$.
	\end{enumerate}
Then we define $\O_1$ := $\scrU_1 \cup \scrV_1 \cup \scrW_1$.  
So we see immediately that $\O_1$ is an open cover of $\X$ with $\text{mesh}(\O_1) < \d_0$, and we will show that order($\O_1) \leq n+3$ (the fact that order($\O_1) \leq n+4$ is clear from conditions (1) and (2) above).  
We break down the construction of $\O_1$ into three parts.

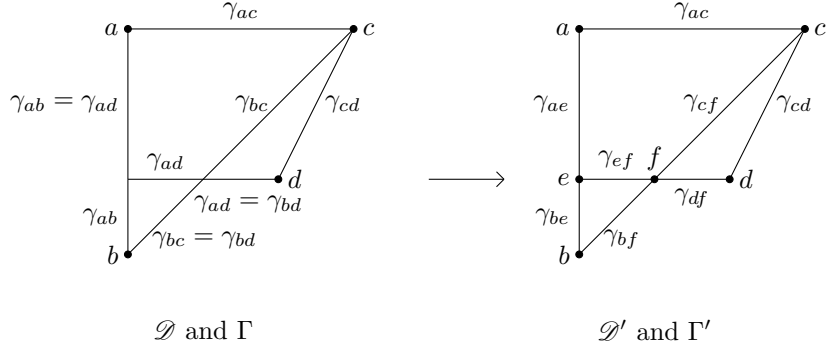
\begin{figure}		
\begin{center}
\begin{tikzpicture}

\draw (0,2) -- (0,0);
\draw (0,-1) -- (0,0);
\draw (0,2) -- (3,2);
\draw (2,0) -- (3,2);
\draw (0,-1) -- (3,2);
\draw (2,0) -- (0,0);
\draw (0,2)node[left]{$a$};
\draw (0,-1)node[left]{$b$};
\draw (3,2)node[right]{$c$};
\draw (2,0)node[right]{$d$};
\draw[fill=black!] (0,2) circle (0.3ex);
\draw[fill=black!] (3,2) circle (0.3ex);
\draw[fill=black!] (0,-1) circle (0.3ex);
\draw[fill=black!] (2,0) circle (0.3ex);
\draw (2.5,1)node[right]{$\g_{cd}$};
\draw (1.5,2)node[above]{$\g_{ac}$};
\draw (0,1)node[left]{$\g_{ab} = \g_{ad}$};
\draw (0,-0.5)node[left]{$\g_{ab}$};
\draw (0.5,0)node[above]{$\g_{ad}$};
\draw (1.6,-0.1)node[below]{$\g_{ad} = \g_{bd}$};
\draw (0.2,-0.8)node[right]{$\g_{bc} = \g_{bd}$};
\draw (2,1)node[left]{$\g_{bc}$};
\draw (1,-2)node{$\D$ and $\G$};

\draw (4,0) -- (5,0);
\draw (4.9,0.1) -- (5,0);
\draw (4.9,-0.1) -- (5,0);

\draw (6,2) -- (6,0);
\draw (6,-1) -- (6,0);
\draw (6,2) -- (9,2);
\draw (8,0) -- (9,2);
\draw (6,-1) -- (9,2);
\draw (8,0) -- (6,0);
\draw (6,2)node[left]{$a$};
\draw (6,-1)node[left]{$b$};
\draw (9,2)node[right]{$c$};
\draw (8,0)node[right]{$d$};
\draw (6,0)node[left]{$e$};
\draw (7,0)node[above]{$f$};
\draw[fill=black!] (6,2) circle (0.3ex);
\draw[fill=black!] (9,2) circle (0.3ex);
\draw[fill=black!] (6,-1) circle (0.3ex);
\draw[fill=black!] (8,0) circle (0.3ex);
\draw[fill=black!] (6,0) circle (0.3ex);
\draw[fill=black!] (7,0) circle (0.3ex);
\draw (8.5,1)node[right]{$\g_{cd}$};
\draw (7.5,2)node[above]{$\g_{ac}$};
\draw (6,1)node[left]{$\g_{ae}$};
\draw (6,-0.5)node[left]{$\g_{be}$};
\draw (6.5,0)node[above]{$\g_{ef}$};
\draw (7.5,0)node[below]{$\g_{df}$};
\draw (6.2,-0.8)node[right]{$\g_{bf}$};
\draw (8,1)node[left]{$\g_{cf}$};
\draw (7,-2)node{$\D'$ and $\G'$};

\end{tikzpicture}
\end{center}
\caption{Constructing $\D'$ and $\G'$ from $\D$ and $\G$.}
\label{fourteenthfig}
\end{figure}

\vskip 20pt
	
\subsection*{The Construction of $\scrU_1$}
For each $x \in \D_1'$ define $\G_x$ to be the set of all geodesic segments from $\G_1'$ which do \emph{not} intersect the point $x$.  
So $\G_x$ does not contain any of the geodesic segments of $\G_1'$ which emanate from $x$.  

For all $x \in \D_1'$ define
	\begin{equation*}
	\eta_x = \inf_{\g \in \G_x} d (x, \text{im}(\g) )
	\end{equation*}
Note that since $\G_x$ is finite and since $d(x, \text{im}(\g) ) > 0$ for all $\g \in \G_x$, we have that $\eta_x > 0$ for all $x \in \D_1'$.  
Then define
	\begin{equation*}
	\a_1 = \frac{1}{2} \inf_{x \in \D_1'} \left\{ \eta_x, \, \frac{\d_0}{3} \right\} \qquad \text{and} \qquad \a_1' = \frac{1}{2} \a_1.
	\end{equation*}
It is clear that for all $x \in \D_1'$, $B(x, \a_1' ) \cap \g = \emptyset$ for all $\g \in \G_x$.  
We then define
	\begin{equation*}
	\scrU_1 = \left\{ b(x, \a_1' ) \right\}_{x \in \D_1'}
	\end{equation*}
Note that $\text{mesh}(\scrU_1) = \a_1 < \frac{\d_0}{3}$ and order($\scrU_1$) = 1.  
%That fact that the mesh is equal to $\a$ is clear, and the fact that the order is one is due to the definition of $\eta_x$.

\vskip 20pt
	
\subsection*{The Construction of $\scrV_1$}
By construction, there is a geodesic in $\G_1$ connecting every pair of points in $\D_1$.  
The images of some of those geodesics may intersect various other geodesic segments of $\G_1$.  
But \emph{everywhere} that two distinct geodesic segments from $\G_1$ meet there exists a point from $\D_1'$ as their point of intersection (this is how we defined $\D_1'$).  
So if we consider the union of all of the geodesic segments from $\G_1$ and delete the parts of these geodesic segments that lie inside of some open set in $\scrU_1$, then what we have left will be a finite disjoint collection of geodesic segments in $\X$.  
That is, the subset of $\X$
	\begin{equation*}
	\Y_1 := \bigcup_{\g \in \G_1} \text{im}(\g) \; \setminus \; \bigcup_{U \in \scrU_1} U 
	\end{equation*}
consists of a finite disjoint collection of geodesic segments in $\X$ (see Figure \ref{fifteenthfig}).  
%For what follows we will consider $\Y_1$ not as a subset of $\X$ but rather as a collection of disjoint geodesics.

We now construct $\scrV_1$ in a very similar way as to how we constructed $\scrU_1$.  
For each geodesic segment $\g \in \Y_1$ define
	\begin{equation*}
	\mu_{\g} := \inf_{\l \in \Y_1 , \, \l \neq \g} d(\g, \l).
	\end{equation*}
Since $\Y_1$ is a disjoint finite set and since each geodesic segment is compact we see that $\mu_{\g} > 0$ for every $\g \in \Y_1$.  
So we then define
	\begin{equation*}
	\b_1 := \frac{1}{3} \inf_{\g \in \Y_1} \left\{ \mu_{\g}, \, \a_1 \right\}	\qquad	\text{and}	\qquad	 \b_1' := \frac{1}{3} \b_1.
	\end{equation*}

	\begin{figure}
\begin{center}
\begin{tikzpicture}

\draw (0,2) -- (0,0);
\draw (0,-1) -- (0,0);
\draw (0,2) -- (6,2);
\draw (4,0) -- (6,2);
\draw (0,-1) -- (6,2);
\draw (4,0) -- (0,0);
\draw (0,2)node[left]{$a$};
\draw (0,-1)node[left]{$b$};
\draw (6,2)node[right]{$c$};
\draw (4,0)node[right]{$d$};
\draw (0,0)node[left]{$e$};
\draw (2,0)node[above]{$f$};
\draw[fill=black!] (0,2) circle (0.3ex);
\draw[fill=black!] (6,2) circle (0.3ex);
\draw[fill=black!] (0,-1) circle (0.3ex);
\draw[fill=black!] (4,0) circle (0.3ex);
\draw[fill=black!] (0,0) circle (0.3ex);
\draw[fill=black!] (2,0) circle (0.3ex);

\draw[dotted] (0.25,2) arc (0:360:0.25cm);
\draw[dotted] (0.25,0) arc (0:360:0.25cm);
\draw[dotted] (0.25,-1) arc (0:360:0.25cm);
\draw[dotted] (2.25,0) arc (0:360:0.25cm);
\draw[dotted] (6.25,2) arc (0:360:0.25cm);
\draw[dotted] (4.25,0) arc (0:360:0.25cm);

\end{tikzpicture}
\end{center}
\caption{The collection $\scrU_1$.}
\label{fifteenthfig}
\end{figure}

Now let $\g \in \Y_1$ be arbitrary.  
Subdivide $\g$ into equidistant subintervals, each of which is of length less than $\b_1'$.  
Let us say that there are $m_{\g}$ subintervals of equal length $\ell_\g$.  
For each subinterval $\l_\g$ of $\g$, let $V_{\l} = b (\l_\g, \frac{1}{2} \ell_\g )$ and let $\scrV_{\g} = \{ V_{\l} \, | \, \l \text{ is a subinterval of } \g \}$.

There are three things to note.  
The first is that since we used $\frac{1}{2} \ell_\g$ as the radius, any $V_{\l}$ intersects only the neighborhoods of the subintervals adjacent to $\l$ (although the boundaries of the $V_\l$'s corresponding to subintervals with exactly one subinterval between them will intersect in that subintervals midpoint).  
So order$(\scrV_{\g}) = 3$, and the only points of order exactly $3$ (as opposed to $\leq 2$) are the midpoints of each subinterval $\l_\g$.  
Secondly, 
	\begin{equation*}
	\text{diam}(V_\l) \leq \frac{1}{2} \ell_{\g} + \ell_{\g} + \frac{1}{2} \ell_{\g} = 2 \ell_{\g} \leq 2 \b_1' = \frac{2}{3} \b_1 < \b_1.
	\end{equation*}  
Thus, mesh($\scrV_{\g}$) $< \b_1 < \a_1 < \d_0 $.  
Finally since diam($V_{\l}$) $< \b_1 \leq \frac{1}{3} \mu_{\g}$, $V_{\l}$ does not intersect any open sets of $\scrV_{\g'}$ for any different geodesic segment $\g' \in \Y_1$.

So we define 
	\begin{equation*}
	\scrV_1 = \bigcup_{V_{\l} \in \scrV_{\g}, \; \g \in \Y_1} V_{\l}. 
	\end{equation*}
By construction, mesh($\scrV_1$) $< \b_1 < \a_1$.  
Also, order($\scrV_1$) $= 3$.  
All that is left to show is that order($\scrU_1 \cup \scrV_1$) $= 3$.  
But this is pretty obvious.  
Any open set $U \in \scrU_1$ will only intersect the neighborhood of the initial segment of each geodesic emanating from the center of $U$.  
But none of these open sets associated with the initial segment of a geodesic in $\Y_1$ can intersect because they are associated with different geodesics (see Figure \ref{sixteenthfig}).  Thus no point of $\X$ can lie in an open set of $\scrU_1$ and in more than one open set of $\scrV_1$.  
And clearly the order of $\scrU_1 \cup \scrV_1$ must be at least three since order($\scrV_1) = 3$.

\vskip 20pt

\subsection*{The Construction of $\scrW_1$}
\begin{comment}
So far we have went to great detail to define open sets around the points or $\D_1'$ ($\scrU_1$) and the geodesics of $\G_1$ ($\scrV_1$).  
While doing this we have kept the order or $\scrU_1 \cup \scrV_1$ as small as possible.  
This will minimize the dimension of the Euclidean space that we will evenutally map $\X$ into.  
But the main reason that we went into such detail instead of just taking an arbitrary open cover as $\O_1$ is that we want to ``protect" the geodesics in $\G_1$.  
More precisely, if we let
	\begin{equation*}
	\mathcal{Z}_1 = \X \setminus \left( \bigcup_{U \in \scrU_1} U \cup \bigcup_{V \in \scrV_1} V   \right)
	\end{equation*}
then for $\O_1$ to be an open cover of $\X$ we need $\scrW_1$ to cover $\mathcal{Z}_1$.  
And when constructing $\scrW_1$ we will do this.  
But what will be most important is that \emph{none of the open sets of $\scrW_1$ meet any of the geodesics of $\G_1$}.  
So the purpose of $\scrU_1$ and $\scrV_1$ was to protect the geodesics in $\G_1$ from the open sets in $\scrW_1$.
\end{comment}

%So to that end, let us construct $\scrW_1$.  
Let
	\begin{equation*}
	\mathcal{Z}_1 = \X \setminus \left( \bigcup_{U \in \scrU_1} U \cup \bigcup_{V \in \scrV_1} V   \right).
	\end{equation*}
Then for $\O_1$ to be an open cover of $\X$ we need $\scrW_1$ to cover $\mathcal{Z}_1$.  
The set $\mathcal{Z}_1$ defined above is a closed subset of $\X$.  %, and is therefore compact.  
So when given the subspace topology, $\calZ_1$ will have covering dimension $ \leq n$.  
Choose some positive number $\e_1 < \b_1$ and let $\overline{\scrW}_1$ be a countable arbitrary (relatively) open covering of $\calZ_1$ with mesh($\overline{\scrW}_1$) $< \e_1/2$ and order $n + 1$ (which can be chosen to be countable since $\X$ is proper).  
We also require that $\overline{\scrW}_1$ be \emph{minimal}, meaning that every set in $\overline{\scrW}_1$ contains at least one point of $\X$ that is not contained in any other element of $\overline{\scrW}_1$.  

Note that while the elements of $\overline{\scrW}_1$ are open in $\calZ_1$ they are not, in general, open in $\X$.  
Members of $\overline{\scrW}_1$ which are contained entirely within the interior of $\calZ_1$ are open in $\X$, but members of $\overline{\scrW}_1$ which meet the boundary of $\calZ_1$ are not open in $\X$.

Every element of $\scrW_1$ will either be an element of $\overline{\scrW}_1$ or an ``enlarged" element of $\overline{\scrW}_1$ so that, in particular, every set in $\overline{\scrW}_1$ is contained in a member of $\scrW_1$.  
We will (possibly) modify each set of $\overline{\scrW}_1$ until we have obtained an open (in $\X$) cover of $\calZ_1$ that meets our specifications.

\begin{figure}
\begin{center}
\begin{tikzpicture}

\draw (0,0) -- (10,0);
\draw[fill=black!] (0,0) circle (0.3ex);
\draw (0,0)node[left]{$a$};
\draw[fill=black!] (10,0) circle (0.3ex);
\draw (10,0)node[right]{$b$};
\draw (1,0) arc (0:360:1cm);
\draw (11,0) arc (0:360:1cm);
\draw[dotted] (1,0.5) -- (9,0.5);
\draw[dotted] (1,-0.5) -- (9,-0.5);
\draw[dotted] (1,0.5) arc (90:270:0.5cm);
\draw[dotted] (2,0.5) arc (90:270:0.5cm);
\draw[dotted] (3,0.5) arc (90:270:0.5cm);
\draw[dotted] (4,0.5) arc (90:270:0.5cm);
\draw[dotted] (5,0.5) arc (90:270:0.5cm);
\draw[dotted] (6,0.5) arc (90:270:0.5cm);
\draw[dotted] (7,0.5) arc (90:270:0.5cm);
\draw[dotted] (8,0.5) arc (90:270:0.5cm);
\draw[dotted] (9,0.5) arc (90:270:0.5cm);
\draw[dotted] (9,-0.5) arc (-90:90:0.5cm);
\draw[dotted] (8,-0.5) arc (-90:90:0.5cm);
\draw[dotted] (6,-0.5) arc (-90:90:0.5cm);
\draw[dotted] (4,-0.5) arc (-90:90:0.5cm);
\draw[dotted] (2,-0.5) arc (-90:90:0.5cm);
\draw[dotted] (3,-0.5) arc (-90:90:0.5cm);
\draw[dotted] (5,-0.5) arc (-90:90:0.5cm);
\draw[dotted] (7,-0.5) arc (-90:90:0.5cm);
\draw[dotted] (1,-0.5) arc (-90:90:0.5cm);
\draw (1,0.1) -- (1,-0.1);
\draw (2,0.1) -- (2,-0.1);
\draw (3,0.1) -- (3,-0.1);
\draw (4,0.1) -- (4,-0.1);
\draw (5,0.1) -- (5,-0.1);
\draw (6,0.1) -- (6,-0.1);
\draw (7,0.1) -- (7,-0.1);
\draw (8,0.1) -- (8,-0.1);
\draw (9,0.1) -- (9,-0.1);
\draw (10,0.1) -- (10,-0.1);
\draw (5,-1)node{$\g$};
\draw (0,-1)node[below]{$U_a$};
\draw (10,-1)node[below]{$U_b$};

\draw (0,0) -- (0,2);
\draw (10,0) -- (10,2);
%\draw[dotted] (-0.5,1) arc (180:360:0.5cm);
%\draw[dotted] (9.5,1) arc (180:360:0.5cm);
%\draw[dotted] (-0.5,1) -- (-0.5,2);
%\draw[dotted] (0.5,1) -- (0.5,2);
%\draw[dotted] (9.5,1) -- (9.5,2);
%\draw[dotted] (10.5,1) -- (10.5,2);

\end{tikzpicture}
\end{center}
\caption{The collection $\scrV_{\g}$.}
\label{sixteenthfig}
\end{figure}

If $\overline{W} \in \overline{\scrW}_1$ is contained entirely within the interior of $\calZ_1$ then we let $W = \overline{W}$ and we do not change the set at all.  
Otherwise $\overline{W}$ has nontrivial intersection with the boundary of $\calZ_1$.  
In this case we need to define an open set $W$ so that $\overline{W} \subset W$, $W$ does not intersect any geodesic segment from $\G_1$, and the resulting collection $\overline{\scrW}_1$, with $\overline{W}$ replaced by $W$, still has order $n + 1$ and mesh less than $\e$.  
For each point $x \in \overline{W} \cap \partial \calZ_1$ choose a small positive number $\e_x$ satisfying:
	\begin{enumerate}
	\item $\ds{b(x, \e_x )}$ does not intersect any geodesic in $\G_1$.  This is possible since $x \in \calZ_1$.
	\item $\ds{b(x, \e_x )}$ intersects at most $n + 1$ members of $\overline{\scrW}_1$.  This is possible since order$(\overline{\scrW}_1) = n + 1$.
	\item $\ds{ \e_x < \frac{1}{2} \left( \e_1 - \text{diam}(\overline{W}_1) \right) }$.  This is possible since diam($\overline{W}$) $< \e_1$.  
	\end{enumerate} 
	
Then we define
	\begin{equation*}
	W = \mathring{W} \cup \bigcup_{x \in \overline{W} \cap \partial \calZ_1} (b(x, \e_x ))
	\end{equation*}
where $\mathring{W}$ denotes the interior of $\overline{W}$.

As the union of open sets, $W$ is clearly open.  
The collection $W$ does not intersect any geodesic segment from $\G_1$ by item (1) above.  
We also know that diam($W$) $< \e_1$ by item (3) above.  
And by item (2), if we replace $\overline{W}$ with $W$ in $\overline{\scrW}_1$, then $\overline{\scrW}_1$ still has order $n + 1$.  
So we just continue this process until we have altered every set in $\overline{\scrW}_1$, leaving us with a new set $\scrW_1$ that is an open cover of $\calZ_1$, does not have any elements who intersect any geodesic segment from $\G_1$, has mesh less than $\e_1$, and has order $n + 1$.  
This process can be iterated since we began with a countable set.
%This completes the construction of $\scrW$.

\vskip 20pt
	
%---------------------Subsection 3.3.3------------------------%
\subsection{Constructing and Metrizing the Nerve $\N_1$ of $\O_1$}
Let $\N_1$ denote the nerve of $\O_1$.  
We want to first argue that order($\O_1) \leq n + 3$, which then would imply that dim($\N_1) \leq n + 2$.  
Note that since order($\scrW_1) = n+1$ and order($\scrU_1 \cup \scrV_1) = 3$, we automatically have that order($\O_1) \leq n+4$.  
So we just need to trim $1$ off of this bound.

The key is to inspect $\scrU_1 \cup \scrV_1$.  
The only points of $\scrU_1 \cup \scrV_1$ which have order three are the midpoints of the subintervals $\{ \l_\g \}_{\g \in \G_1}$.
Without these points, $\scrU_1 \cup \scrV_1$ would have order two and then $\O_1 = \scrU_1 \cup \scrV_1 \cup \scrW_1$ would have order $\leq n+3$.  
But the midpoint of $\l_\g$ lies on the geodesic $\l \in \G_1$, and therefore this point does not lie in the closure of any member of $\scrW_1$ (for $\e_1$ chosen sufficiently small).  
So these points do not increase the order of $\O_1$ and therefore order($\O_1) \leq n+3$.  
Hence, dim($\N_1) \leq n+2$.  

The goal for the remainder of this Subsection is to endow $\N_1$ with a piecewise-flat Euclidean metric which satisfies the following.  
By our construction of $\scrU_1$ we know that for all $x \in \D_1$ there exists a unique % \footnote{In fact, by how we constructed $\scrV$ and $\scrW$, $U_x$ is the only element in $\O$ which contains $x$.} 
$U_x \in \scrU_1$ such that $x \in U_x$.  
Let $u_x \in \calN_1$ denote the vertex corresponding to $U_x$.  
We metrize $\calN_1$ so that   
	\begin{equation}\label{eqn:3.3.3.1} 
	\dx(x, y) = \dn(u_x, u_y) \qquad \text{for all } \, x, y \in \D_1 .
	\end{equation}
	
Let $g$ denote the metric that we will put on $\N_1$.  
We need to define $g$ on the edges of $\N_1$, show that this leads to a Euclidean metric, and show that this metric satisfies Equation \eqref{eqn:3.3.3.1} above.

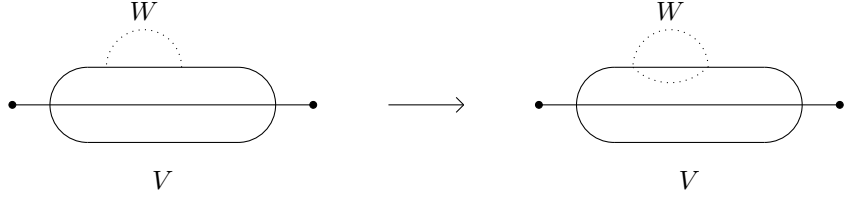
\begin{figure}
\begin{center}
\begin{tikzpicture}

\draw (0,0) -- (4,0);
\draw[fill=black!] (0,0) circle (0.3ex);
\draw[fill=black!] (4,0) circle (0.3ex);
\draw (1,0.5) -- (3,0.5);
\draw (1,-0.5) -- (3,-0.5);
\draw (1,0.5) arc (90:270:0.5cm);
\draw (3,-0.5) arc (-90:90:0.5cm);
\draw (2,-1)node{$V$};
\draw[dotted] (2.25,0.5) arc (0:180:0.5cm);
\draw (1.75,1)node[above]{$\overline{W}$};

\draw (5,0) -- (6,0);
\draw (5.9,0.1) -- (6,0);
\draw (5.9,-0.1) -- (6,0);

\draw (7,0) -- (11,0);
\draw[fill=black!] (7,0) circle (0.3ex);
\draw[fill=black!] (11,0) circle (0.3ex);
\draw (8,0.5) -- (10,0.5);
\draw (8,-0.5) -- (10,-0.5);
\draw (8,0.5) arc (90:270:0.5cm);
\draw (10,-0.5) arc (-90:90:0.5cm);
\draw (9,-1)node{$V$};
\draw[dotted] (9.25,0.5) arc (0:180:0.5cm);
\draw (8.75,1)node[above]{$W$};
\draw[dotted] (8.25,0.5) arc(-135:-45:0.7cm);
%\draw[dotted] (9.25,0.5) -- (9.25,0.25);
%\draw[dotted] (8.25,0.25) -- (9.25,0.25);
%\draw[dotted] (8.25,0.25) -- (8.25,0.5);

\end{tikzpicture}
\end{center}
\caption{Turning $\overline{W}$ into $W$.}
\label{seventeenthfig}
\end{figure}

\begin{remark}\label{remark about g}
The metric $g$ that we are about to define is convenient for proving Lemma \ref{Key Lemma}, but it is not quite what we need for proving the Main Theorem.  
That is why we name the metric ``$g$" instead of ``$g_1$".  
In Section 6 we will recursively define $g_i$ for all stages $i$, and from here it will be clear how to alter $g$ to obtain $g_1$.  
\end{remark}

Edges of $\N_1$ correspond to open sets of $\O_1$ which have nontrivial intersection.  
So to any edge of $\N_1$ there are two open sets of $\O_1$ associated to it.  
For any edge $e$ of $\N_1$ with one (or both) of its associated open sets coming from $\scrW_1$ we define $g(e) := M_1$, where $M_1$ is a large positive constant which we will define later.  

So now consider an edge $e$ of $\N_1$ which satisfies that neither of its vertices correspond to elements of $\scrW_1$, and let $a$ and $b$ denote its corresponding vertices.  
The vertices $a$ and $b$ correspond to open sets $A, B \in \scrU_1 \cup \scrV_1$.  
%Recall that elements of $\scrU$ are open balls centered at points of $F'$ and elements of $\scrV$ are balls around segments of geodesics from $\G$.  
For any set $S \in \scrU_1 \cup \scrV_1$ let $x_S$ denote the center of $S$ if $S \in \scrU_1$ (so $x_S \in \D_1'$) or let $x_S$ denote the midpoint of the geodesic segment associated to $S$ if $S \in \scrV_1$.  
Define
	\begin{equation*}
	g(e) := \dx (x_A, x_B ).
	\end{equation*}
	
We now show that $g$ leads to a Euclidean metric if we choose $M_1$ sufficiently large and $\e_1$ sufficiently small.  
For $\g \in \G_1'$, let $\ell_\g$ denote the length of the subintervals which correspond to elements of $\scrV_1$ associated to $\g$.  
For $\e_1$ small enough any open set of $\scrW_1$ has non-trivial intersection with \emph{at most} two elements of $\scrU_1 \cup \scrV_1$. 
Also, there is no point in $\X$ which is contained in three of more members of $\scrU_1 \cup \scrV_1$.
Therefore any simplex of $\calN_1$ contains at most two vertices, and thus at most one edge, which does \emph{not} correspond to an open set in $\scrW_1$.

%%%%%%%%%%%%%%%%%%%%%%%%%%%%%%%%%%%%%%%%%%%%%%%%%%%%%%%%%%%%%%%%%%%%%%%%%%%%%%%%%%%%%%
\begin{comment}
First, recall that $\e$ satisfies that mesh($\scrW) \leq \e$.  
The reason why we want to choose $\e$ small is the following.  
Recall that the set $\scrU \cup \scrV$ is finite and has order 3.  
Also, by construction, we have that no set in $\scrU \cup \scrV$ is contained in the union of the remaining sets.  
% The conclusion is ok, but I now need to change this argument.
{\color{red}
So if we let 
	\begin{equation*}
	\mu := \inf_{S \in \scrU \cup \scrV} \text{diam} \left( S \setminus \bigcup_{S' \in \scrU \cup \scrV , S' \neq S} S'  \right)
	\end{equation*}
we see that $\mu > 0$.  
Then if we require that $\e < \mu$ we have that any element of $\scrW$ has non-trivial intersection with \emph{at most} 2 elements of $\scrU \cup \scrV$.  
Therefore any simplex of $\calN$ contains at most 2 vertices, and thus at most 1 edge, which does \emph{not} correspond to an open set in $\scrW$.
}
\end{comment}
%%%%%%%%%%%%%%%%%%%%%%%%%%%%%%%%%%%%%%%%%%%%%%%%%%%%%%%%%%%%%%%%%%%%%%%%%%%%%%%%%%%%%%%%%

So now let $\s \in \calN_1$ be an arbitrary $k$-dimensional simplex.  
If zero or one of the vertices of $\s$ correspond to elements of $\scrU_1 \cup \scrV_1$ then every edge of $\s$ has length $M_1$.  
This clearly defines a Euclidean simplex. 
So now suppose that exactly two of the vertices of $\s$ correspond to elements of $\scrU_1 \cup \scrV_1$.  
Then there exists a unique edge $e'$ of $\s$ which satisfies that $g(e') \neq M_1$ (see Figure \ref{eighteenthfig}).  
But by our definition of $g$ we have that $0 < g(e') \leq \a_1$, where $\a_1$ is from our construction of $\scrU_1$.  
A straightforward calculation (which can be found in \cite{Minemyer Geoghegan}) shows that the simplex $\s$ admits an isometric embedding into $\mathbb{E}^k$ if and only if $0 < g(e') < \sqrt{\frac{2k}{k-1}} M_1$.  
Then letting $M_1 \geq \a_1$ verifies that our metric $g$ is Euclidean.

One final remark here is that, since $\scrU_1 \cup \scrV_1 $ is finite, there are only finitely many edges of $\calN_1$ which are not of length $M_1$.  
Therefore, even if $\X$ is not compact, Shapes($\calN_1$) (the isometry types of simplices of $(\N_1,g)$) is finite.

\begin{figure}
\begin{center}
\begin{tikzpicture}[scale=0.8]

\draw[fill=gray!20] (0,0) -- (3,0) -- (1.5,3) -- (0,0);
\draw (0.75,1.5)node[left]{$M$};
\draw (2.25,1.5)node[right]{$M$};
\draw (1.5,0)node[below]{$M$};

\draw[fill=gray!20] (5,0) -- (8,0) -- (6.5,3) -- (5,0);
\draw (5.75,1.5)node[left]{$M$};
\draw (7.25,1.5)node[right]{$M$};
\draw (6.5,0)node[below]{$g(e')$};

\draw (4,-1)node[below]{$g(e') \leq \a$};

\end{tikzpicture}
\end{center}
\caption{The 2 possibilities for simplices in $\calN_1$.}
\label{eighteenthfig}
\end{figure}
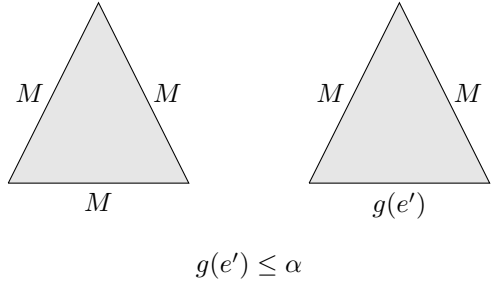

\vskip 20pt

\subsection*{Verifying $g$ Satisfies Equation \eqref{eqn:3.3.3.1}}
The key here is to choose $M_1$ large, apply Bridson's Theorem (Theorem \ref{Bridson}), and then choose $\e_1$ small.  
Recall that we want to show
	\begin{equation*}
	\dx(x, y) = \dn(u_x, u_y) \qquad \text{for all } \, x, y \in \D_1
	\end{equation*}
where for all $x \in \D_1$, $u_x$ denotes the vertex of $\calN_1$ corresponding to the unique element $ U_x \in \scrU_1$ containing $x$.

First recall that for any $x, y \in \D_1$ there exists a geodesic $\g_{x, y} \in \G_1$ joining $x$ to $y$.  
Also recall that $\g_{x, y}$ is the union of geodesic segments from $\G_1'$.  
Corresponding to each geodesic $\g$ in $\G_1'$ is a collection $\scrV_{\g}$ of ``interlinkiing" members of $\scrV_1$ (see Figure \ref{sixteenthfig}).  
These elements can be concatenated to form a chain $C = (v_{a_0}, v_{a_1}, ..., v_{a_{m-1}}, v_{a_m})$ within the 1-skeleton of $\calN_1$ connecting $v_{a_0} = u_x$ to $v_{a_m} = u_y$, where, technically, some of the $a_i$'s may come from members of $\scrU_1$ corresponding to intersection points in $\D_1'$.  
By how we defined our metric $g$ we see that 
	\begin{equation}\label{eqn:  length}
	\ell(C) = \dx (x, a_0) + \sum_{i = 1}^{m-2} \dx (a_i, a_{i+1}) + \dx (a_{m-1}, y) = \dx (x, y).
	\end{equation}  
Thus, since the distance in $\calN_1$ between any two points is defined as the infimum of the lengths of paths connecting those points, we see that
	\begin{equation*}
	\dn(u_x, u_y) \leq \dx (x, y).
	\end{equation*}
	
The reason why we need Theorem \ref{Bridson} is for the reverse inequality. 

To that end let $x, y \in \D_1$, let $\g$ be the unique geodesic in $\G_1$ joining $x$ to $y$, and let $C_{\g}$ be the chain in $\calN_1$ corresponding to $\g$ as described above.  
Consider an arbitrary taut chain $p$ in $\calN_1$ connecting $x$ to $y$.  
We first consider the case that the image of $p$, denoted im$(p)$, is contained in the open star $\text{st}(\text{im}(C_{\g}))$.  
Then we easily see that we can choose $M_1$ large enough so that $\ell(p) \geq \ell(C_{\g})$.  
We could also have that im$(p) \subset \text{st(im}(C_{\l}))$ for other geodesics $\l \in \G_1$.  
But in this case we clearly see that $\ell(p) \geq \ell(C_{\g})$.  

So now consider the case that the path $p$ leaves $\text{st(im}(C_{\g}))$ and is not contained in $\text{st(im}(C_{\l}))$ for other geodesics $\l \in \G_1$.  
We know that $p$ must correspond to a taut chain $C_p$ that leaves $\text{st(im}(C_{\g}))$.  
By choosing $\e_1$ arbitrarily small, we increase the (combinatorial) length of any chain that leaves $\bigcup_{\l \in \G_1} \text{st(im}(C_{\l}))$ while \emph{not} changing Shapes($\calN_1$).  
In essence, by taking $\e_1$ small and $M_1$ large we are ``blowing up" the metric outside of the neighborhoods of the images of the geodesic segments of $\G_1$ in $\calN_1$.  
By Bridson's Theorem \ref{Bridson} the length of the taut chain $C_p$ grows linearly with respect to the number of simplices it intersects.  
Thus if im($C_{p})$ leaves $\bigcup_{\l \in \G_1} \text{st(im}(C_{\l}))$ we can choose $\e_1$ small enough so that $\ell(p) \geq \ell(C_{\g})$.  
Then since $\D_1$ is finite, we can choose a positive $\e_1$ so that equation \eqref{eqn:3.3.3.1} is satisfied for every pair of points in $\D_1$.

\vskip 20pt
	
%---------------------Subsection 3.3.4------------------------%
\subsection{The Maps $\psi$ and $h_1$ and Wrapping Up the Proof of Lemma \ref{Key Lemma}}
Fix a partition of unity $\ds{\{ \phi_i \}_{i \in I}}$ suboordinate to $\O_1$ (where $I$ is an indexing set for $\O_1$).  
Then define $\psi : \X \rightarrow \calN_1 $ by
	\begin{equation*}
	\psi(x) = \sum_{i \in I} \phi_i (x) \mathscr{O}_i
	\end{equation*}
where $\mathscr{O}_i$ is the vertex in $\calN_1$ corresponding to the open set $O_i \in \O_1$, and the sum is a point of $\calN_1$ expressed in barycentric coordinates.  
Using a partition of unity to construct continuous maps from a space to a nerve of a corresponding open cover is a standard construction.  
For more information please see \cite{Nagami} and/or \cite{Nagata}.  
A very important observation is that since every point of any geodesic in $\G_1$ is contained in at most 2 elements of $\O_1$, every geodesic of $\G_1$ is mapped by $\psi$ onto the 1-skeleton of $\calN_1$.  
Moreover, the midpoint of each geodesic segment corresponding to a member of $\scrV_1$ is necessarily mapped onto the vertex of $\N_1$ corresponding to that open set in $\scrV_1$.  
The reason that we used $\scrU_1$ and $\scrV_1$ to ``protect" the geodesics in $\G_1$ was to ensure that they mapped onto the 1-skeleton of $\calN_1$.

By the work done in Subsection 5.2 we know that $\dx (x, y) = \dn (\psi(x), \psi(y)) $ for all $x, y \in \D_1$.  
But also notice that if $\psi(a) = \psi(b)$ for some $a, b \in \X$ then we must have that $a$ and $b$ are both contained in the same members of $\O_1$.  
Otherwise $\psi(a)$ and $\psi(b)$ would reside in different simplices of $\calN_1$.  
Then since mesh($\O_1) < \d_0$ we have that $\dx (a, b) < \d_0$.  
Thus
	\begin{equation*}
	\psi(a) = \psi(b) \qquad \Longrightarrow \qquad \dx (a, b) < \d_0.
	\end{equation*}
Now by Theorem \ref{Minemyer1} there exists an isometric embedding $h_1: \calN_1 \rightarrow \mathbb{E}^{2(n + 2)+1} = \mathbb{E}^{2n+5}$.  
Letting $f := h_1 \circ \psi$ we obtain a map from $\X$ into $\mathbb{E}^{2n+5}$ which satisfies
	\begin{enumerate}
		\item If $f(x) = f(x')$, then $\dx(x, x') < \d_0 $
		
		\item The map $f$ is an isometry when restricted to $\D_1$.  That is, 
		\[
		d_{f(\X)}(f(x), f(x')) = \dx(x, x')
		\]
		for all $x, x' \in \D_1$.
	\end{enumerate}
This completes the proof of Lemma \ref{Key Lemma}.  
	\hfill $\qed$
	
\begin{remark}\label{remark about psi}
Similar to Remark \ref{remark about g} above about the metric $g$, the map $\psi$ defined above is different from what we will need to prove the Main Theorem.  
The actual map $\psi_1$ needed to prove the Main Theorem will be defined with the rest of the collection $(\psi_i)$ in Section 6 below.
\end{remark}

\vskip 20pt

%beginning of Section 5
\section{Step 3:  Recursively constructing the open covers $\{ \O_i \}$, Euclidean polyhedra $\{ (\N_i, g_i) \}$, and the functions $\{ \psi_i \}$, $\{ \v_{i+1,i} \}$, and $\{ h_i \}$}

%\vskip 20pt

\subsection{Construction of $\O_i$}
%Consider all of the constructions in the proof of Lemma \ref{Key Lemma} as the initial constructions in the proof of the Main Theorem.  
%So $\scrU_1 := \scrU$ from Section 3, $\O_i := \O$ from Section 3, and so on.  
In Section 5 we have already defined $\scrU_1$, $\scrV_1$, $\scrW_1$, $\O_1$, and $\N_1$. %$\{ \phi^{1}_{j} \}_{j \in \O_1}$, 
%$\psi_1: \X \to \N_1$, 
%and $h_1: \N_1 \to \E^{3n+6} \subset \R^{3n+6,1}$.  %and $f_1= h_1 \circ \psi_1: \X \to \R^{3n+6}$.  
We also have that mesh($\scrU_1) \leq \a_1$, mesh($\scrV_1) \leq \b_1$, and mesh($\scrW_1) \leq \e_1$ with $0 < \e_1 < \b_1 < \a_1 < \frac{\d_0}{3}$ for some fixed initial $\d_0$.

Now, let $\d_1$ denote the Lebesgue number of $\O_1$.  
Notice that $0 < \d_1 \leq \text{mesh}(\scrW_1) < \e_1$.  
For general $i$ we let $\d_{i-1}$ denote the Lebesgue number of $\O_{i-1}$, and we construct $\scrU_i$, $\scrV_i$, and $\scrW_i$ in such a way that 
	\begin{enumerate}
		\item order($\scrU_i \cup \scrV_i) = 3$
		\item order($\scrW_i$) $\leq n + 1$
		\item $\text{mesh}(\scrU_i) \leq \a_i, \, \text{mesh}(\scrV_i) \leq \b_i, \, \text{and } \text{mesh}(\scrW_i) \leq \e_i$ where $0 < \e_i < \b_i < \a_i < \frac{\d_{i-1}}{3}$
		\item $\O_i := \scrU_i \cup \scrV_i \cup \scrW_i$ is an open cover of $\X$ of order $n+3$ which is a refinement of $\O_{i-1}$.
	\end{enumerate}
The constructions of $\scrU_i$, $\scrV_i$, and $\scrW_i$ are detailed below, and in the case of $\scrU_i$ and $\scrW_i$ the constructions are nearly identical to that of Section 5.  
%Below we simply indicate any modifications that need to be made.
%We note here that, for the same reasons as in Section 3, we will have that order$(\O_i) \leq n+3$.
Also recall that in Section 4 we have already defined $\D_i$, $\G_i$, $\D_i'$, and $\G_i'$.

\vskip 20pt

\subsection*{Construction of $\scrU_i$}
This is, for the most part, identical to the construction of $\scrU_1$ from Section 5.  
At stage $i$ we will require that $\a_i$ is chosen small enough so that no member of $\scrU_i$ contains any ``midpoints" from stage $(i-1)$ (which will make sense after we construct $\scrV_i$ below).
Also, when constructing $\scrV_i$, we will make very small alterations to the sets in $\scrU_i$.  
But otherwise everything is exactly the same.

\vskip 20pt

\subsection*{Construction of $\scrV_i$}
This is the one construction that is a lot different than that in Section 5.  
It is not so much that the construction is different, but we need to ensure that we have a lot of control over certain aspects here for future considerations.

Since we have already constructed $\scrV_1 := \scrV$ in Section 5, we assume that $\scrV_{i-1}$ has already been constructed and use it to construct $\scrV_i$.  
Fix a geodesic $\g \in \G_i$.  
If $\g \nin \G_{i-1}$, then we use the exact same procedure as in Section 5 to produce the associated members of $\scrV_i$ (at least, along any geodesic segment of $\g$ which does not intersect any members of $\G_{i-1}$).

So assume that $\g \in \G_i$ and $\g \in \G_{i-1}$.  
Recall that the image of $\g$ may potentially be subdivided into several geodesic segments contained in $\G_i'$.  
Let $\g_i \in \G_i'$ denote one such geodesic segment of $\g$, and then the construction below can be carried out on each such geodesic segment.
Note that since $\g \in \G_{i-1}$, there exists a geodesic segment $\g_{i-1}$ of $\g$ contained in $\G_{i-1}'$ and containing $\g_i$.
Let 
	\[
	\overline{\g}_{i-1} = \text{im}(\g_{i-1}) \setminus \bigcup_{U \in \scrU_{i-1}} U
	\]
denote the portion of the geodesic segment $\g_{i-1}$ not contained in any open sets in $\scrU_{i-1}$ (so we just truncate off the ends of im($\g_{i-1}$)).  
Recall that, in the construction of $\scrV_{i-1}$ in Section 5, we subdivided $\overline{\g}_{i-1}$ into equidistant subintervals.  
Let us say that there are $N_{i-1}$ of these subintervals, each of length $\xi_{i-1}$.  
Let $x_0, x_1, \hdots, x_{N_{i-1}}$ denote the endpoints of these subintervals, and let $m_1, m_2, \hdots, m_{N_{i-1}}$ denote the midpoints of these subintervals (all as points in $\X$).  %\footnote{In what follows we will frequently abuse notation and consider these points as both elements of $\X$ and as points in the domain of $\g$.  In any case, what is meant will always be clear from context.}.
Then the open sets that $\g_{i-1}$ contributed to $\scrV_{i-1}$ were the sets of the form
	\[
	b\left( [x_{k-1},x_k ], \frac{1}{2} \xi_{i-1} \right)
	\]
where $[x_{k-1},x_k ]$ denotes the geodesic segment of $\g_{i-1}$ from $x_{k-1}$ to $x_k$.  

We now construct the open sets that $\g_i$ contributes to $\scrV_i$.  
What we want to do is to subdivide the geodesic segment
	\[
	\overline{\g}_i := \text{im}(\g_i) \setminus \bigcup_{U \in \scrU_i} U
	\]
into $N_i$ equidistant subintervals of length $\xi_i < < \xi_{i-1}$.  
Let $y_0, y_1, \hdots, y_{N_i}$ denote the endpoints of these subintervals, and let $n_1, n_2, \hdots, n_{N_i}$ denote the midpoints of these subintervals.
There would be no issue with any of this if not for the following.  
We need to ensure that the midpoints $m_1, m_2, \hdots, m_{N_{i-1}}$ of the subintervals of $\overline{\g}_{i-1}$ are still midpoints of the new subintervals corresponding to $\overline{\g}_i$.
What follows is just a detailed explanation showing that this is possible, but the reader who already believes this can consult Figure \ref{subdividing gamma}.

First note that the image of $\g_{i-1}$ may be bigger than that of $\g_i$.  
In what follows we are only going to consider the portion of $\overline{\g}_{i-1}$ which corresponds to $\overline{\g}_i$.  
Also, as mentioned above in the construction of $\scrU_i$, we choose $\a_i$ small enough so that no midpoints of $\overline{\g}_{i-1}$ were lost in the construction of $\overline{\g}_i$ (unless the center of a member of $\scrU_i$ happened to be a midpoint of $\overline{\g}_{i-1}$, but that specific setup will not cause any issues going forward).
%geodesic segments from $\G_{i-1}$ corresponding to sets in $\scrV_{i-1}$ are contained in any members of $\scrU_i$ (this is possible unless a new member of $\D_i'$ is such a midpoint, 

%Now let %\footnote{Note that we are abusing notation in these definitions. 
%Technically, $m_0$ is the point in $\text{im}(\g) \cap b \left( x_0, \frac{1}{2} \xi_{i-1} \right)$ closest to $\g(0)$, and similarly for $m_{N_{i-1}+1}$.
%If this notation is not clear, please see Figure \ref{subdividing gamma}.} 
%$m_0 := \g \left( \g^{-1}(x_0) - \frac{1}{2} \xi_{i-1} \right)$ and let $m_{N_{i-1}+1} := \g \left( \g^{-1}(x_{N_i}) + \frac{1}{2} \xi_{i-1} \right)$.  
Now we define
	\begin{equation}\label{picking K_i}
	\xi_i := \frac{\xi_{i-1}}{K_i}
	\end{equation}
where $K_i$ is a sufficiently large integer to be chosen shortly.  
Then, since $\xi_{i-1}$ is an integer multiple of $\xi_i$, we can subdivide the geodesic segment
	\[
	\left[ \g \left( \g^{-1}(m_1) - \frac{1}{2} \xi_i \right), \g \left( \g^{-1}(m_{N_{i-1}}) + \frac{1}{2} \xi_i \right) \right]
	\]
into equidistant subintervals of length $\xi_i$.  
What is important is to note that each $m_k$, for $1 \leq k \leq N_{i-1}$, will now be a midpoint of one of these subintervals.  
This is due to the fact that $\xi_{i-1}$ is an integer multiple of $\xi_i$, and that we started at $\g \left( \g^{-1}(m_1) - \frac{1}{2} \xi_i \right)$.

Let $a$ and $b$ denote the endpoints of $\g_i$, and   
%i.e., $a = \g(0)$ and $b = \g(1)$.  
let $U_i^a$, and $U_i^b$ denote the neighborhoods about $a$ and $b$ in $\scrU_i$.  
%In general, diam($U_i^a) << \text{diam}(U_{i-1}^a)$, and similarly for $b$.  
Then for $K_i$ chosen sufficiently large we will have that the interval
	\[
	\left[ \g \left( \g^{-1}(m_1) - \frac{1}{2} \xi_i \right), \g \left( \g^{-1}(m_{N_{i-1}}) + \frac{1}{2} \xi_i \right) \right]
	\]
does not cover all of $\overline{\g}_i$.  
In general, there will be ``leftover segments" at each end of $\overline{\g}_i$ which have not yet been partitioned into subintervals.  
But by choosing $K_i$ large, we can make the length of each of these segments as close to an integer multiple of $\xi_i$ as we like.  
So given $\mu_i > 0$ arbitrarily small, we can choose $K_i$ large enough so that the length of each of these end regions is within $\mu_i$ of being an integer multiple of $\xi_i$.  
Consider the points
	\[
	z_a := \text{im}(\g) \cap \partial U_i^a 		\qquad	z_b := \text{im}(\g) \cap \partial U_i^b.
	\]
There exist positive numbers $\mu_a, \mu_b < \mu_i$ such that the ``enlarged $\overline{U}_i$'s" defined by
	\begin{equation}\label{enlarging U}
	\overline{U}_i^a := U_i^a \cup b(z_a,\mu_a) 	\qquad 	\overline{U}_i^b := U_i^b \cup b(z_b,\mu_b)
	\end{equation}
will create end regions whose lengths {\it are} integer multiples of $\xi_i$.  
So we replace $U_i^a$ and $U_i^b$ with $\overline{U}_i^a$ and $\overline{U}_i^b$ in $\scrU_i$.
Note that, since $\D_i'$ is finite, we can choose $\mu_a$ and $\mu_b$ small enough so that the set $\scrU_i$, with $\overline{U}_i^a$ and $\overline{U}_i^b$ inserted in place of $U_i^a$ and $U_i^b$, still satisfies the necessary properties for $\O_i$ listed at the beginning of this Section.
%$\overline{U}_i^a$ and $\overline{U}_i^b$ do not cause any issues with any of the other geodesics in $\G$.  
Then, the neighborhoods that $\g$ contributes to $\scrV_i$ are the $\frac{1}{2} \xi_i$ neighborhoods of each of these subdivisions.

\begin{figure}
\begin{center}
\begin{tikzpicture}

\draw (0,0) -- (10,0);
\draw[fill=black!] (0,0) circle (0.3ex);
\draw (0,0)node[left]{$a=x_0$};
\draw[fill=black!] (10,0) circle (0.3ex);
\draw (10,0)node[right]{$b = x_{10}$};
%\draw (1,0) arc (0:360:1cm);
%\draw (11,0) arc (0:360:1cm);
\begin{comment}
\draw[dotted] (1,0.5) -- (9,0.5);
\draw[dotted] (1,-0.5) -- (9,-0.5);
\draw[dotted] (1,0.5) arc (90:270:0.5cm);
\draw[dotted] (2,0.5) arc (90:270:0.5cm);
\draw[dotted] (3,0.5) arc (90:270:0.5cm);
\draw[dotted] (4,0.5) arc (90:270:0.5cm);
\draw[dotted] (5,0.5) arc (90:270:0.5cm);
\draw[dotted] (6,0.5) arc (90:270:0.5cm);
\draw[dotted] (7,0.5) arc (90:270:0.5cm);
\draw[dotted] (8,0.5) arc (90:270:0.5cm);
\draw[dotted] (9,0.5) arc (90:270:0.5cm);
\draw[dotted] (9,-0.5) arc (-90:90:0.5cm);
\draw[dotted] (8,-0.5) arc (-90:90:0.5cm);
\draw[dotted] (6,-0.5) arc (-90:90:0.5cm);
\draw[dotted] (4,-0.5) arc (-90:90:0.5cm);
\draw[dotted] (2,-0.5) arc (-90:90:0.5cm);
\draw[dotted] (3,-0.5) arc (-90:90:0.5cm);
\draw[dotted] (5,-0.5) arc (-90:90:0.5cm);
\draw[dotted] (7,-0.5) arc (-90:90:0.5cm);
\draw[dotted] (1,-0.5) arc (-90:90:0.5cm);
\end{comment}
\draw (1,0.1) -- (1,-0.1);
\draw (2,0.1) -- (2,-0.1);
\draw (3,0.1) -- (3,-0.1);
\draw (4,0.1) -- (4,-0.1);
\draw (5,0.1) -- (5,-0.1);
\draw (6,0.1) -- (6,-0.1);
\draw (7,0.1) -- (7,-0.1);
\draw (8,0.1) -- (8,-0.1);
\draw (9,0.1) -- (9,-0.1);
%\draw (10,0.1) -- (10,-0.1);
\draw (5,1)node{$\g$};
%\draw (0,-1)node[below]{$U_{i-1}^a$};
%\draw (10.2,-1)node[below]{$U_{i-1}^b$};
\draw (1,0)node[below]{$x_1$};
\draw (2,0)node[below]{$x_2$};
\draw (3,0)node[below]{$x_3$};
\draw (4,0)node[below]{$x_4$};
\draw (5,0)node[below]{$x_5$};
\draw (6,0)node[below]{$x_6$};
\draw (7,0)node[below]{$x_7$};
\draw (8,0)node[below]{$x_8$};
\draw (9,0)node[below]{$x_9$};
\draw[fill=black!] (0.5,0) circle (0.3ex);
\draw[fill=black!] (1.5,0) circle (0.3ex);
\draw[fill=black!] (2.5,0) circle (0.3ex);
\draw[fill=black!] (3.5,0) circle (0.3ex);
\draw[fill=black!] (4.5,0) circle (0.3ex);
\draw[fill=black!] (5.5,0) circle (0.3ex);
\draw[fill=black!] (6.5,0) circle (0.3ex);
\draw[fill=black!] (7.5,0) circle (0.3ex);
\draw[fill=black!] (8.5,0) circle (0.3ex);
\draw[fill=black!] (9.5,0) circle (0.3ex);
\draw (0.5,0)node[above]{$m_1$};
\draw (1.5,0)node[above]{$m_2$};
\draw (2.5,0)node[above]{$m_3$};
\draw (3.5,0)node[above]{$m_4$};
\draw (4.5,0)node[above]{$m_5$};
\draw (5.5,0)node[above]{$m_6$};
\draw (6.5,0)node[above]{$m_7$};
\draw (7.5,0)node[above]{$m_8$};
\draw (8.5,0)node[above]{$m_9$};
\draw (9.5,0)node[above]{$m_{10}$};

\draw[|-|] (4.5,-1.2) -- (5.5,-1.2);
\draw (5,-1.5)node{$\xi_{i-1}$};

\draw (0,-3) -- (10,-3);
\draw[fill=black!] (0,-3) circle (0.2ex);
\draw (-0.1,-3)node[left]{$a$};
\draw[fill=black!] (10,-3) circle (0.2ex);
\draw (10.1,-3)node[right]{$b$};
\draw (0.2,-3) arc (0:360:0.2cm);
\draw (10.2,-3) arc (0:360:0.2cm);
\draw (0,-3.2)node[below]{$U_i^a$};
\draw (10,-3.2)node[below]{$U_i^b$};
\draw[fill=black!] (0.5,-3) circle (0.2ex);
\draw[fill=black!] (1.5,-3) circle (0.2ex);
\draw[fill=black!] (2.5,-3) circle (0.2ex);
\draw[fill=black!] (3.5,-3) circle (0.2ex);
\draw[fill=black!] (4.5,-3) circle (0.2ex);
\draw[fill=black!] (5.5,-3) circle (0.2ex);
\draw[fill=black!] (6.5,-3) circle (0.2ex);
\draw[fill=black!] (7.5,-3) circle (0.2ex);
\draw[fill=black!] (8.5,-3) circle (0.2ex);
\draw[fill=black!] (9.5,-3) circle (0.2ex);
\draw[dotted] (0.5,0) -- (0.5,-3);
\draw[dotted] (1.5,0) -- (1.5,-3);
\draw[dotted] (2.5,0) -- (2.5,-3);
\draw[dotted] (3.5,0) -- (3.5,-3);
\draw[dotted] (4.5,0) -- (4.5,-3);
\draw[dotted] (5.5,0) -- (5.5,-3);
\draw[dotted] (6.5,0) -- (6.5,-3);
\draw[dotted] (7.5,0) -- (7.5,-3);
\draw[dotted] (8.5,0) -- (8.5,-3);
\draw[dotted] (9.5,0) -- (9.5,-3);
\begin{comment}
\draw (0.5,-3)node[below]{$m_0$};
\draw (1.5,-3)node[below]{$m_1$};
\draw (2.5,-3)node[below]{$m_2$};
\draw (3.5,-3)node[below]{$m_3$};
\draw (4.5,-3)node[below]{$m_4$};
\draw (5.5,-3)node[below]{$m_5$};
\draw (6.5,-3)node[below]{$m_6$};
\draw (7.5,-3)node[below]{$m_7$};
\draw (8.5,-3)node[below]{$m_8$};
\draw (9.5,-3)node[below]{$m_9$};
\end{comment}
\draw (0.4,-2.9) -- (0.4,-3.1);
\draw (0.6,-2.9) -- (0.6,-3.1);
\draw (0.8,-2.9) -- (0.8,-3.1);
\draw (1,-2.9) -- (1,-3.1);
\draw (1.2,-2.9) -- (1.2,-3.1);
\draw (1.4,-2.9) -- (1.4,-3.1);
\draw (1.6,-2.9) -- (1.6,-3.1);
\draw (1.8,-2.9) -- (1.8,-3.1);
\draw (2,-2.9) -- (2,-3.1);
\draw (2.2,-2.9) -- (2.2,-3.1);
\draw (2.4,-2.9) -- (2.4,-3.1);
\draw (2.6,-2.9) -- (2.6,-3.1);
\draw (2.8,-2.9) -- (2.8,-3.1);
\draw (3,-2.9) -- (3,-3.1);
\draw (3.2,-2.9) -- (3.2,-3.1);
\draw (3.4,-2.9) -- (3.4,-3.1);
\draw (3.6,-2.9) -- (3.6,-3.1);
\draw (3.8,-2.9) -- (3.8,-3.1);
\draw (4,-2.9) -- (4,-3.1);
\draw (4.2,-2.9) -- (4.2,-3.1);
\draw (4.4,-2.9) -- (4.4,-3.1);
\draw (4.6,-2.9) -- (4.6,-3.1);
\draw (4.8,-2.9) -- (4.8,-3.1);
\draw (5,-2.9) -- (5,-3.1);
\draw (5.2,-2.9) -- (5.2,-3.1);
\draw (5.4,-2.9) -- (5.4,-3.1);
\draw (5.6,-2.9) -- (5.6,-3.1);
\draw (5.8,-2.9) -- (5.8,-3.1);
\draw (6,-2.9) -- (6,-3.1);
\draw (6.2,-2.9) -- (6.2,-3.1);
\draw (6.4,-2.9) -- (6.4,-3.1);
\draw (6.6,-2.9) -- (6.6,-3.1);
\draw (6.8,-2.9) -- (6.8,-3.1);
\draw (7,-2.9) -- (7,-3.1);
\draw (7.2,-2.9) -- (7.2,-3.1);
\draw (7.4,-2.9) -- (7.4,-3.1);
\draw (7.6,-2.9) -- (7.6,-3.1);
\draw (7.8,-2.9) -- (7.8,-3.1);
\draw (8,-2.9) -- (8,-3.1);
\draw (8.2,-2.9) -- (8.2,-3.1);
\draw (8.4,-2.9) -- (8.4,-3.1);
\draw (8.6,-2.9) -- (8.6,-3.1);
\draw (8.8,-2.9) -- (8.8,-3.1);
\draw (9,-2.9) -- (9,-3.1);
\draw (9.2,-2.9) -- (9.2,-3.1);
\draw (9.4,-2.9) -- (9.4,-3.1);
\draw (9.6,-2.9) -- (9.6,-3.1);

\draw[|-|] (5,-3.3) -- (5.2,-3.3);
\draw (5.1,-3.3)node[below]{$\xi_i$};
\draw (2.5,-3.3)node[below]{$\g \left( \g^{-1}(m_1) - \frac{1}{2} \xi_i \right)$};
\draw[<-] (0.42,-3.1) -- (1,-3.6);
\draw (7.7,-3.3)node[below]{$\g \left( \g^{-1}(m_10) + \frac{1}{2} \xi_i \right)$};
\draw[->] (9.2,-3.6) -- (9.58,-3.1);

%\draw (0,0) -- (0,2);
%\draw (10,0) -- (10,2);
%\draw[dotted] (-0.5,1) arc (180:360:0.5cm);
%\draw[dotted] (9.5,1) arc (180:360:0.5cm);
%\draw[dotted] (-0.5,1) -- (-0.5,2);
%\draw[dotted] (0.5,1) -- (0.5,2);
%\draw[dotted] (9.5,1) -- (9.5,2);
%\draw[dotted] (10.5,1) -- (10.5,2);

\end{tikzpicture}
\end{center}
\caption{Constructing $\scrV_i$ so that midpoints remain midpoints (with $N_{i-1} = 10$).  In general, it will not be the case that $a=x_0$ or $b=x_{10}$.}
\label{subdividing gamma}
\end{figure}
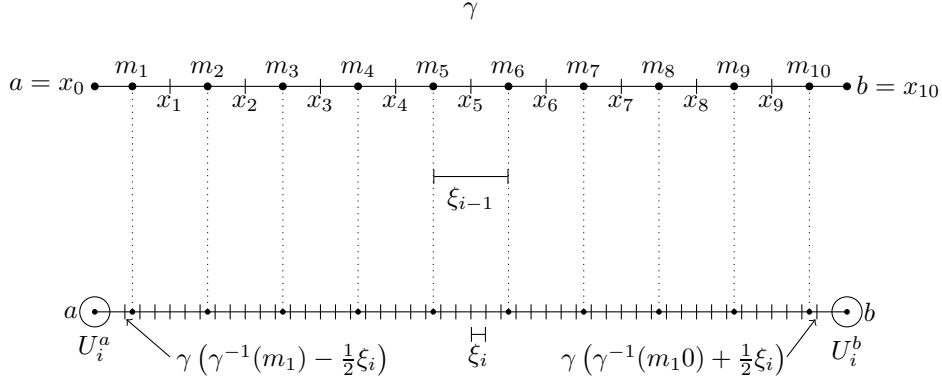

\begin{remark}\label{remark:  not intersecting geodesics in next stage}
Since $\D_{i+1}$ is obtained from $\D_i$ by adding a point, $\G_{i+1}$ is obtained from $\G_i$ by inserting ``new" geodesics from this point to the points of $\D_i$.
Then since the entire collection $\{ \G_i' \}_{i=1}^\infty$ is defined before any of the collections $\O_i$ are constructed, we require $\xi_i$ to be small enough so that none of the neighborhoods in $\scrV_i$ meet any of the ``new" geodesic segments from $\G_{i+1}$ (unless the image of that geodesic intersects a member of $\G_i$, of course).  
\end{remark}

\vskip 20pt

\subsection*{Construction of $\scrW_i$}
This is identical to the construction of $\scrW_1$ from Section 5.

\vskip 20pt

\subsection{Construction of $\N_i$, and endowing $\N_i$ with an initial metric $g_i'$}

Just as in Section 5, $\N_i$ denotes the nerve of $\O_i$.  
The same argument as in Section 5 also shows that order($\O_i) \leq n+3$, implying that dim($\N_i) \leq n+2$.  
What we do now is define a ``preliminary" metric $g_i'$ on $\N_i$.  
Later in this Section, while in the process of defining the maps $h_i$, we will (possibly) alter the metric $g_i'$ slightly to obtain a new metric $g_i$ so that the maps $\psi_i$, $\v_{i+1,i}$, and $h_i$ satisfy certain desirable geometric properties.

Let %$\{ \rho_i \}_{i = 1}^\infty , 
$( \omega_i )_{i = 1}^\infty$ be a monotone decreasing sequence of positive real numbers which converges to $0$. 
The sequence $( \omega_i )_{i = 1}^\infty$ will be required to converge to $0$ at a sufficiently fast rate, to be specified later when defining the maps $\v_{i+1,i}$.  

If any vertex of any edge $e \in \N_i$ corresponds to an open set in $\scrW_i$, define $g_i'(e) := M_i'$, where $M_i'$ is a large constant.  
Otherwise, both vertices of $e$ correspond to open sets in $\scrU_i \cup \scrV_i$.  
Let $A, B \in \scrU_i \cup \scrV_i$ denote the open sets corresponding to the vertices of $e$, respectively.
Since no two open sets of $\scrU_i$ intersect nontrivially, at least one of $A$ or $B$ is in $\scrV_i$.  
Let $\g \in \G_i'$ be the geodesic corresponding to this open set, and note that $\g$ is well-defined since if both $A, B \in \scrV_i$ and $A \cap B \neq \emptyset$ then $A$ and $B$ correspond to the same geodesic.

Recall that in the definition of $\scrV_i$ we subdivided the geodesic segment
	\[
	\overline{\g} := \text{im}(\g) \setminus \bigcup_{U \in \scrU_i} U
	\]
into equidistant subintervals of length $\xi_i$.  
Then we define
	\begin{equation}\label{definition of g'}
	g_i'(e) := (1-\omega_i) \xi_i.
	\end{equation}

The exact same argument as in Section 5 shows that, by choosing $M_i'$ large enough, we have that $g_i'$ is a Euclidean metric on $\N_i$.  
Also, by choosing $\e_i$ sufficiently small, we can ensure that the shortest path between two vertices corresponding to members of $\scrU_i$ is the edge-path corresponding to the geodesic in $\G_i$ between those points.

\vskip 20pt

\subsection{The maps $\psi_i$ and $\v_{i+1,i}$}

%%%%%%%%%%%%%%%%%%%%%%%%%%%%%%%%%%%%%%%%%%%%%%%%%%%%%%%%%%%%%%%%%%%%%%%%%
\begin{comment}
Let $\{ \phi_k^i \}_{k \in \Lambda_i}$ be a partition of unity subordinate to $\O_i$.  
Just as in Section 3 we use this partition of unity to define the map $\psi_i: \X \to \N_i$.  
Note that the midpoints of subintervals of geodesics used to construct $\scrV_i$ are, by necessity, mapped under $\psi_i$ to the vertex in $\N_i$ corresponding to that subinterval.
\end{comment}
%%%%%%%%%%%%%%%%%%%%%%%%%%%%%%%%%%%%%%%%%%%%%%%%%%%%%%%%%%%%%%%%%%%%%%%%%

In this Subsection we define, for all $i \in \mathbb{N}$, maps $\psi_i : \X \to \N_i$ and $\v_{i+1,i} : \N_{i+1} \to \N_i$.  
Unfortunately, by necessity, our construction is is sort of roundabout.  
What we first do is define $\psi_i$ on the geodesic segments of $\G_i'$ for all $i$.
From there we define the map $\v_{i+1,i}$ again for all $i$, and then we use the collection $( \v_{i+1,i} )$ to finish defining $\psi_i$.  
%As a last subtle point we verify that the definition of $\psi_i$ over all of $\X$ is continuous, which is necessary due to its piecewise definition.

\vskip 20pt

\subsection*{Defining $\psi_i : \text{im}(\G_i') \to \N_i$}
Let $\g \in \G_i'$, let $m_1, m_2, \hdots, m_{N_i}$ denote the midpoints of the geodesic segments of $\g$ corresponding to the construction of $\scrV_i$, and let $x$ and $y$ denote the endpoints of $\g$.  
By construction, each $m_j$ is contained in a unique open set $V_j \subset \O_i$. 
So we define $\psi_i (m_j) = v_j$, where $v_j$ is the vertex of $\N_i$ corresponding to $V_j$.   
Likewise, both $x$ and $y$ are contained in unique open sets $U_x, U_y \in \O_i$.  
So we define $\psi_i(x) = v_x$ and $\psi_i(y) = v_y$, where $v_x$ and $v_y$ are the vertices of $\N_i$ corresponding to $U_x$ and $U_y$, respectively.

We now need to extend $\psi_i$ over all of $\g$.  
For the geodesic segments $[m_1, m_2], $ $[m_2, m_3], \hdots, [m_{N_i - 1}, m_{N_i}]$ we just extend linearly.  
But on the geodesic segments $[x, m_1]$ and $[m_{N_i}, y]$ the map $\psi_i$ will {\it not} be linear.  
To define $\psi_i$ here, let $m_0$ and $m_{N_i + 1}$ be as in the construction of $\scrV_i$ (or see Figure \ref{subdividing gamma} and/or \ref{defining varphi}).  
We define $\psi_i ([x,m_0]) = v_x$ and $\psi_i ([m_{N_i + 1}, y]) = v_y$, and we define $\psi_i$ linearly on $[m_0,m_1]$ and $[m_{N_i}, m_{N_i + 1}]$.  
Lastly, note that $\psi_i$ is $(1 - \omega_i)$-Lipschitz over all of $\text{im}(\G_i)$.

\vskip 20pt

\subsection*{Defining $\v_{i+1,i}: \N_{i+1} \to \N_i$}
%and then extending linearly (which is well-defined since the closed star of any vertex in $\N_i$ is uniquely geodesic).
Recall that the mesh of $\O_{i+1}$ is strictly less than one third of the Lebesgue number of $\O_i$.  
Therefore $\O_{i+1}$ is a {\it star-refinement} of $\O_i$.  
This means that for all $U \in \O_{i+1}$, there exists $V \in \O_i$ such that $V$ contains both $U$ and every member of $\O_{i+1}$ which meets $U$.
%So every member of $\O_{i+1}$ is contained in at least one open set of $\O_i$.  

Let us first define $\v_{i+1,i}: \N_{i+1} \to \N_i$ on the vertices of $\N_{i+1}$ which correspond to members of $\scrW_{i+1}$. 
Let $v \in \N_{i+1}$ be such a vertex, and let $W_v \in \scrW_{i+1}$ denote the open set corresponding to $v$.  
Then define $\v_{i+1,i}(v)$ to be the barycenter of the maximal simplex in $\N_i$ whose vertices all correspond to open sets in $\O_i$ which contain $W_v$, which clearly exists since $\O_{i+1}$ refines $\O_i$.  

Now suppose $v \in \N_{i+1}$ is a vertex corresponding to some set $U_v \in \scrU_{i+1}$.  
Let $x_v \in \D_{i+1}'$ denote the center of $U_v$.  
Either $x_v \in \D_i'$ or $x_v \nin \D_i'$.  
If $x_v \in \D_i'$, then $\v_{i+1,i}$ maps $v$ to the vertex in $\N_i$ corresponding to the member of $\scrU_i$ containing $x_v$.  
Otherwise $x_v \nin \D_i$ and we treat $v$ as if $U_v$ were in $\scrW_{i+1}$ and map $v$ to the barycenter of the simplex spanned by all vertices corresponding to members of $\O_i$ which contain $U_v$.  
Note that this definition is consistent, since if $x_v \in \D_i'$ then $U_v$ is contained in only one element of $\O_i$.  

Lastly, suppose that $v \in \N_{i+1}$ is a vertex corresponding to some set $V \in \scrV_{i+1}$, and let $\g \in \G_{i+1}$ be the geodesic corresponding to $V$.
If $\g$ does not correspond to a geodesic in $\G_i$ then just as above we treat $v$ as if $V$ were in $\scrW_{i+1}$ and map $v$ to the barycenter of the maximal simplex spanned by all of the vertices corresponding to members of $\O_i$ which contain $V$.

So suppose $\g$ {\it does} correspond to a geodesic in $\G_i$.  
Let $[x,y]$ denote the geodesic segment of $\g$ corresponding to $V$, and let $m$ denote the midpoint of this segment.
If $m$ is the midpoint of some subinterval corresponding to a set $V' \in \scrV_i$, then we define $\v_{i+1,i}(v)$ to be the vertex in $\N_i$ corresponding to $V'$. 
So suppose that $m$ is not the midpoint of any subinterval corresponding to a member of $\scrV_i$.  
Using the same notation as in Section 5, let $m_0, m_1, \hdots, m_{N_i + 1}$ denote the midpoints of the subintervals of $\g$ with respect to the cover $\scrV_i$ (please see figure \ref{defining varphi}).  
Let $a$ and $b$ be the endpoints of $\g$ so, in particular, $a, b \in \D_i'$.    
Let $m_j, m_{j+1}$ be the two midpoints of geodesic segments corresponding to elements of $\scrV_i$ nearest to $m$ (and where one of these midpoints {\it could} instead be either $a$ or $b$, please see Figure \ref{defining varphi}). 
Let $V_j, V_{j+1} \in \scrV_i$ be the open sets corresponding to $m_j$ and $m_{j+1}$, and let $v_j, v_{j+1} \in \N_i$ be the corresponding vertices in $\N_i$.  
Then the vertices $v_j$ and $v_{j+1}$ are adjacent in $\N_i$ via some edge $e$, and if neither $m_j = a$ nor $m_{j+1} = b$ then we have the following two equalities
	\begin{equation*}
	\dx(m_j,m) + \dx(m,m_{j+1}) = \xi_i 	\qquad	d_i' (v_j,v_{j+1}) = (1 - \omega_i) \xi_i
	\end{equation*}
where $d_i'(,)$ denotes the path metric on $\N_i$ induced by the metric $g_i'$.  
In this case, define $\v_{i+1,i}(v) = p \in e$ where $p$ is the unique point on $e$ which satisfies both
	\begin{equation}\label{defining varphi on V}
	d_i'(v_j,p) = (1-\omega_i) \dx(m_j,m) 	\qquad	d_i'(p,v_{j+1}) = (1-\omega_i) \dx(m,m_{j+1}).
	\end{equation}
	
Now, if either $m_j = a$ or $m_{j+1} = b$ then $\psi_i(m_j) = \psi_i(m_{j+1})$ and therefore $v_j = v_{j+1}$.  
So $d_i'(v_j,v_{j+1}) = 0$ and by necessity we define $\v_{i+1,i}(v) = v_j$.
For a picture, please see Figure \ref{defining varphi}

Now, to extend $\v_{i+1,i}$ to all of $\N_{i+1}$ just note that by the ``star-refinement" remark, adjacent vertices of $\N_{i+1}$ are mapped to adjacent simplices of $\N_i$.  
The metric $g_i'$ restricted to these two adjacent simplices is uniquely geodesic.  
So we extend $\v_{i+1,i}$ along the $1$-skeleton of $\N_{i+1}$ via these geodesics, and analogously we extend to all of $\N_{i+1}$.

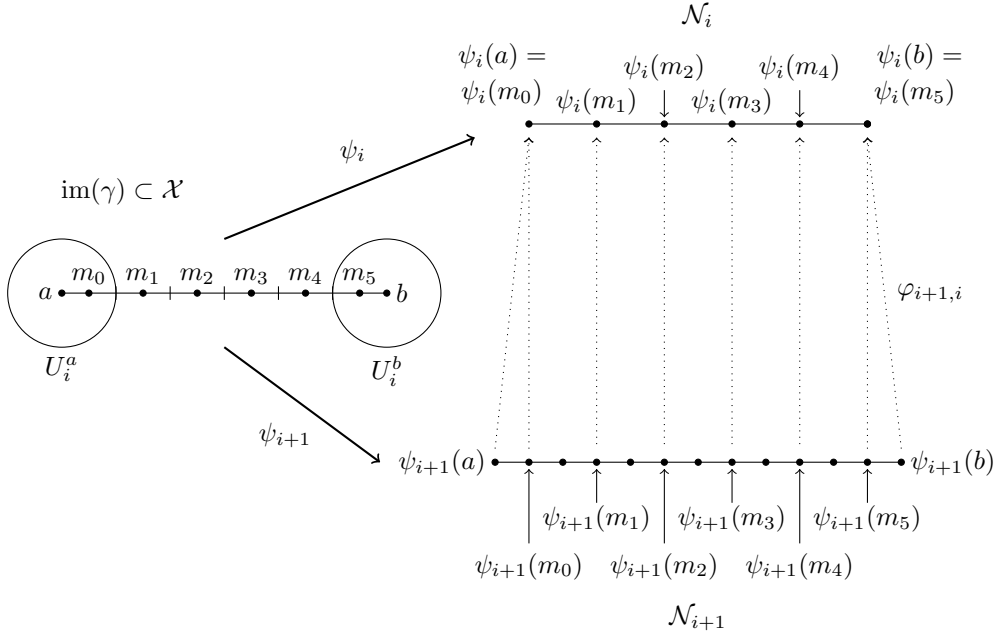
\begin{figure}
\begin{center}
\begin{tikzpicture}[scale=0.9]

\draw (-5.5,0)node{$\text{im}(\g) \subset \X$};
\draw (-6.4,-1.5) -- (-1.6,-1.5);
\draw[fill=black!] (-6.4,-1.5) circle (0.3ex);
\draw (-6.4,-1.5)node[left]{$a$};
\draw[fill=black!] (-1.6,-1.5) circle (0.3ex);
\draw (-1.6,-1.5)node[right]{$b$};
\draw (-5.6,-1.5) arc (0:360:0.8cm);
\draw (-0.8,-1.5) arc (0:360:0.8cm);
\draw (-6.4,-2.3)node[below]{$U_i^a$};
\draw (-1.6,-2.3)node[below]{$U_i^b$};

\draw (-5.6,-1.6) -- (-5.6,-1.4);
\draw (-4.8,-1.6) -- (-4.8,-1.4);
\draw (-4,-1.6) -- (-4,-1.4);
\draw (-3.2,-1.6) -- (-3.2,-1.4);
\draw (-2.4,-1.6) -- (-2.4,-1.4);

\draw[fill=black!] (-6,-1.5) circle (0.3ex);
\draw[fill=black!] (-5.2,-1.5) circle (0.3ex);
\draw[fill=black!] (-4.4,-1.5) circle (0.3ex);
\draw[fill=black!] (-3.6,-1.5) circle (0.3ex);
\draw[fill=black!] (-2.8,-1.5) circle (0.3ex);
\draw[fill=black!] (-2,-1.5) circle (0.3ex);
\draw (-6,-1.5)node[above]{$m_0$};
\draw (-5.2,-1.5)node[above]{$m_1$};
\draw (-4.4,-1.5)node[above]{$m_2$};
\draw (-3.6,-1.5)node[above]{$m_3$};
\draw (-2.8,-1.5)node[above]{$m_4$};
\draw (-2,-1.5)node[above]{$m_5$};

\draw[->][thick] (-4,-0.7) -- (-0.3,0.8);
\draw[->][thick] (-4,-2.3) -- (-1.7,-4);
\draw (-2.1,0.6)node{$\psi_i$};
\draw (-3.1,-3.6)node{$\psi_{i+1}$};

\draw (0.5,1) -- (5.5,1);
\draw[fill=black!] (0.5,1) circle (0.3ex);
\draw (0.9,2)node[left]{$\psi_i(a) = $};
\draw (0.1,1.8)node[below]{$\psi_i(m_0)$};
\draw[fill=black!] (5.5,1) circle (0.3ex);
\draw (5.5,2)node[right]{$\psi_i(b) =$};
\draw (6.2,1.8)node[below]{$\psi_i(m_5)$};

\draw (3,2.6)node{$\N_i$};

\draw[fill=black!] (1.5,1) circle (0.3ex);
\draw[fill=black!] (2.5,1) circle (0.3ex);
\draw[fill=black!] (3.5,1) circle (0.3ex);
\draw[fill=black!] (4.5,1) circle (0.3ex);
\draw[fill=black!] (5.5,1) circle (0.3ex);

\draw (1.5,1)node[above]{$\psi_i(m_1)$};
\draw (2.5,1.5)node[above]{$\psi_i(m_2)$};
\draw[->] (2.5,1.5) -- (2.5,1.1);
\draw (3.5,1)node[above]{$\psi_i(m_3)$};
\draw (4.5,1.5)node[above]{$\psi_i(m_4)$};
\draw[->] (4.5,1.5) -- (4.5,1.1);

%\draw[|-|] (4.5,-1.2) -- (5.5,-1.2);
%\draw (5,-1.5)node{$\xi_{i-1}$};

\draw (3,-6.3)node{$\N_{i+1}$};
\draw (0,-4) -- (6,-4);
\draw[fill=black!] (0,-4) circle (0.3ex);
\draw (0,-4)node[left]{$\psi_{i+1}(a)$};
\draw[fill=black!] (6,-4) circle (0.3ex);
\draw (6,-4)node[right]{$\psi_{i+1}(b)$};

\draw[fill=black!] (0.5,-4) circle (0.3ex);
\draw[fill=black!] (1,-4) circle (0.3ex);
\draw[fill=black!] (1.5,-4) circle (0.3ex);
\draw[fill=black!] (2,-4) circle (0.3ex);
\draw[fill=black!] (2.5,-4) circle (0.3ex);
\draw[fill=black!] (3,-4) circle (0.3ex);
\draw[fill=black!] (3.5,-4) circle (0.3ex);
\draw[fill=black!] (4,-4) circle (0.3ex);
\draw[fill=black!] (4.5,-4) circle (0.3ex);
\draw[fill=black!] (5,-4) circle (0.3ex);
\draw[fill=black!] (5.5,-4) circle (0.3ex);

\draw[<-][dotted] (0.5,0.8) -- (0,-3.8);
\draw[<-][dotted] (1.5,0.8) -- (1.5,-3.8);
\draw[<-][dotted] (2.5,0.8) -- (2.5,-3.8);
\draw[<-][dotted] (3.5,0.8) -- (3.5,-3.8);
\draw[<-][dotted] (4.5,0.8) -- (4.5,-3.8);
\draw[<-][dotted] (5.5,0.8) -- (6,-3.8);
\draw[->][dotted] (0.5,-3.8) -- (0.5,0.8);
\draw[->][dotted] (5.5,-3.8) -- (5.5,0.8);

\draw (0.5,-5.2)node[below]{$\psi_{i+1}(m_0)$};
\draw[->] (0.5,-5.2) -- (0.5,-4.1);
\draw (1.5,-4.5)node[below]{$\psi_{i+1}(m_1)$};
\draw[->] (1.5,-4.6) -- (1.5,-4.1);
\draw (2.5,-5.2)node[below]{$\psi_{i+1}(m_2)$};
\draw[->] (2.5,-5.2) -- (2.5,-4.1);
\draw (3.5,-4.5)node[below]{$\psi_{i+1}(m_3)$};
\draw[->] (3.5,-4.6) -- (3.5,-4.1);
\draw (4.5,-5.2)node[below]{$\psi_{i+1}(m_4)$};
\draw[->] (4.5,-5.2) -- (4.5,-4.1);
\draw (5.5,-4.5)node[below]{$\psi_{i+1}(m_5)$};
\draw[->] (5.5,-4.6) -- (5.5,-4.1);

\draw (5.8,-1.5)node[right]{$\v_{i+1,i}$};

\end{tikzpicture}
\end{center}
\caption{The construction of $\v_{i+1,i}:\N_{i+1} \to \N_i$, where $N_i = 4$.
Note that the Figure is not to scale.  
Distances in $\N_{i+1}$ are larger than their counterparts in $\N_i$ since $\omega_{i+1} < \omega_i$.}
\label{defining varphi}
\end{figure}

\vskip 20pt

\subsection*{Defining $\psi_i$ over all of $\X$}
Note that, for $i < j$, we have a map $\v_{j,i}: \N_j \to \N_i$ defined by composition. 
In \cite{Isbell} Isbell proves that, if each map $\v_{i+1,i}$ mapped {\it every} vertex of $\N_{i+1}$ to the barycenter of the simplex spanned by all members of $\O_i$ which contained the corresponding open set, then $\X$ would be homeomorphic to the inverse limit of the system $(\N_i, \v_{j, i})$.  
Let us call such a map a {\it barycentric map}, and such a system a {\it barycentric system}.  
One remark is that, for Isbell's result, it is necessary that $\O_{i+1}$ be a star-refinement of $\O_i$ for each $i$ (as is the case here).  

The map $\v_{i+1,i}$ as we have defined it is very nearly a barycentric map.  
It is barycentric on all vertices of $\N_{i+1}$ corresponding to members of $\scrU_{i+1} \cup \scrW_{i+1}$.  
If $v \in \N_{i+1}$ is a vertex corresponding to a member $V \in \scrV_{i+1}$ corresponding to a geodesic $\g \in \G_{i+1}$, then $\v_{i+i,i}$ is barycentric with respect to $v$ if either $\g \nin \G_i$ or if $v$ corresponds to a midpoint of $\g$ at stage $i$.  
So the only case in which $\v_{i+1,i}$ is {\it not} barycentric with respect to $v$ is when the corresponding geodesic $\g$ is in $\G_i$ and $v$ does not correspond to a midpoint for $\g$ in the $i^{th}$ stage.  
But for such vertices $v$, the barycentric map would just send $v$ to the barycenter of the corresponding edge in $\N_i$, whereas our definition via equation \eqref{defining varphi on V} just shifts the image along this edge by a distance of at most $\frac{1}{2} (1 - \omega_i) \xi_i$.  

We need to analyze Isbell's argument in \cite{Isbell} to ensure that it still applies to our setting.  
So a quick outline is as follows.
For any $x \in \X$, let $\s^i(x) = \s^i$ denote the closed simplex in $\N_i$ corresponding to the set of all elements in $\O_i$ which contain $x$.  
Let $f_{j,i}: \N_j \to \N_i$ be the composition of the barycentric maps.  
We have that $f_{i+1,i}(\s^{i+1}) \subseteq \s^i$ since $\O_{i+1}$ is a star-refinement of $\O_i$.  
So $(\s^i, f_{j,i})$ is an inverse system of compact spaces, and therefore its limit is a nonempty subset of the limit space of the larger barycentric system $(\N_i, f_{j,i})$.  
Isbell endows each $\N_i$ with the metric where the distance between two points is defined as the maximum difference in their corresponding barycentric coordinates.
If $v \in \N_{i+1}$ is a vertex and $u \in \N_i$ is a vertex corresponding to an open set which contains the star of the open set corresponding to $v$, then $f_{n+1,n}$ sends the closed star of $v$ into the set of all points with $u^{th}$ barycentric coordinate at least 1/(N+1) (where $N=n+2$ denotes the dimension of $\N_i$ for all i).
Then each map $f_{i+1,i}$ is (N/(N+1))-Lipschitz.  %(where, $n+2 = \text{dim} (\N_i)$).  
Therefore, the inverse limit of $(\s^i, f_{j,i})$ is a single point in the limit space of $(\N_i, f_{j,i})$, which we identify with $x$.  

First, note that the metric that we are putting on $\N_i$ in this paper is very different than what Isbell uses.  
But the conclusion of Isbell's result is purely topological, and so this does not cause any issue since both metrics induce the same topology as that inherited from the simplicial complex structure of $\N_i$.
We will show that for $\e_{i+1}$ sufficiently small the map $\v_{i+1,i}$ is still $(N/(N+1))$-Lipschitz with respect to Isbell's metric.
Then $\X$ will be naturally identified with the inverse limit of the system $(\N_i, \v_{j,i})$, and we will show that the projection maps $\psi_i$ agree with our definition above on $\G_i$.

Let $v \in \N_{i+1}$ be a vertex on which $\v_{i+1,i}$ is not barycentric.  
Let $\g \in \G_{i+1}$ be the corresponding geodesic, let $m_v \in \X$ be the midpoint corresponding to $v$ corresponding to stage $(i+1)$, and let $V \in \scrV_{i+1}$ be the corresponding open set.
Since $\v_{i+1,i}$ is not barycentric on $v$, we know that $\g \in \G_i$ and that $m_v$ is not a midpoint of $\g$ corresponding to stage $i$.  
Let $m_j, m_{j+1}$ denote the midpoints nearest $m_v$ in stage $i$ with corresponding vertices $v_j, v_{j+1} \in \N_i$.  
Lastly,  let $V_j, V_{j+1} \in \scrV_i$ denote the open sets corresponding to $v_j$ and $v_{j+1}$, respectively.

The two key observations are the following.  
The first is that since $m_v$ is not a midpoint at stage $i$, we may choose $\e_{i+1}$ small enough so that {\it both} $V_j$ and $V_{j+1}$ contain the star of $V$. 
The second is that, by the construction of $\O_{i+1}$, $V_j$ and $V_{j+1}$ are the only members of $\O_i$ which have nontrivial intersection with $V$.  
So if $W \in \scrW_{i+1}$ is in the star of $V$, then the elements of $\O_i$ which contain $W$ are precisely $V_j$ and $V_{j+1}$.  
Therefore, if $w$ is the vertex in $\N_{i+1}$ corresponding to $W$, then 
	\begin{equation*}
	\v_{i+1,i}(w) = \frac{1}{2} v_j + \frac{1}{2} v_{j+1}.
	\end{equation*}
So it is clear that the maximal difference in the barycentric coordinates of $\v_{i+1,i}(v)$ and $\v_{i+1,i}(w)$ is less than or equal to 1/2.
The vertex $v$ is also adjacent to two other vertices corresponding to members of $\scrV_{i+1}$.  
But it is clear that $\v_{i+1,i}$ reduces the difference in the barycentric coordinates of these vertices by more than N/(N+1) as well (in $\N_{i+1}$ their difference is 1, whereas in $\N_i$ it is less than $\frac{1}{2} (1-\omega_i)\xi_i << N/(N+1)$).

Therefore, the same argument as in Isbell's paper holds, and we have that $\X$ is the inverse limit of the system $(\N_i, \v_{j,i})$.  
The map $\psi_i: \X \to \N_i$ is then the natural projection map.  
If we fix $\g \in \G$, then this map $\psi_i$ will agree with our original definition above on all points in the image of $\g$ which are eventually midpoints of some subinterval in the construction of some $\scrV_j$.  
This forms a dense set in the image of $\g$, and so by continuity $\psi_i$ will agree with our original definition over all of $\g$.

\begin{comment}
To see how we extend $\v_{i+1,i}$, consider a simplex $\s \in \N_{i+1}$ which contains the vertex $v$.  
All other vertices of $\s$ correspond to open sets in $\O_{i+1}$ which have nontrivial intersection with $O_v$.  
So if $\tau \in \N_i$ is the simplex containing $\v_{i+1,i}(v)$, then the closed star $St(\tau)$ contains the images of all of the vertices of $\s$.  
For a sufficiently fine cover of $\X$ (so, for $\d_i$ sufficiently small), $St(\tau)$ will be uniquely geodesic.
So we can extend the map $\v_{i+1,i}$ linearly over $\s$, and thus to all of $\N_{i+1}$.  
It is clear that this definition will agree on the intersection of adjacent simplices of $\N_{i+1}$.  
\end{comment}

\vskip 20pt

\subsection{The geometry of $\v_{i+1,i}$}
In this Subsection we will show that $\v_{i+1,i}$ is 1-Lipschitz over ``most" of $\N_{i+1}$ for $M_{i+1}'$ chosen large enough.  
As is discussed in Section 3, the piecewise flat Euclidean metric on $\N_{i+1}$ is completely determined by the lengths that $g_{i+1}'$ associates to each edge of $\N_{i+1}$.  
So for each edge $e$ of $\N_{i+1}$ we need to consider the length of the piecewise linear segments of $\v_{i+1,i}(e)$.  

Let $e \in \N_{i+1}$ be an edge.  
If one or both of the vertices of $e$ correspond to members of $\scrW_{i+1}$, then $\ell(e) = M_{i+1}'$.  
If $M_{i+1}'$ is chosen large enough then we can ensure that $\ell(\v_{i+1,i}(e)) < \ell(e)$ for all such edges $e$ (since $\N_{i+1}$ is finite).
Moreover, we can choose $M_{i+1}'$ large enough so that $\v_{i+1,i}$ is 1-Lipschitz over all equilateral simplices of $\N_{i+1}$ with edge lengths $M_{i+1}'$.  
%So we see that we can choose $M_{i+1}'$ large enough so that $\v_{i+1,i}$ is $1$-Lipschitz when restricted to any edge with a vertex corresponding to a member of $\scrW_{i+1}$.  

So now suppose that neither vertex of $e$ corresponds to an open set in $\scrW_{i+1}$.  
Let $\g \in \G_{i+1}$ be the (unique) associated geodesic, and let $\xi_{i+1}$ denote the length of the subintervals associated to $\g$ in stage $i+1$.  
There are two cases:

{\it Case 1: }  $\g \in \G_i$.  
If $e$ corresponds to a geodesic segment on one of the ``ends" of $\g$, then $\v_{i+1,i}$ maps all of $e$ to a vertex in $\N_i$ corresponding to an element of $\scrU_i$ (see Figure \ref{defining varphi}).  
So, in particular, $\v_{i+1,i}$ is $1$-Lipschitz when restricted to this edge (since the image has length $0$).  
So now suppose $e$ corresponds to an ``interior" segment of $\g$.  
Then by equation \ref{defining varphi on V} we have that
	\begin{equation*}
	\ell(\v_{i+1,i}(e)) = \frac{1 - \omega_i}{1 - \omega_{i+1}} \ell(e).
	\end{equation*}
Since $\omega_{i+1} < \omega_i$, we see that, again, $\v_{i+1,i}$ is $1$-Lipschitz when restricted to this edge $e$.  
So again it is easy to see that we may choose $M_{i+1}'$ large enough so that $\v_{i+1,i}$ is 1-Lipschitz on all simplices containing such an edge $e$.

\vskip 10pt

{\it Case 2: }  $\g \nin \G_i$.
This is the one and only case when $\v_{i+1,i}$ may not be $1$-Lipschitz, but due to linearity may instead be $1$-expanding.  
To examine this case, let $a$ and $b$ denote the vertices of $e$.  
Also, let $U_a, U_b \in \O_{i+1}$ denote the open sets corresponding to $a$ and $b$.  
By Remark \ref{remark:  not intersecting geodesics in next stage}, if $\g \nin \G_i$ then the only members of $\O_i$ which contain $U_a$ and $U_b$ come from $\scrW_i$.  
Therefore, $a$ and $b$ are mapped by $\v_{i+1,i}$ to the barycenter of equilateral simplices in $\N_i$, each of whose edges have length $M_i'$.  
If $U_a$ and $U_b$ are contained in the exact same open sets of $\O_i$, then $\v_{i+1,i}(a) = \v_{i+1,i}(b)$ and $\v_{i+1,i}$ is $1$-Lipschitz on $e$ (since the length of the image is $0$).  
Otherwise, $\v_{i+1,i}(a)$ and $\v_{i+1,i}(b)$ are the barycenters of adjacent simplices in $\N_i$.  
In this case, it is clear that for a sufficiently large choice of $M_i'$ that
	\begin{equation}\label{1-expanding edges}
	\ell(\v_{i+1,i}(e)) \geq 2 d_i'(b(\s),\partial (\s)) > (1 - \omega_{i+1}) \xi_{i+1} = \ell(e)
	\end{equation}
where $\s$ denotes an equilateral $(n+2)$-dimensional simplex whose edges all have length $M_i'$, $b(\s)$ denotes the barycenter of $\s$, and $\partial (\s)$ denotes the boundary of $\s$.  

So in conclusion, $\v_{i+1,i}$ is 1-Lipschitz except on portions of the open star of the image under $\psi_{i+1}$ of geodesic segments in $\G_{i+1} \setminus \G_i$.

%{\color{red}Need to make Figure ``defining varphi".}

\vskip 20pt

\subsection{Altering the metric $g_i'$ to obtain the metric $g_i$, and the construction of the map $h_i$}
For each $i \in \mathbb{N}$ we have now defined maps $\psi_i: \X \to \N_i$ and $\v_{i+1,i}: \N_{i+1} \to \N_i$.  
Also, for $j > i$, we can also define the map $\v_{j,i}: \N_j \to \N_i$ as the composition $\v_{i+1,i} \circ \v_{i+2,i+1} \circ \hdots \circ \v_{j,j-1}$.  
Note that, since $\v_{i+1,i}$ was defined on vertices and then extended linearly, each map $\v_{j,i}$ is piecewise linear (pl).

As a preliminary step, we first (sequentially) enlarge the constants $M_i'$ to obtain new constants $M_i''$ so that inequality \eqref{1-expanding edges} is satisfied for all edges $e$ for which that setup applies.
Let us call this new metric $g_i''$.    
The order in which we define things is as follows.  
Once we have defined $g_1$ and $h_1$, we use Lemma \ref{technical lemma} to approximate $h_1 \circ \v_{2,1}$ with a map denoted $H_2$.  
The use of Lemma \ref{technical lemma} possibly requires us to enlarge the constants $M_2''$, which will lead to the metric $g_2$.  
Then from here we use Theorem \ref{Minemyer} to approximate $H_2$ with our desired map $h_2$.  
We then iterate this construction.  
So we use Lemma \ref{technical lemma} to construct $H_3$ which approximates $h_2 \circ \v_{3,2}$, and so on.

%We define the metric $g_i$ and the pl isometric embedding $h_i: \N_i \to \R^{3n+6,1}$ sequentially.  
%So we begin by defining the metric $g_1 = g_1''$, then we use $g_1$ to define the map $h_1$, then we use $h_1$ to define the metric $g_2$, then we use $g_2$ to define the map $h_2$, and so on.  
%We now recursively define the collection of maps $\{ h_i \}$, $h_i: \N_i \to \R^{3n+6,1}$, and in the process we will see the last scaling that we need to make in order to obtain the metric $g_i$.  

The map $h_1: \N_1 \to \E^{3n+6} \subset \R^{3n+6,1}$ is simply any pl isometric embedding, whose existence is guaranteed by Theorem \ref{Minemyer} in Section 3 (or see \cite{Minemyer1}).  
Let $( \rho_i)_{i=1}^\infty$ be a monotone decreasing sequence of positive real numbers which converges to zero.  
The map $h_i$ will be a $\rho_i$ approximation of $h_{i-1}$ for each $i$, and in Section 7 we will put restrictions on how quickly we demand the sequence $(\rho_i)$ to converge to $0$.

%Start rewording here using Lemma \ref{technical lemma}

Let $H_2' := h_1 \circ \v_{2,1} : \N_2 \to \R^{3n+6,1}$. 
Since $H_2'$ is the composition of pl maps, it is pl.  
So let $\T_2$ be a triangulation of $\N_2$ on which $H_2'$ is simplicial.  
Define $\mathcal{S}_2 \subset \T_2$ by
	\begin{equation*}
	\mathcal{S}_2 = \{ \Delta \in \T_2 \bigr| \, \exists \, \g \in \G_2 \setminus \G_1 \text{ such that } \Delta \cap \text{im}(\psi_2 \circ \g) \neq \emptyset \}.
	\end{equation*}
So $\mathcal{S}_2$ is just the set of all simplices in $\T_2$ with an edge corresponding to a geodesic which is in $\G_2$ but not in $\G_1$.  
From Subsection 6.4 we know that $\v_{2,1}$, and thus $H_2'$, is 1-Lipschitz outside on $\T_2 \setminus \S_2$, and on $\S_2$ the map $H_2'$ is expanding along (most of) the edges which correspond to members of $\G_2'$.  

So we apply Lemma \ref{technical lemma} (and Remark \ref{technical lemma remark}) sequentially to each edge in $\S_2$ corresponding to a geodesic in $\G_2'$ on which $H_2'$ is expanding.
In applying Lemma \ref{technical lemma} to any such edge we may need to increase the constant $M_2''$ as required by the Lemma.
We then define $M_2$ to be the maximum required constant over all of the edges, and this defines the Euclidean metric $g_2$.  
The resulting map, call it $H_2$, will now be 1-Lipschitz over all of $\N_2$ and can be constructed so as to be a $\frac{\rho_2}{2}$-approximation of $H_2'$.  
Let $\T_2'$ be a triangulation of $\N_2$ on which $H_2$ is simplicial.

Let $\pi^+: \R^{3n+6,1} \to \R^{3n+6,1}$ be the projection onto the first $(3n+6)$ ``positive" coordinates, and similarly define $\pi^-: \R^{3n+6,1} \to \R^{3n+6,1}$ to be the projection onto the one ``negative" coordinate.  
Let $H_2^+ = \pi^+ \circ H_2$ and $H_2^- = \pi^- \circ H_2$.  
Let $G_2^+$, $G_2^-$, and $G_2$ denote the quadratic forms induced by $H_2^+$, $H_2^-$, and $H_2$, respectively.  
Also, by an abuse of notation, we will simply use $g_2$ to denote the quadratic form induced by $g_2$.  
It is a straightforward calculation (see Section 3 or \cite{Minemyer4}) that the induced quadratic form of $H_2$ splits as
	\begin{equation*}
	G_2 = G_2^+ + G_2^-
	\end{equation*}  
where $G_2^+$ is positive semi-definite and $G_2^-$ is negative semi-definite.  
The fact that $H_2$ is 1-Lipschitz means that
	\begin{equation*}
	g_2 \geq G_2 = G_2^+ + G_2^-
	\end{equation*}
and so
	\begin{equation*}
	G_2^+ \leq g_2 - G_2^- .
	\end{equation*}
Now we apply Theorem \ref{Minemyer} to the map $H_2^+$ to obtain a pl embedding $h_2^+: \N_2 \to \mathbb{E}^{3n+6}$ whose induced quadratic form, denoted $Q_2^+$, will satisfy 
	\begin{equation*}
	Q_2^+ = g_2 - G_2^- \qquad \Longrightarrow \qquad Q_2^+ + G_2^- = g_2
	\end{equation*}
over all simplices of some subdivision of $\T_2'$ on which $h_2^+$ is simplicial.

Thus, the pl map $h_2 : \N_2 \to \R^{3n+6,1}$ defined as the concatenation of $h_2^+$ and $H_2^-$ will be an isometric embedding.  
When applying Theorem \ref{Minemyer} we require that $h_2^+$ is a $\frac{\rho_2}{2}$-approximation of $H_2^+$, so that $h_2$ is a $\rho_2$-approximation of $h_1 \circ \v_{2,1}$.

One quick note on the construction of $g_3$ (and the construction of the subsequent metrics $g_i$).  
When defining $g_2$ we enlarged the constant $M_2''$ to obtain $M_2$.  
In doing this, we may have caused $\v_{3,2}$ to no longer satisfy the geometric properties from the previous Subsection.  
But we can fix this simply by scaling $M_3''$, and then the above procedure goes through directly.  

\begin{remark}\label{negative direction does not change}
Let $f_i = h_i \circ \psi_i$ and let $f_i^- = \pi^- \circ h_i \circ \psi_i$.  
Let $N_{i+1} = st(\psi_{i+1}(\G_{i+1} \setminus \G_i))$ denote the open star of the image under $\psi_{i+1}$ of the geodesic segments that are in $\G_{i+1}$ but are not in $\G_i$.
It is important to note that the construction of $h_{i+1}$ from $h_i$ only changes the negative coordinate of points in $N_{i+1}$.  
That is, if $p \nin N_{i+1}$ then $h_{i+1}^- (p) = h_i^- (p)$.  
In particular, if $x \in \text{im}(\g)$ for some $\g \in \G_i$, then $f_i^-(x) = f_{i+1}^-(x)$.  
\end{remark}

\vskip 20pt

%beginning of Section 7
\section{Step 4:  Finishing the proof of the Main Theorem}
Via the constructions in Section 6, we have that $\X$ is the inverse limit of the system $(\N_i, \v_{j,i})$ with ``projection" maps $\psi_i: \X \to \N_i$ (the quotes are because these maps are definitely not $1$-Lipschitz).  
We have constructed isometric embeddings $h_i : \N_i \to \R^{3n+6,1}$ in such a way that $h_i$ is a $\rho_i$-approximation of $h_{i-1} \circ \v_{i,i-1}$ for each $i$.  
So we then define $f_i: \X \to \R^{3n+6,1}$ by $f_i := h_i \circ \psi_i$, and let $\ds{f = \lim_{i \to \infty} f_i}$.  

Notice that for all $x \in \X$ and for all $i$:
	\begin{align*}
	|f_{i+1}(x) - f_i(x)| &= | (h_{i+1} \circ \psi_{i+1})(x) - (h_i \circ \psi_i)(x)|  \\
	&\leq |(h_i \circ \v_{i+1,i} \circ \psi_{i+1})(x) -  (h_i \circ \psi_i)(x)|  + \rho_i  \\
	&= |(h_i \circ \psi_i)(x) - (h_i \circ \psi_i)(x)| + \rho_i  \\
	&= \rho_i
	\end{align*}
where $| \cdot |$ denotes the Euclidean norm on $\E^{3n+7}$.
Thus, for $\rho_i$ chosen sufficiently small, the maps $(f_i)$ converge uniformly to $f$.  
Hence $f$ is continuous.

In the next three Subsections we show that $f$ is injective, that $f$ preserves the energy of paths in $\G$, and that this construction can be done so that $\pi^+ \circ f$ is not locally Lipschitz (all for $(\rho_i)$ chosen sufficiently small).  
In Subsection 7.4 we will then show how to extend this proof to spaces which are proper instead of compact.
One final remark is that, throughout this Section, $| \cdot |$ will always denote the Euclidean norm on $\E^{3n+7}$ and $\langle , \rangle$ will denote the Lorentzian quadratic form on $\R^{3n+6,1}$.

\vskip 20pt

\subsection{Verifying that $f$ is injective}

This is a pretty standard trick due to Nash in \cite{Nash1}.  
Let 
	\begin{equation*}
	\Delta_k = \left\{ (x,y) \in \X \times \X \bigr| d_{\X}(x,y) \geq 2^{-k} \right\} .
	\end{equation*}
Being a closed subset of the compact space $\X \times \X$, the set $\Delta_k$ is compact.  

Recall that $\a_i$ equals the mesh of the covering $\O_i$.  
For each $k \in \mathbb{N}$, let $k'$ be the smallest positive integer so that $\a_{k'} < 2^{-k}$.  
Of course, this property will also be satisfied for all $\ell \geq k'$.  
If a point pair $(x,y) \in \Delta_k$ then the points $x$ and $y$ are necessarily contained in different members of $\O_{k'}$.  
Therefore, $\psi_{k'} (x) \neq \psi_{k'} (y)$.  

For each $\ell \geq k'$ define a function $\zeta_\ell : \Delta_k \to \R$ defined by
	\begin{equation*}
	\zeta_\ell (x,y) = |f_\ell(x) - f_\ell(y)|.
	\end{equation*}
Since $\psi_\ell$ separates points in $\Delta_k$ and $h_\ell$ is an embedding, $\zeta_\ell > 0$ over all of $\Delta_k$.  
Then since $\Delta_k$ is compact, there exists some $\mu_\ell > 0$ so that $\zeta_\ell (x,y) \geq \mu_\ell$ for all $(x,y) \in \Delta_k$.  
Therefore, at each stage $\ell \geq k'$, if we choose $\rho_\ell < \frac{\mu_\ell}{2^\ell}$ then no pair of points in $\Delta_k$ can come together in the limit $\lim_{i \to \infty} f_i$.  
Observe that at any stage $i$, there are only finitely many $k$ so that $k' < i$.  
So this only ever results in a finite set of choices for $\rho_i$.  
Eventually any two distinct points are contained in $\Delta_k$ for some $k$, completing the proof that $f$ is injective for $(\rho_i)$ sufficiently small.

\vskip 20pt 

\subsection{Verifying that $f$ preserves the energy of any path contained in $\G$}
We begin with the following Lemma.

\begin{lemma}
Let $\g \in \G$.  
Then
	\begin{equation*}
	E(\g) = \lim_{k \to \infty} f_k^* E(\g).
	\end{equation*}
\end{lemma}

\begin{proof}
Let $\g \in \G$ and let $i$ be the smallest positive integer so that $\g \in \G_i$.  
Then certainly the endpoints of $\g$ are contained in $\D_i$, but there may be other points in the image of $\g$ that are also in $\D_i'$.
Let $u_1, u_2, \hdots, u_{k-1}$ denote the members of $\D_i'$ that lie in the image of $\g$ excluding the endpoints of $\g$ (and where it is very possible that $k - 1 = 0$).  
Then it follows directly from the construction of the map $\psi_i$ and the metric $g_i$ that
	\begin{equation*}
	\ell(\psi_i \circ \g) = (1 - \omega_i) \ell(\g) - 2k \a_i + \e
	\end{equation*}
where the $\e$ term only exists due to Subsection 6.1 when we altered the sets of $\scrU_i$ in order to force midpoints of subintervals of $\g$ at any stage to remain midpoints of subintervals at future stages.
But in this construction we could make $\e$ as small as we like by choosing $K_i$ arbitrarily large (see equations \eqref{picking K_i} and \eqref{enlarging U}, and the discussion therein).  
So in what follows we will ignore this term and write
	\begin{equation*}
	\ell(\psi_i \circ \g) = (1 - \omega_i) \ell(\g) - 2k \a_i.
	\end{equation*}
	
Now, as we increment to stage $i+1$, we add exactly one new point to $\D_{i+1}$.  
In general this could add many new points to $\D_{i+1}'$, but at most $2|\D_i|$ of these new points can lie on the image of $\g$.  
So we have that
	\begin{equation*}
	(1 - \omega_{i+1}) \ell(\g) - 2(k+2|\D_i|) \a_{i+1} \leq \ell (\psi_{i+1} \circ \g) \leq (1 - \omega_{i+1}) \ell(\g) - 2k \a_{i+1}
	\end{equation*} 
and for stage $i + m$ for any $m$ we have that
	\begin{equation*}
	(1 - \omega_{i+m}) \ell(\g) - 2(k+2m(|\D_i|+m)) \a_{i+m} \leq \ell (\psi_{i+m} \circ \g) \leq (1 - \omega_{i+m}) \ell(\g) - 2k \a_{i+m}.
	\end{equation*} 
Recall that we construct the entire collection $(\D_i)$ before we construct any of the $(\a_i)$.  
So for appropriately small choices of $(\a_i)$ we have that
	\begin{align}
	&\lim_{m \to \infty} \ell (\psi_{i+m} \circ \g) = \lim_{m \to \infty} (1-\omega_{i+m}) \ell(\g) = \ell (\g)  \nonumber \\
	\Longrightarrow 	\quad	&\lim_{k \to \infty} \ell(\psi_k \circ \g) = \ell (\g).  	\label{preserving length in the nerve}
	\end{align} 

By the construction of the metric $g_k$, the path $\psi_k \circ \g$ is a shortest path in $\N_k$.  
If $v_\g$ denotes the (constant) velocity of the path $\g$, then the velocity $v_k(t)$ of $\psi_k \circ \g$ at $t \in \text{domain}(\g)$ (and for $k \geq i$) is
	\begin{equation*}
	\ds{ v_k(t) = 
	\left\{ \begin{array}{cc} 
	(1-\omega_k)v_\g 	& \text{   if } 	d_{\X}(\g(t),x) > (\a_k - \xi_k) \text{ for all }x \in \D_k' \\ %\g(t) \nin \bigcup_{U \in \scrU_k}U  \\ 
	0 				&  \text{else} %\text{   if } \g(t) \in \bigcup_{U \in \scrU_k} U.
	 \end{array} \right.  }	
	\end{equation*}
This is because all of the points in the image of $\g$ contained in an $(\a_k - \xi_k)$ ball about a point in $\D_k'$ %members of $\scrU_k$
 are mapped onto a vertex of $\N_k$ corresponding to a member of $\scrU_k$, while the rest of the image of $\g$ is linearly mapped onto the edges of $\N_k$ corresponding to $\g$.  

By Remark \ref{energy is a scalar multiple of length} and the above calculation we know that 
	\[
	E(\g) = v_\g \ell (\g)
	\]
and
	\[
	E(\psi_k \circ \g) = (1 - \omega_k) v_\g \left( \ell(\g) - \ell \left( \text{im}(\g) \cap \bigcup_{U \in \scrU_k} U \right) \right).
	\] 
But since both $\omega_k$ and $\ell \left( \text{im}(\g) \cap \bigcup_{U \in \scrU_k} U \right)$ approach $0$ as $k \to \infty$, we have that
	\begin{equation}\label{convergence of energies}
	\lim_{k \to \infty} E(\psi_k \circ \g) = v_\g \ell(\g) = E(\g).
	\end{equation}
	
Since $h_k$ preserves the quadratic forms on the simplices of $N_k$ (after a sufficiently fine subdivision), we have that
	\begin{align}
	&E(h_k \circ \psi_k \circ \g) = E(\psi_k \circ \g)  \nonumber \\
	\Longrightarrow  \quad &\lim_{k \to \infty} E(h_k \circ \psi_k \circ \g) = E(\g)  \nonumber \\
	\Longrightarrow  \quad &\lim_{k \to \infty} f_k^*E(\g) = E(\g).    	\label{limit of f preserving energy}
	\end{align}
\end{proof}

Of course, equation \eqref{limit of f preserving energy} is a necessary but not sufficient condition for $f$ to preserve the energy of $\g$.  
But we can combine equation \eqref{limit of f preserving energy} with the following Lemma to show that $f$ preserves the energies of the paths in $\G$.  
%The purpose of the above argument was really equation \eqref{convergence of energies} which shows that the energies of $(\psi_i \circ \g)$ converge to the energy of $\g$ within the sequence $(\N_i)$.
%But equation \eqref{limit of f preserving energy} is still useful, since we will now complete the proof with the following Lemma.

\begin{lemma}
Let $\g \in \G$.  
Then for the sequence $(\rho_k)$ chosen sufficiently small, we have that
	\begin{equation*}
	f^*E(\g) = \lim_{k \to \infty} f_k^*E(\g).
	\end{equation*}
\end{lemma}

\begin{proof}	
First note that the energy functional $E()$ is neither upper nor lower semi-continuous in our setting since the metric in $\R^{3n+6,1}$ has both positive and negative eigenvalues.  
But the key observation is due to Remark \ref{negative direction does not change}.  
For all $k, l > i$ (where $i$ is still the minimal stage at which $\g$ appears in $\G_i$) 
	\begin{equation*}
	\pi^- \circ f_k (x) = \pi^- \circ f_l (x) 	\qquad	\text{for all }x \in \text{im}(\g).
	\end{equation*}
That is, the negative coordinate for any point in $f_k(\text{im}(\g))$ never changes after the $i^{th}$ stage.  
This implies that the energy functional restricted to $\g$ is lower semicontinous, that is, 
	\[
	E(f \circ \g) \leq \lim_{k \to \infty} E(f_k \circ \g) = E(\g).
	\]
	
Showing the reverse inequality is mainly an application of Theorem \ref{perturbed energy theorem}.
To this end, first note that $f_k(\text{im}(\g))$ is a piecewise linear path in $\R^{3n+6,1}$ (this is clear since $h_k$ is piecewise linear).  
Let $y_0, y_1, \hdots, y_m$ denote the break points of $f_k(\text{im}(\g))$ in $\R^{3n+6,1}$.  
Then
	\[
	E(f_k \circ \g) = E(h_k \circ \psi_k \circ \g) = \sum_{j=1}^m \langle y_j - y_{j-1}, y_j - y_{j-1} \rangle.
	\]
But the bilinear form $\langle , \rangle$ splits over the positive and negative direction(s) of $\R^{3n+6,1}$, giving
	\begin{align}
	E(f_k \circ \g) &= \sum_{j=1}^m \langle \pi^+(y_j - y_{j-1}), \pi^+(y_j - y_{j-1}) \rangle_{\E^{3n+6}} - \sum_{j=1}^m (\pi^-(y_j - y_{j-1}))^2  \label{splitting form}  \\
	&= E(\pi^+ \circ f_k \circ \g) + E(\pi^- \circ f_k \circ \g)  \label{splitting energy}.
	\end{align}
In equation \eqref{splitting form} the notation $\langle , \rangle_{\E^{3n+6}}$ denotes the Euclidean quadratic form on $\R^{3n+6,0} \cong \E^{3n+6}$, and in equation \eqref{splitting energy} the different energy functionals are with respect to the quadratic forms on $\R^{3n+6,0}$ and $\R^{0,1}$, respectively.
Equation \eqref{splitting energy} shows that
	\begin{equation}\label{positive energy}
	E(f_k^+ \circ \g) = E(f_k \circ \g) - E(f_k^- \circ \g).
	\end{equation}
where $f_k^+ = \pi^+ \circ f_k$ and similarly for $f_k^-$.
The importance of equation \eqref{positive energy} is as follows.  
For $k > i$, the map $f_k^- \circ \g$ is fixed and so does not depend on $k$.  
Therefore $E(f_k^- \circ \g)$ is a nonpositive constant.
Naming this constant $C_\g$, we then have that
	\begin{equation*}
	E(f_k^+ \circ \g) = E(f_k \circ \g) - C_\g
	\end{equation*}
or
	\begin{equation*}
	(f_k^+)^*E(\g) = f_k^*E(\g) - C_\g.
	\end{equation*}

The function $f_k^+: \X \to \R^{3n+6,0} \cong E^{3n+6}$ is now a map between metric spaces, and so we can apply Theorem \ref{perturbed energy theorem}.  
For all $k \geq i$ choose 
	\begin{equation*}
	\rho_{k+1} < \frac{1}{2} \min \left\{ \rho_k, \d (f_k^+, \frac{1}{k}, \left( 1 + \frac{1}{k} \right), \g)  \right\}
	\end{equation*}
where $\d$ is as in Theorem \ref{perturbed energy theorem}.
We then have that
	\begin{equation*}
	| f^+(x) - f_k^+(x) | < \sum_{i = k+1}^\infty \rho_k < \d (f_k^+, \frac{1}{k}, \left( 1 + \frac{1}{k} \right), \g).
	\end{equation*}
Therefore, by Theorem \ref{perturbed energy theorem} we see that
	\begin{align*}
	&(f_k^+)^*E(\g) < \left( 1 + \frac{1}{k} \right) (f^+)^*E(\g) + \frac{1}{k}  \\
	\Longrightarrow \qquad &\lim_{k \to \infty} (f_k^+)^*E(\g) \leq (f^+)^*E(\g)  \\
	\Longrightarrow \qquad &\lim_{k \to \infty} (f_k^+)^*E(\g) + C_\g \leq (f^+)^*E(\g) + C_\g  \\
	\Longrightarrow \qquad &\lim_{k \to \infty} f_k^*E(\g) \leq f^*E(\g).
	\end{align*}
Note that at each stage $i$ we introduce several new geodesics to $\G_i$.  
So, instead of applying Theorem \ref{perturbed energy theorem}, we really apply Corollarly \ref{perturbed energy corollary}.  
\end{proof}

This completes the proof of the compact version of the Main Theorem.

\vskip 20pt

\subsection{Verifying that $f$ can be constructed so that $f^+$ is not locally Lipschitz}
Let $x \in \X$ and choose a path $\a:[0,1] \to \X$ with $\a(0) = x$ and whose image does not correspond to a geodesic in $\G$.  
Moreover, assume that the image of $\a$ intersects the image of any $\g \in \G$ in at most a finite number of points (which can always be done assuming the local covering dimension about $x$ is greater than one).
For sufficiently small choices of $\d_k$ at each stage $k$, the only portions of the image of $\a$ that will be contained in $\scrU_k \cup \scrV_k$ are portions near points of intersection with either members of $\D_k^\prime$ or geodesic segments in $\G_k$.  
If this is the situation for all $k$ then it is clear that 
	\begin{equation*}
	\lim_{k \to \infty} \psi_k^*E(\a) = \infty
	\end{equation*}
for sufficiently large choices of $M_k$, since the image of $\psi_k \circ \a$ will lie in simplices of $\N_k$ whose vertices all correspond to members of $\scrW_k$.  

If we put further restrictions on $\rho_k$ by adding $\a$ to the list of paths for which we apply Corollary \ref{perturbed energy corollary}, then we will have that
	\begin{equation*}
	\lim_{k \to \infty} f^*E(\a) = \infty.
	\end{equation*}
Even more, we will have that $f^*E(\a|_{[a,b]}) = \infty$ for any $0 < a < b < 1$.  
So we see that $f$ will not be locally Lipschitz at any point contained in the image of $\a$ and, in particular, at $x$.  

So to ensure that $f^+$ is not locally Lipschitz, at each stage $i$ we choose a point $x_i$ and a path $\a_i$ as above.  
We choose the points $(x_i)$ in such a way that the collection of all $x_i$ will be dense in $\X$.  
Then at each stage $i$ we add the requirement desired above for our choice of $\d_i$, and we add $\a_i$ to the list of geodesics for which we apply Corollary \ref{perturbed energy corollary}.

\vskip 20pt

\subsection{Adjusting the proof to deal with proper rather than compact spaces}
At each stage $i$ nearly every set in our construction is finite.  
So it should not really matter if $\X$ is compact or not.  
But we need each open cover $\O_i$ to have a positive Lebesgue number $\d_i$, and this is the main place where we use compactness.

If $\X$ is proper instead of compact then we proceed as follows.  
Fix $x \in \X$, and for each stage $i$ let $r_i$ be a positive number large enough so that im($\G_i) \subset B(x,\frac{1}{2} r_i)$.  
We also choose $r_i > r_{i-1}$.   %so that the increasing collection $\{ B(x,r_i) \}_{i=1}^\infty$ exhausts the space $\X$.  
At stage $i$ we perform the construction with $\X = B(x,r_i)$.  
We will then have a positive Lebesgue number $\d_i$, but the map $\psi_i$ will only be defined on the region $B(x,r_i)$.  
When we extend to stage $(i+1)$ the domain of $\psi_i$ will increase, but we will have to restrict the domain of $\v_{i+1,i}$ to the image of $B(x,r_i)$ under $\psi_{i+1}$.
This restricts the domain of $f_{i+1}$ to $B(x,r_i)$.  
These restrictions do not present an issue though, since at stage $(i+2)$ we will be able to define $\v_{i+2,i+1}$ on the image of $B(x,r_{i+1})$ under $\psi_{i+2}$.  
We can therefore extend the domain of $f_{i+2}$ to $B(x,r_{i+1})$.  
So in the limit the domains of $(f_i)$ will exhaust $\X$, and the limiting map $f$ will again be continuous, injective, and will preserve the energy of the paths in $\G$.

\vskip 20pt
%beginning of section 8
\section{Preliminaries and the proof of Proposition \ref{Proposition}}\label{section: preliminaries}

\subsection{The parallelogram law, Rademacher's Theorem, and the proof of Proposition \ref{Proposition}}

Given a norm $\| \cdot \|$ on a vector space $V$, a natural question is whether or not this norm is induced by some inner product $\langle , \rangle$.  
ie, does there exist an inner product $\langle , \rangle$ on $V$ such that $\| x \|^2 =\langle x, x \rangle$ for all $x \in V$?   
It is well known that the norm $\| \cdot \|$ is induced by an inner product if and only if 
	\begin{equation}\label{parallelogram law}
	2 \|x\|^2 + 2\|y\|^2 = \|x+y\|^2 + \|x-y\|^2
	\end{equation}
for all $x, y \in V$.  
Equation \eqref{parallelogram law} is called the {\it parallelogram law}.

Let $(\X, \dx)$ and $(\Y, \dy)$ be metric spaces and $f:\X \rightarrow \Y$ a continuous map.  
The map $f$ is \emph{1-Lipschitz} (or {\it short}) if for any two points $x, x' \in \X$ we have that $d_{\Y}(f(x), f(x')) \leq d_{\X}(x, x')$.  
% and $f$ is \emph{strictly short} if $d_{\Y}(f(x), f(x')) < d_{\X}(x, x')$ for any $x, x' \in \X$ with $x \neq x'$.  
Conversely, we say that $f$ is {\it expanding} (or 1-expanding) if $\dx(x,x') \leq \dy(f(x),f(x'))$ for all $x, x' \in \X$.  

The statement of Rademacher's Theorem, as can be found in \cite{Federer}, is as follows

\vskip 10pt

\begin{theorem}[Rademacher's Theorem]\label{Rademacher}
Suppose that $f:U \to \R^n$ is Lipschitz where $U \subseteq \R^m$ open.  
Then $f$ is differentiable at almost all points of $U$ (with respect to the Lebesgue measure), meaning that for almost all points $u \in U$, there exists a linear map $L_u: \R^m \to \R^n$ such that
	\begin{equation*}
	\lim_{x \to u} \frac{| f(x) - f(u) - L_u(x-u)|}{|x-u|} = 0.
	\end{equation*}
\end{theorem}

\vskip 10pt

We now prove Proposition \ref{Proposition} from the Introduction.

%Let us first prove Proposition \ref{Proposition} from back in the Introduction.

\begin{proof}[Proof of Proposition \ref{Proposition}]
It is clear that we must have that $p > 0$, and the necessity of $q > 0$ is by Le Donne's result.  
To show that both $\pi_p \circ f$ and $\pi_q \circ f$ cannot be locally Lipschitz, suppose the contrary.  
Clearly, if one of these maps is locally Lipschitz but the other is not, then the map $f$ is not an isometry.  
So they must both be locally Lipschitz.
But then $f$ is locally Lipschitz, and so by the same argument as in \cite{Le Donne} the Finsler norm on $\X$ must be induced by the map $f$.  

Let $f^+ := \pi_p \circ f$ and $f^- := \pi_q \circ f$.  
Let $Q^+$ and $Q^-$ denote the quadratic forms induced by $f^+$ and $f^-$, respectively.
If $| \cdot |$ denotes the Finsler norm on $\X$, then at any point $p \in \X$ we must have that
	\[
	| \cdot |_p = Q^+_p + Q^-_p 	\qquad	\Longrightarrow 	\qquad 	| \cdot |_p - Q^-_p = Q^+_p.
	\]
Since $Q^+_p$ is induced by an inner product, it satisfies the parallelogram law \eqref{parallelogram law}.  
Thus, $| \cdot |_p - Q^-_p$ must satisfy the parallelogram law.  
But $-Q^-_p$ is also induced by an inner product and therefore satisfies the parallelogram law as well.
Then a simple calculation shows that this implies that $| \cdot |_p$ must also satisfy the parallelogram law.  
Therefore $| \cdot |_p$ is induced by an inner product and hence $(\X, | \cdot |)$ is Riemannian, a contradiction.
\end{proof}

\vskip 15pt

\subsection{Covering Dimension and the Nerve of an Open Cover}
		
The definitions in this Subsection are from \cite{Nagami} and/or \cite{Nagata}.
			
An open covering $\mathscr{U}$ of a space $\X$ is said to be of \emph{order $n$} if for all $x \in X$ there exists a neighborhood $U$ of $x$ such that $U$ has non-trivial intersection with at most $n$ members of $\mathscr{U}$.  
The \emph{covering dimension} or \emph{topological dimension} of a metric space $\X$, denoted by dim$(\X)$, is at most $n$ if every open covering of $\X$ can be refined by an open covering whose order is at most $n + 1$.  
If dim$(\X) \leq n$ and dim$(\X) \nleq n - 1$, then we say that dim$(\X) = n$.  
Of course, if dim$(\X) \nleq n$ for any $n$, then we say that dim$(\X) = \infty$.
			
As always, in the above definition we call an open cover $\mathscr{V}$ a \emph{refinement} or another open cover $\mathscr{U}$ if for every $V \in \mathscr{V}$ there exists $U \in \mathscr{U}$ such that $V \subset U$.  
The \emph{mesh} of an open cover of a metric space $\X$ is the supremum of the diameters of the open sets contained in that cover.
	
Let $\{ U_\a \}_{\a \in I}$ be an indexed open covering of a topological space $\X$.  
Assume that $U_\a \neq \emptyset$ for all $\a$.  
The \emph{nerve} $\calN$ associated with $\{ U_\a \}$ is the abstract simplicial complex whose simplices are defined as follows:
	\begin{enumerate}
		\item $\emptyset \in \calN$
		%\item every one point subset of $I$ is contained in $\calN$ (this is also a consequence of (3) below)
		\item a subset $J \subset I$ is contained in $\calN$ if and only if $\ds{\bigcap_{\a \in J}  U_\a  \neq \emptyset}$
	\end{enumerate}
It is not hard to see that the collection $\calN$ defined above satisfies the conditions to be an abstract simplicial compex.  An important observation is that if an open covering has order $n + 1$, then its corresponding nerve has dimension $n$.

\vskip 20pt

\subsection{Taut Chains and a Theorem due to Bridson}
	
The majority of the material in this Subsection comes from \cite{Bridson} and/or \cite{BH}.
	
\begin{comment}
A \emph{Euclidean polyhedron} $\P = (\X, \T)$ is a metric space $\X$ which admits a locally finite (simplicial) triangulation $\T$ so that, 
for each $ k \in \mathbb{N}$, every $k$-dimensional simplix of $\T$ (with its induced length metric) is affinely isometric to a simplex in Euclidean space $\E^k$.  
%We denote the metric on $\X$ by $d_{\X}( \, , \, )$ and, if there is no risk of confusion, simply by $d( \, , \, )$.
\end{comment}
	
Let $(\X, \T)$ be a Euclidean polyhedron and let $x, y \in \X$.  
An \emph{m-chain from x to y} is an $(m+1)$-tuple $C = (x_0, x_1, ..., x_m)$ of points in $\X$ such that $x = x_0$, $y = x_m$, and for each index $i > 0$ there exists a simplex $S(i) \in \T$ such that $x_{i-1}, x_i \in S(i)$.  
Every $m$-chain determines a pl path in $\X$ given by concatenation of the line segments $x_{i-1}x_i$.  
Given an $m$-chain $C= (x_0, x_1, ..., x_m)$ we can compute the length of the associated pl path as 
\[
\ell(C) := \sum_{i = 1}^{m} d_{S(i)}(x_{i-1}, x_i)
\]
where $d_{S(i)}(,)$ denotes the induced length metric on $S(i)$.
%Note that $d_{S(i)(,)}$ may not agree with $d(,)$ if there exist shorter paths outside of $S(i)$ between two points.

\begin{comment}
The following Lemma is due to Moussong in \cite{Moussong}.
	
\begin{lemma}\label{Moussong}
Let $(\X, \T)$ be a compact Euclidean polyhedron and let $x, y \in \X$.  
If $x$ and $y$ can be joined by an $m$-chain in $\X$ (for some fixed integer $m$) then there exists a shortest $m$-chain from $x$ to $y$ in $\X$.
\end{lemma}
\end{comment}

An $m$-chain $C = (x_0, x_1, ..., x_m)$ in a Eucildean polyhedron $(\X, \T)$ is \emph{taut} if it satisfies the following two conditions for all $1 \leq i \leq m-1$:
	\begin{enumerate}
	\item No simplex contains all three of the points $\{ x_{i-1}, x_i, x_{i+1} \}$.
	\item If $x_{i-1}, x_i \in S(i)$ and $x_i, x_{i+1} \in S(i+1)$ then the concatenation of the line segments $x_{i-1}x_i$ and $x_i x_{i+1}$ is a geodesic segment in $S(i) \cup S(i+1)$.
	\end{enumerate}

Notice that only the first and last points in a taut chain can be contained in the interior of a maximal simplex of $\T$.  Condition (1) is really a technical condition which allows us to disregard chains which are not, in some sense, minimal.  Intuitively, it is easy to see why we would consider condition (2) above if we are trying to identify which chains correspond to geodesics in our space $\X$.  
%But be warned that it is \emph{not} true that all taut chains are local geodesics.  
The next straightforward Lemma can be found in \cite{Bridson} or \cite{BH}.
		
\begin{lemma}
If, for some fixed integer $m$, $C$ is an $m$-chain from $x$ to $y$ in $(\X, \T)$ of minimal length, then there exists a taut $n$-chain $C'$, with $n \leq m$, such that the path determined by $C'$ is exactly the same as the path determined by $C$.
\end{lemma}
		
\begin{corollary}
$d_{\X}(x, y) = \text{inf} \, \{ l(C) \, | \, C $ is a taut chain from $x$ to $y$ $\} $.
\end{corollary}
		
In \cite{Bridson} and \cite{BH} the metric simplicial complexes which are considered are \emph{not} locally finite, which is the case that we will be interested in for this paper.  
But they do satisfy a different sort of ``local compactness" property.  For $\T$ a metric simplicial complex, let \emph{Shapes($\T$)} denote the isometry types of the simplices of $\T$.  Then the condition required in \cite{Bridson} and \cite{BH} is that Shapes($\T$) be finite.
		
We are now prepared to state the Theorem due to Bridson that we will need in Section \ref{Proof of Key Lemma}.

\vskip 10pt
		
\begin{theorem}[Bridson]\label{Bridson}
Let $(\X, \T)$ be a Euclidean polyhedron with Shapes($\T$) finite.  Then for every $\l > 0$ there exists an integer $N > 0$, which depends only on Shapes($\T$), such that for every taut $m$-string in $\T$ of length at most $\l$ we have that $m \leq N$.
\end{theorem}

\vskip 20pt

\subsection*{Acknowledgements} 

The idea of studying the isometric embedding problem for geodesic metric spaces was originally mentioned to the author by Pedro Ontaneda as an idea for the author's Ph. D. thesis.  
The author would like to thank Ontaneda, as well as Jean-Fran\c{c}ois Lafont and Anton Petrunin, for various helpful remarks and suggestions throughout the duration of this research.

%{\color{red} To Do List

%\begin{enumerate}

%\item  Define ``1-expanding" in Section 6.

%\item  Discuss Rademacher's Theorem in Section 6.

%\item  Add in Section 1 that the Main Theorem would give an isometric embedding if the target were Euclidean space.

%\item  Change definition of $\D_{i+1}$ to be obtained from $\D_i$ by adding exactly one point.

%\item  Specify the domains of all geodesics in $\G$ in Section 3 so that all velocities are 1 at all times $t$.

%\item  Define ``induced quadratic forms" for pl maps (in Section 6).

%\item  Add the ``Krat/Akopyan" Theorem into Section 6.

%\end{enumerate}}

%end of Section 6


\begin{thebibliography}{XXXXX}

\bibitem[Ako07]{Akopyan} A.V. Akopyan, \textit{PL-analogue of Nash-Kuiper theorem}, preliminary version (in Russian): \\ 
http://www.moebiuscontest.ru/files/2007/akopyan.pdf  \\ www.moebiuscontest.ru

%\bibitem[AT08]{AT} A.V. Akopyan and A.V. Tarasov, \textit{A constructive proof of Kirszbraun's theorem},
%Math. Notes, {\bf 84}(2008), no. 5-6, 725-728.

\bibitem[BBI01]{BBI} D. Burago, Y. Burago, S. Ivanov, \textit{A course in metric geometry},
Graduate Studies in Mathematics {\bf 33}, American Mathematical Society, Providence, RI, 2001.

%\bibitem[Bha07]{Bhatia} R. Bhatia, \textit{Positive definite matrices}, 
%Princeton Series in Applied Mathematics, 2007, 2-3.

%\bibitem[Bre81]{Brehm} U. Brehm, \textit{Extensions of distance reducing mappings to piecewise congruent mappings on $\R^m$},
%J. Geom., {\bf 16}(1981), no. 2, 187-193.

\bibitem[Bri91]{Bridson} M. Bridson, \textit{Geodesics and curvature in metric simplicial complexes}, 
Ph. D. Thesis, Cornell University (1991).

\bibitem[BH91]{BH} M. Bridson and A. Haefliger, \textit{Metric spaces of non-positive curvature},
Springer-Verlag Berlin Heidelberg, 1999.

\bibitem[BI94]{BI} D. Burago and S. Ivanov, \textit{Isometric embeddings of Finsler manifolds},
St. Petersburg Math. J., {\bf 5}(1994), no. 1, 159-169.

\bibitem[BZ96]{BZ} Y.D. Burago and V.A. Zalgaller, \textit{Isometric piecewise linear immersions of two-dimensional manifolds with polyhedral metrics into $\mathbb{R}^3$},
St. Petersburg Math. J., {\bf 7}(1996), no. 3, 369-385.

%\bibitem[EM02]{EM} Y. Eliashberg and N. Mishachev, \textit{Introduction to the $h$-principle},
%American Mathematical Society, Providence, RI, 2002.

\bibitem[Fed69]{Federer} H. Federer, \textit{Geometric measure theory},
Springer-Verlag New York Inc., 1969.

\bibitem[GZ15]{GZ} P. Galashin and V. Zolotov, \textit{Extensions of isometric embeddings of pseudo-Euclidean metric polyhedra},
preprint, arXiv: 1501.05037 (2015).

\bibitem[Gre70]{Greene} R.E. Greene, \textit{Isometric embeddings of Riemannian And pseudo-Riemannian manifolds},
Mem. Amer. Math. Soc., Providence, RI, 1970.

\bibitem[Gro70]{GR} M. Gromov and V. Rokhlin, \textit{Embeddings And immersions In Riemannian geometry},
Russ. Math. Surv., {\bf 25}(1970), no. 5, 1-57.

\bibitem[Gro80]{Gromov PDR} M. Gromov, \textit{Partial differential relations},
Springer-Verlag, 1980, 213.

%\bibitem[Gro99]{Gromov green} M. Gromov, \textit{Metric structures for Riemannian and non-Riemannian spaces}, 
%Birkhauser, 1999.

\bibitem[G\"{u}n89]{Gunther} M. G\"{u}nther, \textit{Isometric embeddings of Riemannian manifolds},
Proceedings of the International Congress of Mathematicians, {\bf I, II}(1990), 1137-1143, Math. Soc. Japan, Tokyo (1991).

%\bibitem[HY88]{HY} J. Hocking and G. Young, \textit{Topology}, Dover
%(1988), 214--215.

\bibitem[Isb59]{Isbell} J.R. Isbell, \textit{Embeddings of inverse limits},
Ann. of Math. (2), {\bf 70}(1959), 73-84.

\bibitem[Kra04]{Krat} S. Krat, \textit{Approximation problems in length geometry},
Ph. D. Thesis, The Pennsylvania State University (2004).

\bibitem[Kui55]{Kuiper} N. Kuiper, \textit{On $C^1$-isometric imbeddings},
Indag. Math., {\bf 17}(1955), 545-556.

\bibitem[LeD12]{Le Donne} E. Le Donne, \textit{Lipschitz and path isometric embeddings of metric spaces},
Geom Dedicata, {\bf 166}(2013), 47-66.

\bibitem[M16]{Minemyer4} B. Minemyer, \textit{Approximating continuous maps by isometries},
preprint, arXiv:  1508.00435.

\bibitem[Min16]{Minemyer Geoghegan} B. Minemyer, \textit{Intrinsic geometry of a Euclidean simplex}, 
accepted to the conference proceedings `Topological Methods in Group Theory:  A Conference in Honor of Ross Geoghegan's 70th Birthday'.

\bibitem[Min15]{Minemyer1} B. Minemyer, \textit{Isometric embeddings of polyhedra into Euclidean space},
J. Topol. Anal., {\bf 7}(2015), no. 4, 677-692.

\bibitem[Mi16]{Minemyer2} B. Minemyer, \textit{Simplicial isometric embeddings of polyhedra},
preprint, arXiv: 1211.0584.

%\bibitem[Mi15]{Minemyer3} B. Minemyer, \textit{Isometric embeddings of pro-Euclidean spaces},
%Anal. Geom. Metr. Spaces, {\bf 3}(2015), 317-324.

%\bibitem[Mou88]{Moussong} G. Moussong, \textit{Hyperbolic coxeter groups}, 
%Ph. D. Thesis, The Ohio State University (1988).

\bibitem[Nag70]{Nagami} K. Nagami, \textit{Dimension theory}, 
Academic Press, New York and London, 1970.

\bibitem[Nag83]{Nagata} J. Nagata, \textit{Modern dimension theory}, 
Heldermann Verlag Berlin, 1983.

%\bibitem[Min13]{Minemyer2} Barry Minemyer, \textit{Isometric Embeddings of Polyhedra},
%Ph. D. Thesis, The State University of New York at Binghamton (2013)

\bibitem[Nas54]{Nash1} J. Nash, \textit{$C^1$ isometric imbeddings},
Ann. of Math. (2), {\bf 60}(1954), 383-396.

\bibitem[Nas56]{Nash2} J. Nash, \textit{The imbedding problem for Riemannian manifolds},
Ann. of Math. (2), {\bf 63}(1956), 20-63.

\bibitem[OS94]{OS} Y. Otsu and T. Shioya, \textit{The Riemannian structure of Alexandrov spaces},
J. Differential Geom., {\bf 39}(1994), 629-658.

\bibitem[Pet11]{Petrunin} A. Petrunin, \textit{On intrinsic isometries to Euclidean space},
St. Petersburg Math. J., {\bf 22}(2011), no. 5, 803-812

%\bibitem[PY15]{PY} A. Petrunin and A. Yashinski, \textit{Piecewise distance preserving maps},
%to appear in {\it St. Petersburg Math. J.}

%\bibitem[Whi36]{Whitney1} Hassler Whitney, \textit{Differentiable manifolds},
%The Annals of Mathematics (2), Vol. 37 (1936), pp. 645-680

%\bibitem[Whi44]{Whitney2} Hassler Whitney, \textit{The self-intersections of a smooth $n$-manifold in $2n$-space},
%The Annals of Mathematics (2), Vol. 45 (1944), pp. 220-246

\bibitem[Riv03]{Rivin} I. Rivin, \textit{Some observations on the simplex},
preprint, arXiv: 0308239.

\bibitem[Sal12]{Salviano} C. Salviano Veiga, \textit{The index space of a geodesic},
Ph. D. Thesis, Binghamton University (2012).

\bibitem[Zal58]{Zalgaller} V.A. Zalgaller, \textit{Isometric imbedding of polyhedra},
Dokl. Akad. Nauk SSSR (in Russian), {\bf 123}(1958), no. 4, 599-601.

\end{thebibliography}
\end{document}